\documentclass[twoside]{article}
\usepackage{amsmath,amsthm,amssymb}
\usepackage{newtxtext,newtxmath}
\usepackage[text={12.5cm,19.1cm},centering,paperwidth=16.96cm,paperheight=23.96cm]{geometry}
\usepackage[fontsize=10.4pt]{scrextend}
\usepackage[hyperfootnotes=false,colorlinks=true,allcolors=blue]{hyperref}
\usepackage{array}
\usepackage{multicol}
\usepackage{multirow}
\usepackage{tabularx}

\def\titlerunning#1{\gdef\titrun{#1}}

\makeatletter
\def\author#1{\gdef\autrun{\def\and{\unskip, }#1}\gdef\@author{#1}}

\makeatother

\def\email#1{E-mail: \href{mailto:#1}{#1}}
\def\subjclass#1{\par\bigskip\noindent\textsl{Mathematics Subject Classification 2020. }#1}
\def\keywords#1{\par\smallskip\noindent\textsl{Keywords and Phrases. }#1}

\newenvironment{dedication}{\itshape\center}{\par\medskip}
\newenvironment{acknowledgments}{\bigskip\small\noindent\textit{Acknowledgments.}}{\par}

\newtheorem{thm}{Theorem}[section]

\newtheorem{lem}[thm]{Lemma}

\theoremstyle{definition}

\numberwithin{equation}{section}

\begin{document}

\titlerunning{$g$-invariant on unary Hermitian lattices over imaginary quadratic fields}

\title{\textbf{$g$-invariant on unary Hermitian lattices \\ over imaginary quadratic fields \\ with class number $2$ or $3$}}

\author{Jingbo Liu}

\date{}

\maketitle

\begin{dedication}
Department of Mathematical, Physical, and Engineering Sciences, \\ Texas A\&M University-San Antonio, San Antonio, Texas 78224, USA \\ \email{jliu@tamusa.edu}
\end{dedication}


\subjclass{Primary 11E39. Secondary 11E41, 11Y40.}

\keywords{Waring's problem, unary Hermitian lattices, sums of norms.}

\begin{abstract}
In this paper, we study the unary Hermitian lattices over imaginary quadratic fields.
Let $E=\mathbb{Q}\big(\sqrt{-d}\big)$ be an imaginary quadratic field for a square-free positive integer $d$, and let $\mathcal{O}$ be its ring of integers.
For each positive integer $m$, let $I_m$ be the free Hermitian lattice over $\mathcal{O}$ with an orthonormal basis, let $\mathfrak{S}_d(1)$ be the set consisting of all positive definite integral unary Hermitian lattices over $\mathcal{O}$ that can be represented by some $I_m$, and let $g_d(1)$ be the smallest positive integer such that all Hermitian lattices in $\mathfrak{S}_d(1)$ can be uniformly represented by $I_{g_d(1)}$.
The main results of this paper determine the explicit form of $\mathfrak{S}_d(1)$ and the exact value of $g_d(1)$ for all imaginary quadratic fields $E$ with class number $2$ or $3$, generalizing naturally the {\sf Lagrange}'s {\sf four-square theorem}.
\end{abstract}


\section{Introduction}\label{Sec:Int}
A positive integer $a$ is said to be represented by a quadratic form $f$, (often) written as $a\to f$, provided the Diophantine equation $f(\vec{x})=a$ has an integral solution.
A typical question is to determine those positive integers which can be represented by the form $I_m:=x_1^2+\cdots+x_m^2$, where $m$ is a positive integer.
In 1770, Lagrange proved the famous {\sf four-square theorem}: the quadratic form $x_1^2+x_2^2+x_3^2+x_4^2(=I_4)$ represents all positive integers; such quadratic forms are said to be {\sl universal} in the literature.
Noticing $7$ can not be a sum of any three integer squares, it is clear the sum of four squares is optimal for the universality over $\mathbb{Z}$.
This result has been generalized in several directions, and one important direction is to consider the representations of sums of squares in totally real number fields: a positive definite integral quadratic form over a totally real number field $F$ is said to be universal if it represents all totally positive integers in this field.
In 1928, G\"{o}tzky \cite{fG} proved that $x_1^2+x_2^2+x_3^2+x_4^2$ is universal over $\mathbb{Q}\big(\sqrt{5}\big)$.
Later, Maass \cite{hM} further proved the universality of $x_1^2+x_2^2+x_3^2$ over $\mathbb{Q}\big(\sqrt{5}\big)$.
However, this will not occur in any other field $F$, as Siegel \cite{clS} proved that $F$ admits a sum of squares which is universal if and only if either $F=\mathbb{Q}$ or $F=\mathbb{Q}\big(\sqrt{5}\big)$.

In this paper, we will consider the Hermitian analog of this representation problem over imaginary quadratic fields.
Let $E=\mathbb{Q}\big(\sqrt{-d}\big)$ be an imaginary quadratic field with $d$ a square-free positive integer, and denote by $\mathcal{O}$ its ring of integers.
Knowing that $\mathcal{O}$ is not necessarily a principle ideal domain in general, a positive integer $a$ corresponds to an unary Hermitian lattice ($\mathcal{O}v$ with the Hermitian map $h$ satisfying $h(v)=a$), but not vice versa; so, instead of considering only positive integers, the more general concept of unary Hermitian lattices will be utilized.
For each positive integer $m$, denote by $I_m$ the free Hermitian lattice of rank $m$ having an orthonormal basis; through the standard correspondence between Hermitian forms and free Hermitian lattices, $I_m$ corresponds to the Hermitian form $x_1\overline{x_1}+\cdots+x_m\overline{x_m}$.

Over the imaginary quadratic fields $E=\mathbb{Q}\big(\sqrt{-d}\big)$ with class number $1$, {\sl i.e.}, those fields by $d=1,2,3,7,11,19,43,67,163$, each positive definite integral unary Hermitian lattice is of the form $L=\mathcal{O}v$, where $h(v)$ is a positive integer.
$L$ is represented by $I_m$ if and only if $h(v)$ can be written as a sum of $m$ norms of integers in $\mathcal{O}$.
In addition, from Kim, Kim and Park \cite[Table 3]{KKP}, one easily observes the universality of $I_2$ when $d=1,2,3,7,11$ and that of $I_3$ when $d=19$; when $d=43,67,163$, the universality of $I_4$ is a direct consequence of the {\sf Lagrange}'s {\sf four-square theorem} as the norm and the square of a rational integer are the same.

For an imaginary quadratic field $E$ with class number $\mathtt{c}>1$, assume the integral ideals $\mathfrak{U}_1=\mathcal{O},\mathfrak{U}_2,\ldots,\mathfrak{U}_{\mathtt{c}}$ are the representatives of the ideal classes in $E$.
It is well known that $\mathfrak{U}_j\overline{\mathfrak{U}_j}=k_j\mathcal{O}$ for an integer $k_j\in\mathbb{Z}_{>0}$, and $\mathfrak{U}_j=k_j\mathcal{O}+\alpha_j\mathcal{O}$ (\cite [22:5]{otO}) for an integer $\alpha_j\in\mathcal{O}$, $1\leq j\leq\mathtt{c}$.
Let $L_j=\mathfrak{U}_jv_j$ be a positive definite unimodular unary Hermitian lattice with $h(v_j)=1/k_j$.
For an arbitrary positive definite integral unary Hermitian lattice $\mathfrak{U}_jw$, noting $\mathfrak{U}_j\overline{\mathfrak{U}_j}h(w)\subseteq\mathcal{O}$, one has $k_jh(w)\in\mathbb{Z}$; so, $\mathfrak{U}_jw$ is in the isometry class of the lattice $L_j^{r_j}:=\mathfrak{U}_jv_j^{r_j}$ with $v_j^{r_j}$ a vector satisfying $h\big(v_j^{r_j}\big)=r_jh(v_j)$ for a positive integer $r_j$.
An unary Hermitian lattice $L$ is said to be represented by some $I_m$, provided there is an injective linear map from $L$ to $I_m$ preserving Hermitian maps.
Hence, it is natural to ask if there is a positive integer $m$ such that $I_m$ can represent all positive definite integral unary Hermitian lattices over $\mathcal{O}$.
The answer, unfortunately, is no: we will see later that a positive definite unimodular Hermitian lattice $\mathfrak{U}v$ can be represented by some $I_m$ if and only if $\mathfrak{U}$ is a principle ideal.
Thus, we shall restrict our attention to a smaller set as done in \cite{BCIL} and \cite{jL}.
Denote by $\mathfrak{S}_d(1)$ the set consisting of all positive definite integral unary Hermitian lattices over $\mathcal{O}$ which can be represented by some $I_m$, and by $g_d(1)$ the smallest positive integer such that all Hermitian lattices in $\mathfrak{S}_d(1)$ are uniformly represented by $I_{g_d(1)}$.
In \cite{jL}, I found an upper bound of $g_d(1)$ for all imaginary quadratic fields $E$; in this paper, I determine the explicit form of the set $\mathfrak{S}_d(1)$ and the exact value of its associated integer $g_d(1)$ for all imaginary quadratic fields $E$ with class number $2$ or $3$.

In Section 2, we introduce the geometric language of Hermitian spaces and Herm-itian lattices.
In Section 3, we determine the set $\mathfrak{S}_d(1)$ and the value of $g_d(1)$ for every imaginary quadratic field $E$ with class number $2$, and in Section 4, we do so for every imaginary quadratic field $E$ with class number $3$.
Finally, we list the representatives of ideal classes of $E$ with class number $2$ or $3$ in the Appendices.

\section{Preliminaries}\label{Sec:Pre}
Throughout this paper, the notations and terminologies for lattices from the classical monograph of O'Meara \cite{otO} will be adopted.
For background information and terminologies specific to the Hermitian setting, one may consult the work of Shimura \cite{gS} and Gerstein \cite{ljG}.

Let $E=\mathbb{Q}\big(\sqrt{-d}\big)$ be an imaginary quadratic field for a square-free positive integer $d$, and let $\mathcal{O}$ be its ring of integers.
Then, one has $\mathcal{O}=\mathbb{Z}+\mathbb{Z}\omega$ with $\omega=\omega_d:=\sqrt{-d}$ if $d\equiv1,2~(\mathrm{mod}~4)$ and $\omega=\omega_d:=\big(1+\sqrt{-d}\big)/2$ if $d\equiv3~(\mathrm{mod}~4)$.
For each $\alpha\in E$, denote by $\overline\alpha$ the complex conjugate of $\alpha$ and define the norm of $\alpha$ to be $N(\alpha):=\alpha\overline\alpha$.
A Hermitian space $(V,h)$ is a vector space $V$ over $E$, which admits a Hermitian map $h:V\times V\to E$ satisfying, for all $\alpha,\beta\in E$ and $v,v_1,v_2,w\in V$,
\begin{description}
\item[~~(1)] $h(v,w)=\overline{h(w,v)}$;
\item[~~(2)] $h(\alpha v_1+\beta v_2,w)=\alpha h(v_1,w)+\beta h(v_2,w)$.
\end{description}
Write $h(v):=h(v,v)$ for brevity.
From condition {\bf(1)}, one has $h(v)=\overline{h(v)}$, and thus, $h(v)\in\mathbb{Q}$ for every $v\in V$.

A Hermitian lattice $L$ is defined as a finitely generated $\mathcal{O}$-module in the Hermitian space $V$, and $L$ is said to be {\sl a lattice on} $V$ whenever $V=EL$.
We shall assume in the sequel all Hermitian lattices $L$ are positive definite integral in the sense that $h(v)\in\mathbb{Z}_{>0}$ for all nonzero elements $v\in L$ and that $h(v,w)\in\mathcal{O}$ for all pairs of elements $v,w\in L$.
$L$ is said to be represented by another Hermitian lattice $K$, written as $L\to K$, provided there exists an injective linear map $\sigma:L\to K$ which preserves Hermitian maps, {\sl i.e.}, $h_L(v,w)=h_K(\sigma(v),\sigma(w))$ for all $v,w\in L$.
We thus call $\sigma$ a representation from $L$ to $K$.
In addition, $L$ and $K$ are said to be isometric if $\sigma$ is bijective, and the isometry class containing $L$ is denoted by $\mathrm{cls}(L)$.

Assume that $V$ is an $n$-dimensional Hermitian space and $L$ is a Hermitian lattice on $V$.
Then, there exist fractional $\mathcal{O}$-ideals $\mathfrak{A}_1,\ldots,\mathfrak{A}_n$ and a basis $\{v_1,\ldots,v_n\}$ of $V$ such that $L=\mathfrak{A}_1v_1+\cdots+\mathfrak{A}_nv_n$.
In particular, if we have $\mathfrak{A}_1=\cdots=\mathfrak{A}_n=\mathcal{O}$ for some basis $\{v_1,\ldots,v_n\}$, then we call $L$ a {\sl free} Hermitian lattice and associate to it a Gram matrix $M_L=(h(v_l,v_{l'}))$.
For example, a free unary Hermitian lattice $L=\mathcal{O}v$ with $h(v)=a$ can be identified with the Gram matrix $M_L=\langle a\rangle$, and $L\to I_m$ if and only if one can write $a$ as a sum of $m$ norms of integers in $\mathcal{O}$.

The {\sl scale} $\mathfrak{s}L$ of $L$ is the fractional $\mathcal{O}$-ideal generated by the set $\{h(v,w):v,w\in L\}$, and the {\sl volume} of $L$ is the fractional $\mathcal{O}$-ideal $\mathfrak{v}L:=\big(\mathfrak{A}_1\overline{\mathfrak{A}_1}\big)\cdots\big(\mathfrak{A}_n\overline{\mathfrak{A}_n}\big)\mathrm{det}(h(v_l,v_{l'}))$.
Each $\mathfrak{A}_j$ can be written as a product of integral powers of prime ideals in $\mathcal{O}$.
For each prime ideal $\mathfrak{P}$ in $\mathcal{O}$, denote by $p$ the prime number which lies below $\mathfrak{P}$.
Then, one has $\mathfrak{P}\overline{\mathfrak{P}}=p^2\mathcal{O}$ when $p$ is inert in $E$ and $\mathfrak{P}\overline{\mathfrak{P}}=p\mathcal{O}$ otherwise.
Consequently, there exists a unique positive rational number $\delta_L$ with the property that $\mathfrak{v}L=\delta_L\mathcal{O}$, and $\delta_L$ is said to be the {\sl discriminant} of $L$.
Let $\mathcal{I}$ be an $\mathcal{O}$-ideal.
$L$ is said to be $\mathcal{I}$-{\sl modular} if $\mathfrak{s}L=\mathcal{I}$ and $\mathfrak{v}L=\mathcal{I}^n$; in particular, $L$ is said to be {\sl unimodular} when $\mathcal{I}=\mathcal{O}$.

Now, let's briefly go over the main ideas of the proofs.
Knowing that two isometric Hermitian lattices have the same representation properties, from the discussions in the introduction, one only needs to consider the representations of the lattices $L_j^{r_j}=\mathfrak{U}_jv_j^{r_j}$ by $I_m$, with $h\big(v_j^{r_j}\big)=r_j/k_j$ for $1\leq j\leq\mathtt{c}$, where $\mathtt{c}=2$ or $3$.
We need to determine the values of $r_j$ such that $L_j^{r_j}=\mathfrak{U}_jv_j^{r_j}\in\mathfrak{S}_d(1)$ for each $j=1,2,3$.
When $\mathfrak{U}_1=\mathcal{O}$, then $h(v_1)=1$ and $L_1^{r_1}\cong\langle r_1\rangle$ is represented by $I_4$, so that $L_1^{r_1}\in\mathfrak{S}_d(1)$ for every positive integer $r_1$.
When $j=2,3$, rewrite $L_j^{r_j}=\mathfrak{U}_jv_j^{r_j}$ as $\mathcal{O}k_jv_j^{r_j}+\mathcal{O}(s_j+t_j\omega)v_j^{r_j}$ and assume there exists a representation $\sigma:L_j^{r_j}\to I_m$ with $\sigma\big(k_jv_j^{r_j}\big)=\gamma_1z_1+\cdots+\gamma_mz_m$, where $\{z_1,\ldots,z_m\}$ is an orthonormal basis of $I_m$ and $\gamma_\ell=a_\ell+b_\ell\omega\in\mathcal{O}$ for $1\leq\ell\leq m$; then, it follows that
\begin{equation*}
\begin{aligned}
\sigma\big((s_j+t_j\omega)v_j^{r_j}\big)&=\frac{s_j+t_j\omega}{k_j}\sigma\big(k_jv_j^{r_j}\big)\\
&=\frac{s_j+t_j\omega}{k_j}\gamma_1z_1+\cdots+\frac{s_j+t_j\omega}{k_j}\gamma_mz_m.
\end{aligned}
\end{equation*}
By the fact that $\frac{s_j+t_j\hspace{-0.2mm}\omega}{k_j}\gamma_\ell\in\mathcal{O}$, we get congruence conditions on $a_\ell,b_\ell$ for $1\leq\ell\leq m$.
So, $L_j^{r_j}\to I_m$ if and only if there are integers $\gamma_\ell=a_\ell+b_\ell\omega\in\mathcal{O}$ with $a_\ell,b_\ell$ satisfying these congruence conditions such that $r_j/k_j=\sum_{\ell=1}^mN\big(\gamma_\ell/k_j\big)$.
Then, we can use this technique to determine the values of $r_j$ such that $L_j^{r_j}=\mathfrak{U}_jv_j^{r_j}\in\mathfrak{S}_d(1)$ for $j=2,3$, and further determine the value of $g_d(1)$.

The following lemmata play important roles in the proofs of our main results in Sections 3 and 4.

\begin{lem}\label{Lem1}
Let $E=\mathbb{Q}\big(\sqrt{-d}\big)$ be an imaginary quadratic field, and let $\mathcal{O}$ be its ring of integers.
Given an integral ideal $\mathfrak{U}=\mathcal{O}k+\mathcal{O}(s+t\omega)$ with $k\in\mathbb{Z}_{>0}$ and $s+t\omega\in\mathcal{O}$, set $L^r=\mathfrak{U}v^r$ for a positive integer $r$ with $h(v^r)=r/k$.
Suppose there is a representation $\sigma:L^r\to I_m$ with $\sigma(kv^r)=\gamma_1z_1+\cdots+\gamma_mz_m$, where $\{z_1,\ldots,z_m\}$ is an orthonormal basis of $I_m$ and $\gamma_\ell=a_\ell+b_\ell\omega\in\mathcal{O}$ for $1\leq\ell\leq m$.
Then, for $1\leq\ell\leq m$, one has
\begin{description}
\item[~~(1.1)] $k|(sa_\ell-dtb_\ell)$ and $k|(ta_\ell+sb_\ell)$ when $d\equiv1,2~(\mathrm{mod}~4)$;
\item[~~(1.2)] $k|\big(sa_\ell-\frac{1+d}{4}tb_\ell\big)$ and $k|(ta_\ell+(s+t)b_\ell)$ when $d\equiv3~(\mathrm{mod}~4)$.
\end{description}
\end{lem}

\begin{proof}
Note that $\sigma((s+t\omega)v^r)\in I_m$ if and only if $\frac{s+t\hspace{-0.27mm}\omega}{k}\gamma_\ell\in\mathcal{O}$ for $1\leq\ell\leq m$.
When $d\equiv1,2~(\mathrm{mod}~4)$, then $\omega=\sqrt{-d}$ and
\begin{equation*}
\begin{aligned}
(s+t\omega)\gamma_\ell&=\big(s+t\sqrt{-d}\big)\big(a_\ell+b_\ell\sqrt{-d}\big)\\
&=(sa_\ell-dtb_\ell)+(ta_\ell+sb_\ell)\omega;
\end{aligned}
\end{equation*}
so, $\frac{s+t\hspace{-0.27mm}\omega}{k}\gamma_\ell\in\mathcal{O}$ if and only if $k|(sa_\ell-dtb_\ell)$ and $k|(ta_\ell+sb_\ell)$ for each $1\leq\ell\leq m$.
In the other case where $d\equiv3~(\mathrm{mod}~4)$, we have $\omega=\big(1+\sqrt{-d}\big)/2$ and
\begin{equation*}
\begin{aligned}
(s+t\omega)\gamma_\ell&=\bigg(s+t\frac{1+\sqrt{-d}}{2}\bigg)\bigg(a_\ell+b_\ell\frac{1+\sqrt{-d}}{2}\bigg)\\
&=\bigg(sa_\ell-\frac{1+d}{4}tb_\ell\bigg)+(ta_\ell+(s+t)b_\ell)\omega;
\end{aligned}
\end{equation*}
therefore, $\frac{s+t\hspace{-0.27mm}\omega}{k}\gamma_\ell\in\mathcal{O}$ if and only if $k|\big(sa_\ell-\frac{1+d}{4}tb_\ell\big)$ and $k|(ta_\ell+(s+t)b_\ell)$ for each $1\leq\ell\leq m$.
\end{proof}

\begin{lem}\label{Lem2}
Suppose $E=\mathbb{Q}\big(\sqrt{-d}\big)$ is an imaginary quadratic field, and $\mathcal{O}$ is its ring of integers.
Let $m_d$ be the smallest number of norms whose sum represents all positive integers.
Then, we have
\begin{description}
\item[~~(2.1)] when $d=1, 2, 3, 7, 11$, then $m_d=2$;
\item[~~(2.2)] when $d\equiv1,2~(\mathrm{mod}~4)$, then $m_d=3$ if $5\leq d\leq7$ and $m_d=4$ if $d\geq8$;
\item[~~(2.3)] when $d\equiv3~(\mathrm{mod}~4)$,  then $m_d=3$ if $15\leq d\leq27$ and $m_d=4$ if $d\geq28$.
\end{description}
\end{lem}

\begin{proof}
When $d=1, 2, 3, 7, 11$, then $m_d=2$ follows from \cite[Table 3]{KKP}.

When $d\equiv1,2~(\mathrm{mod}~4)$, we easily derive $N\big(a+b\sqrt{-d}\big)=a^2+db^2$ for an integer $a+b\sqrt{-d}\in\mathcal{O}$.
Thus,
\begin{equation*}
\begin{aligned}
&\sum_{\ell=1}^3N\big(a_\ell+b_\ell\sqrt{-d}\big)\\
=&\,a_1^2+db_1^2+a_2^2+db_2^2+a_3^2+db_3^2\cong\langle1,1,1,d,d,d\rangle.
\end{aligned}
\end{equation*}
For $5\leq d\leq7$, we know $\langle1,1,1,d\rangle$ is universal from Conway \cite{jhC} (or \cite[Theorem 1]{mB}) but $3$ cannot be written as a sum of two norms; hence, we need at least three norms to represent all positive integers.
For $d\geq8$, seeing $7$ cannot be written as a sum of three norms, we need at least four norms to represent all positive integers.

When $d\equiv3~(\mathrm{mod}~4)$, we instead have $N\Big(a+b\frac{1+\sqrt{-d}}{2}\Big)=a^2+ab+\frac{1+d}{4}b^2$ for an integer $a+b\frac{1+\sqrt{-d}}{2}\in\mathcal{O}$.
Thus,
\begin{equation*}
\begin{aligned}
&\sum_{\ell=1}^3N\bigg(a_\ell+b_\ell\frac{1+\sqrt{-d}}{2}\bigg)\\
=&\,a_1^2+a_1b_1+\frac{1+d}{4}b_1^2+a_2^2+a_2b_2+\frac{1+d}{4}b_2^2+a_3^2+a_3b_3+\frac{1+d}{4}b_3^2,
\end{aligned}
\end{equation*}
and one can use Bhargava and Hanke \cite[Theorem 1]{BH} to check its universality.
Among the criterion set (see \eqref{Eq2.2} below), all the numbers are sums of three squares except for $7,15,23,31$.
Set $f_d(a_1,b_1,a_2,b_2,a_3,b_3):=\sum_{\ell=1}^3N\Big(a_\ell+b_\ell\frac{1+\sqrt{-d}}{2}\Big)$ for $15\leq d\leq27$; the table below shows $7,5,23,31$ are represented by some sums of three norms.
Thus, the sum of three norms is universal and optimal, since $3$ cannot be written as a sum of two norms.
As $\frac{1+d}{4}>7$ if $d\geq28$ and $7$ cannot be written as a sum of three norms, we need at least four norms to represent all positive integers.
\end{proof}

{\small\begin{center}
\setlength{\arrayrulewidth}{0.2mm}
\setlength{\tabcolsep}{2pt}
\renewcommand{\arraystretch}{0.81}
\begin{tabular}{|m{1.56cm}m{3.44cm}|m{1.56cm}m{3.44cm}|}
\hline\vskip2pt \centering\multirow{5}{*}{$\boxed{d=15}$} & \vskip2pt $f_{15}(1,1,1,0,0,0)=7$  & \vskip2pt \centering\multirow{5}{*}{$\boxed{d=19}$} & \vskip2pt $f_{19}(1,1,0,0,0,0)=7$  \\[0.6ex]
                                                          &           $f_{15}(2,1,1,0,2,0)=15$ &                                                     &
$f_{19}(1,1,2,0,2,0)=15$ \\[0.6ex]
                                                          &          $f_{15}(1,1,1,0,4,0)=23$  &                                                     &
$f_{19}(1,1,4,0,0,0)=23$ \\[0.6ex]
                                                          &          $f_{15}(1,1,5,0,0,0)=31$  &                                                     &
$f_{19}(5,0,1,0,0,1)=31$ \\[0.6ex]

\hline\vskip2pt \centering\multirow{5}{*}{$\boxed{d=23}$} & \vskip2pt $f_{23}(1,0,0,1,0,0)=7$  & \vskip2pt \centering\multirow{5}{*}{$\boxed{d=27}$} & \vskip2pt $f_{27}(0,1,0,0,0,0)=7$  \\[0.6ex]
                                                          &           $f_{23}(1,0,0,1,1,1)=15$ &                                                     &
$f_{27}(0,1,2,0,2,0)=15$ \\[0.6ex]
                                                          &           $f_{23}(1,0,0,1,4,0)=23$ &                                                     &
$f_{27}(0,1,4,0,0,0)=23$ \\[0.6ex]
                                                          &           $f_{23}(5,0,0,1,0,0)=31$ &                                                     &
$f_{27}(2,1,3,0,3,0)=31$ \\[0.6ex]
\hline
\end{tabular}
\end{center}}

\begin{lem}\label{Lem3}
Assume $E=\mathbb{Q}\big(\sqrt{-d}\big)$ is an imaginary quadratic field, and $\mathcal{O}$ is its ring of integers.
Let $r_j$ and $k_j$ be defined as above for $j=2,3$.
Then, for each positive integer $r_j\equiv0~(\mathrm{mod}~k_j)$, there are integers $\gamma_1=a_1+b_1\omega,\ldots,\gamma_{m_d}=a_{m_d}+b_{m_d}\omega\in\mathcal{O}$, with $a_\ell,b_\ell$ satisfying the conditions given in Lemma \ref{Lem1} and with $m_d$ given in Lemma \ref{Lem2}, such that $r_j/k_j=\sum_{\ell=1}^{m_d}N\big(\gamma_\ell/k_j\big)$.
\end{lem}

\begin{proof}
When $r_j\equiv0~(\mathrm{mod}~k_j)$, then $r_j/k_j$ is a positive integer.
From Lemma \ref{Lem2}, we know that $I_{m_d}$ represents all positive integers.
Suppose $r_j/k_j=\sum_{\ell=1}^{m_d}N\big(\tilde{a}_\ell+\tilde{b}_\ell\omega\big)$ for $\tilde{a}_\ell+\tilde{b}_\ell\omega\in\mathcal{O}$.
Then, $a_\ell=k_j\tilde{a}_\ell$ and $b_\ell=k_j\tilde{b}_\ell$ satisfy the congruence conditions in Lemma \ref{Lem1} for $1\leq\ell\leq m_d$, and $r_j/k_j=\sum_{\ell=1}^{m_d}N\big(\gamma_\ell/k_j\big)$.
\end{proof}

Several important results are used frequently in the subsequent discussions to check the universality of some given forms, and I describe them briefly as follows.

To proceed, first recall that the $x$-th triangular number is of the form $T_x=\frac{x(x+1)}{2}$ for an integer $x$.
A quadratic form, or a weighted mixed sum of triangular numbers and squares, is said to be {\sl universal} if it represents all positive integers.
\vskip2pt\noindent $\bullet$ For quadratic forms, the {\sf Conway-Schneeberger fifteen theorem} (see Conway \cite{jhC} and Bhargava \cite[Theorem 1]{mB}) gives a classification of positive definite universal quadratic forms whose corresponding Gram matrices are integral, and this result states that such a quadratic form is universal if and only if it represents the numbers
\begin{equation}\label{Eq2.1}
1,2,3,5,6,7,10,14,15.
\end{equation}
\vskip0pt\noindent $\bullet$ A more general case of positive definite integer-valued quadratic forms (whose Gram matrices are half integral) was considered by Bhargava and Hanke \cite[Theorem 1]{BH}, and the {\sf Bhargava-Hanke 290-theorem} provides the criterion set
\begin{equation}\label{Eq2.2}
\begin{aligned}
&\{1,2,3,5,6,7,10,13,14,15,17,19,21,22,23,26,29,\\
&\,\,\,30,31,34,35,37,42,58,93,110,145,203,290\}
\end{aligned}
\end{equation}
for all such universal forms; that is, a positive definite integer-valued quadratic form is universal if and only if it represents the numbers in the set \eqref{Eq2.2}.
\vskip2pt\noindent $\bullet$ For weighted mixed sums of triangular numbers and squares, Sun \cite[Page 1368]{zS15} provided the list of all universal ternary mixed sums, and we shall use the universality of $2T_x+y^2+z^2$ and $2T_x+2T_y+z^2$ in the sequel.
\vskip2pt\noindent $\bullet$ Occasionally, we also need Jagy, Kaplansky and Schiemann \cite[Table 1]{JKS} to identify regular ternary quadratic forms, and Sun \cite[Theorem 1.2]{zS17} to determine the universal quadratic polynomial of the form $x(ax+1)+y(by+1)+z(cz+1)$.

\section{Computations in the case of class number 2}\label{Sec:Cl2}
In this section, we consider the imaginary quadratic fields $E$ with class number $2$.

\begin{thm}\label{Thm1}
Assume $E=\mathbb{Q}\big(\sqrt{-d}\big)$ is an imaginary quadratic field with class number $2$, and $\mathcal{O}$ is its ring of integers.
Then, the set $\mathfrak{S}_d(1)$ and the value of $g_d(1)$ are listed in the table below
\begin{center}
\setlength{\arrayrulewidth}{0.3mm}
\setlength{\tabcolsep}{3pt}
\renewcommand{\arraystretch}{1.0}
\begin{tabular}{|m{1.68cm}|m{8.96cm}|m{0.96cm}|}
\hline\vskip2pt $\mathbb{Q}\big(\sqrt{-d}\big)$   & \vskip2pt $\mathfrak{S}_d(1)$                                                   & \vskip2pt $g_d(1)$               \\[0.6ex]

\hline\vskip2pt $\mathbb{Q}\big(\sqrt{-5}\big)$   & \vskip2pt $L_1^{r_1}=\mathcal{O}v_1^{r_1}$ with $h\big(v_1^{r_1}\big)=r_1\geq1$ & \vskip2pt $3$                    \\[0.8ex]
                                                  &       $L_2^{r_2}=\mathfrak{U}_2v_2^{r_2}$ with $h\big(v_2^{r_2}\big)=r_2/2$ for $r_2\geq2$                       & \\[0.6ex]

\hline\vskip2pt $\mathbb{Q}\big(\sqrt{-6}\big)$   & \vskip2pt $L_1^{r_1}=\mathcal{O}v_1^{r_1}$ with $h\big(v_1^{r_1}\big)=r_1\geq1$ & \vskip2pt $3$                    \\[0.8ex]
                                                  &       $L_2^{r_2}=\mathfrak{U}_2v_2^{r_2}$ with $h\big(v_2^{r_2}\big)=r_2/2$ for $r_2\geq2$                       & \\[0.6ex]

\hline\vskip2pt $\mathbb{Q}\big(\sqrt{-10}\big)$  & \vskip2pt $L_1^{r_1}=\mathcal{O}v_1^{r_1}$ with $h\big(v_1^{r_1}\big)=r_1\geq1$ & \vskip2pt $4$                    \\[0.8ex]
                                                  &       $L_2^{r_2}=\mathfrak{U}_2v_2^{r_2}$ with $h\big(v_2^{r_2}\big)=r_2/2$ for $r_2\geq2$ and $r_2\neq3$        & \\[0.6ex]

\hline\vskip2pt $\mathbb{Q}\big(\sqrt{-13}\big)$  & \vskip2pt $L_1^{r_1}=\mathcal{O}v_1^{r_1}$ with $h\big(v_1^{r_1}\big)=r_1\geq1$ & \vskip2pt $4$                    \\[0.8ex]
                                                  &       $L_2^{r_2}=\mathfrak{U}_2v_2^{r_2}$ with $h\big(v_2^{r_2}\big)=r_2/2$ for $r_2\geq2$ and $r_2\neq3,5$      & \\[0.6ex]

\hline\vskip2pt $\mathbb{Q}\big(\sqrt{-15}\big)$  & \vskip2pt $L_1^{r_1}=\mathcal{O}v_1^{r_1}$ with $h\big(v_1^{r_1}\big)=r_1\geq1$ & \vskip2pt $3$                    \\[0.8ex]
                                                  &       $L_2^{r_2}=\mathfrak{U}_2v_2^{r_2}$ with $h\big(v_2^{r_2}\big)=r_2/2$ for $r_2\geq2$                       & \\[0.6ex]

\hline\vskip2pt $\mathbb{Q}\big(\sqrt{-22}\big)$  & \vskip2pt $L_1^{r_1}=\mathcal{O}v_1^{r_1}$ with $h\big(v_1^{r_1}\big)=r_1\geq1$ & \vskip2pt $4$                    \\[0.8ex]
                                                  &       $L_2^{r_2}=\mathfrak{U}_2v_2^{r_2}$ with $h\big(v_2^{r_2}\big)=r_2/2$ for $r_2\geq2$ and $r_2\neq3,5,7,9$  & \\[0.6ex]

\hline\vskip2pt $\mathbb{Q}\big(\sqrt{-35}\big)$  & \vskip2pt $L_1^{r_1}=\mathcal{O}v_1^{r_1}$ with $h\big(v_1^{r_1}\big)=r_1\geq1$ & \vskip2pt $4$                    \\[0.8ex]
                                                  &       $L_2^{r_2}=\mathfrak{U}_2v_2^{r_2}$ with $h\big(v_2^{r_2}\big)=r_2/5$ for $r_2\geq3$ and $r_2\neq4$        & \\[0.6ex]

\hline\vskip2pt $\mathbb{Q}\big(\sqrt{-37}\big)$  & \vskip2pt $L_1^{r_1}=\mathcal{O}v_1^{r_1}$ with $h\big(v_1^{r_1}\big)=r_1\geq1$ & \vskip2pt $4$                    \\[0.8ex]
                                                  &       $L_2^{r_2}=\mathfrak{U}_2v_2^{r_2}$ with $h\big(v_2^{r_2}\big)=r_2/2$ for $r_2\geq2$ and $r_2\neq3,5,7,9,$ & \\[0.8ex]
                                                  &       $11,13,15,17$                                                                                              & \\[0.6ex]

\hline\vskip2pt $\mathbb{Q}\big(\sqrt{-51}\big)$  & \vskip2pt $L_1^{r_1}=\mathcal{O}v_1^{r_1}$ with $h\big(v_1^{r_1}\big)=r_1\geq1$ & \vskip2pt $4$                    \\[0.8ex]
                                                  &       $L_2^{r_2}=\mathfrak{U}_2v_2^{r_2}$ with $h\big(v_2^{r_2}\big)=r_2/5$ for $r_2\geq3$ and $r_2\neq4,7$      & \\[0.6ex]
\hline
\end{tabular}
\end{center}

\newpage\noindent{\color{blue}table continued}

\begin{center}
\setlength{\arrayrulewidth}{0.3mm}
\setlength{\tabcolsep}{3pt}
\renewcommand{\arraystretch}{1.0}
\begin{tabular}{|m{1.68cm}|m{8.96cm}|m{0.96cm}|}
\hline\vskip2pt $\mathbb{Q}\big(\sqrt{-58}\big)$  & \vskip2pt $L_1^{r_1}=\mathcal{O}v_1^{r_1}$ with $h\big(v_1^{r_1}\big)=r_1\geq1$ & \vskip2pt $4$                    \\[0.8ex]
                                                  &       $L_2^{r_2}=\mathfrak{U}_2v_2^{r_2}$ with $h\big(v_2^{r_2}\big)=r_2/2$ for $r_2\geq2$ and $r_2\neq3,5,7,9,$ & \\[0.8ex]
                                                  &       $11,13,15,17,19,21,23,25,27$                                                                               & \\[0.6ex]

\hline\vskip2pt $\mathbb{Q}\big(\sqrt{-91}\big)$  & \vskip2pt $L_1^{r_1}=\mathcal{O}v_1^{r_1}$ with $h\big(v_1^{r_1}\big)=r_1\geq1$ & \vskip2pt $4$                    \\[0.8ex]
                                                  &       $L_2^{r_2}=\mathfrak{U}_2v_2^{r_2}$ with $h\big(v_2^{r_2}\big)=r_2/7$ for $r_2\geq5$ and $r_2\neq6,8,9,$   & \\[0.8ex]
                                                  &       $11,16$                                                                                                    & \\[0.6ex]

\hline\vskip2pt $\mathbb{Q}\big(\sqrt{-115}\big)$ & \vskip2pt $L_1^{r_1}=\mathcal{O}v_1^{r_1}$ with $h\big(v_1^{r_1}\big)=r_1\geq1$ & \vskip2pt $4$                    \\[0.8ex]
                                                  &       $L_2^{r_2}=\mathfrak{U}_2v_2^{r_2}$ with $h\big(v_2^{r_2}\big)=r_2/5$ for $r_2\geq5$ and $r_2\neq6,8,9,$   & \\[0.8ex]
                                                  &       $11,13,16,18$                                                                                              & \\[0.6ex]

\hline\vskip2pt $\mathbb{Q}\big(\sqrt{-123}\big)$ & \vskip2pt $L_1^{r_1}=\mathcal{O}v_1^{r_1}$ with $h\big(v_1^{r_1}\big)=r_1\geq1$ & \vskip2pt $4$                    \\[0.8ex]
                                                  &       $L_2^{r_2}=\mathfrak{U}_2v_2^{r_2}$ with $h\big(v_2^{r_2}\big)=r_2/3$ for $r_2\geq3$ and $r_2\neq4,5,7,8,$ & \\[0.8ex]
                                                  &       $10,13,16,19$                                                                                              & \\[0.6ex]

\hline\vskip2pt $\mathbb{Q}\big(\sqrt{-187}\big)$ & \vskip2pt $L_1^{r_1}=\mathcal{O}v_1^{r_1}$ with $h\big(v_1^{r_1}\big)=r_1\geq1$ & \vskip2pt $4$                    \\[0.8ex]
                                                  &       $L_2^{r_2}=\mathfrak{U}_2v_2^{r_2}$ with $h\big(v_2^{r_2}\big)=r_2/7$ for $r_2\geq7$ and $r_2\neq8,9,10,$  & \\[0.8ex]
                                                  &       $12,13,15,16,19,20,23,26,27,30,37$                                                                         & \\[0.6ex]

\hline\vskip2pt $\mathbb{Q}\big(\sqrt{-235}\big)$ & \vskip2pt $L_1^{r_1}=\mathcal{O}v_1^{r_1}$ with $h\big(v_1^{r_1}\big)=r_1\geq1$ & \vskip2pt $4$                    \\[0.8ex]
                                                  &       $L_2^{r_2}=\mathfrak{U}_2v_2^{r_2}$ with $h\big(v_2^{r_2}\big)=r_2/5$ for $r_2\geq5$ and $r_2\neq6,7,8,9,$ & \\[0.8ex]
                                                  &       $11,12,14,16,17,19,21,22,24,27,29,32,34,37,42$                                                             & \\[0.6ex]

\hline\vskip2pt $\mathbb{Q}\big(\sqrt{-267}\big)$ & \vskip2pt $L_1^{r_1}=\mathcal{O}v_1^{r_1}$ with $h\big(v_1^{r_1}\big)=r_1\geq1$ & \vskip2pt $4$                    \\[0.8ex]
                                                  &       $L_2^{r_2}=\mathfrak{U}_2v_2^{r_2}$ with $h\big(v_2^{r_2}\big)=r_2/3$ for $r_2\geq3$ and $r_2\neq4,5,7,8,$ & \\[0.8ex]
                                                  &       $10,11,13,14,16,17,19,20,22,25,28,31,34,37,40,43$                                                          & \\[0.6ex]

\hline\vskip2pt $\mathbb{Q}\big(\sqrt{-403}\big)$ & \vskip2pt $L_1^{r_1}=\mathcal{O}v_1^{r_1}$ with $h\big(v_1^{r_1}\big)=r_1\geq1$ & \vskip2pt $4$                    \\[0.8ex]
                                                  &       $L_2^{r_2}=\mathfrak{U}_2v_2^{r_2}$ with $h\big(v_2^{r_2}\big)=r_2/11$ for $r_2\geq11$ and $r\neq12,14,$   & \\[0.8ex]
                                                  &       $15,16,17,18,19,20,21,23,25,27,28,29,30,32,34,36,38,$                                                      & \\[0.8ex]
                                                  &       $40,41,43,45,47,49,51,54,56,58,60,67,69,71,80,82$                                                          & \\[0.6ex]

\hline\vskip2pt $\mathbb{Q}\big(\sqrt{-427}\big)$ & \vskip2pt $L_1^{r_1}=\mathcal{O}v_1^{r_1}$ with $h\big(v_1^{r_1}\big)=r_1\geq1$ & \vskip2pt $4$                    \\[0.8ex]
                                                  &       $L_2^{r_2}=\mathfrak{U}_2v_2^{r_2}$ with $h\big(v_2^{r_2}\big)=r_2/7$ for $r_2\geq7$ and $r\neq8,9,10,$    & \\[0.8ex]
                                                  &       $11,12,13,15,16,18,19,20,22,23,25,26,27,29,30,32,33,$                                                      & \\[0.8ex]
                                                  &       $36,37,39,40,43,44,46,47,50,53,54,57,60,64,67,71,74,$                                                      & \\[0.8ex]
                                                  &       $81,88$                                                                                                    & \\[0.6ex]
\hline
\end{tabular}
\end{center}
\end{thm}

\begin{proof}
For the sake of clearness and completeness, I shall prove each individual case, although the ideas of the proofs in some cases are similar.

Before to proceed to the proof, recalling Lemmata \ref{Lem2} and \ref{Lem3}, there is the smallest positive integer $m_d$ such that $L_1^{r_1}\to I_{m_d}$ for all positive integers $r_1$ and $L_2^{r_2}\to I_{m_d}$ for every positive integer $r_2\equiv0~(\mathrm{mod}~k_2)$; so, these lattices are in $\mathfrak{S}_d(1)$ and $g_d(1)\geq m_d$.
Thus, it suffices to consider only the cases where $r_2\not\equiv0~(\mathrm{mod}~k_2)$.
We will prove that if $L_2^{r_2}\to I_m$ for some positive integer $m$, then $L_2^{r_2}\to I_{m_d}$, which leads to $g_d(1)=m_d$ (and which seems unknown in the literature and nontrivial).

\vskip8pt\noindent{\bf Case 1.} $E=\mathbb{Q}\big(\sqrt{-5}\big)$ with $\mathcal{O}=\mathbb{Z}+\mathbb{Z}\omega$ for $\omega=\sqrt{-5}$.

By Lemma \ref{Lem2}, one sees $m_5=3$.
Moreover, in this case, through Table 1, we have $\mathfrak{U}_2=\mathcal{O}2+\mathcal{O}(1+\omega)$ and $h(v_2)=1/2$.
So, it follows from Lemma \ref{Lem1} that
\begin{equation*}
(1+\omega)\gamma_\ell=(1+\omega)(a_\ell+b_\ell\omega)=(a_\ell-5b_\ell)+(a_\ell+b_\ell)\omega
\end{equation*}
for each $1\leq\ell\leq m$, which leads to $2|(a_\ell+b_\ell)$.
Therefore, one has
\begin{equation*}
h\big(v_2^{r_2}\big)=\frac{r_2}{2}=\sum_{\ell=1}^mN\Big(\frac{\gamma_\ell}{2}\Big)
=\sum_{\ell=1}^m\bigg(\frac{a_\ell^2}{4}+\frac{5b_\ell^2}{4}\bigg)=:\sum_{\ell=1}^mP_5(a_\ell,b_\ell),
\end{equation*}
where $a_\ell,b_\ell$ are either both odd integers or are both even integers.

When $r_2$ is a positive odd integer, $r_2/2$ is a positive half integer; so, $r_2$ is at least $3$.
We choose $a_1=2\tilde{a}_1+1,b_1=1$ to be odd integers, and $a_\ell=2\tilde{a}_\ell,b_\ell=0$ to be even integers for $2\leq\ell\leq3$.
Then, $\sum_{\ell=1}^3P_5(a_\ell,b_\ell)$ has the form below
\begin{equation*}
\sum_{\ell=1}^3P_5(a_\ell,b_\ell)=\frac{3}{2}+\tilde{a}_1^2+\tilde{a}_1+\tilde{a}_2^2+\tilde{a}_3^2=\frac{3}{2}+2T_{\tilde{a}_1}+\tilde{a}_2^2+\tilde{a}_3^2.
\end{equation*}
Using \cite{zS15}, the above weighted mixed sum $2T_{\tilde{a}_1}+\tilde{a}_2^2+\tilde{a}_3^2$ is universal.
Thus, one sees $r_2/2$ is represented by $\sum_{\ell=1}^3P_5(a_\ell,b_\ell)$ for all odd integers $r_2\geq3$.
Therefore, we have $\mathfrak{S}_5(1)=\big\{\mathcal{O}v_1^{r_1}:r_1\geq1\big\}\bigcup\big\{\mathfrak{U}_2v_2^{r_2}:r_2\geq2\big\}$ and $g_5(1)=3$.

\vskip8pt\noindent{\bf Case 2.} $E=\mathbb{Q}\big(\sqrt{-6}\big)$ with $\mathcal{O}=\mathbb{Z}+\mathbb{Z}\omega$ for $\omega=\sqrt{-6}$.

By Lemma \ref{Lem2}, one sees $m_6=3$.
Moreover, in this case, through Table 1, we have $\mathfrak{U}_2=\mathcal{O}2+\mathcal{O}\omega$ and $h(v_2)=1/2$.
Hence, $\omega\gamma_\ell=\omega(a_\ell+b_\ell\omega)=-6b_\ell+a_\ell\omega$ follows from Lemma \ref{Lem1} for each $1\leq\ell\leq m$, yielding $2|a_\ell$.
Therefore, one has
\begin{equation*}
h\big(v_2^{r_2}\big)=\frac{r_2}{2}=\sum_{\ell=1}^mN\Big(\frac{\gamma_\ell}{2}\Big)
=\sum_{\ell=1}^m\bigg(\frac{a_\ell^2}{4}+\frac{6b_\ell^2}{4}\bigg)=:\sum_{\ell=1}^mP_6(a_\ell,b_\ell),
\end{equation*}
where $a_\ell$ is an even integer and $b_\ell$ is an arbitrary integer.

When $r_2$ is a positive odd integer, $r_2/2$ is a positive half integer; so, $r_2$ is at least $3$.
We choose $a_1=2\tilde{a}_1, b_1=1$ and $a_\ell=2\tilde{a}_\ell,b_\ell=2\tilde{b}_\ell$ for $2\leq\ell\leq3$.
Then,
\begin{equation*}
\sum_{\ell=1}^3P_6(a_\ell,b_\ell)=\frac{3}{2}+\tilde{a}_1^2+\tilde{a}_2^2+\tilde{a}_3^2+6\tilde{b}_2^2+6\tilde{b}_3^2.
\end{equation*}
Recall $\langle1,1,1,6,6\rangle$ is universal via the {\sf Conway-Schneeberger fifteen theorem}.
Thus, $r_2/2$ is represented by $\sum_{\ell=1}^3P_6(a_\ell,b_\ell)$ for all odd integers $r_2\geq3$.
Therefore, we have $\mathfrak{S}_6(1)=\big\{\mathcal{O}v_1^{r_1}:r_1\geq1\big\}\bigcup\big\{\mathfrak{U}_2v_2^{r_2}:r_2\geq2\big\}$ and $g_6(1)=3$.

\vskip8pt\noindent{\bf Case 3.} $E=\mathbb{Q}\big(\sqrt{-10}\big)$ with $\mathcal{O}=\mathbb{Z}+\mathbb{Z}\omega$ for $\omega=\sqrt{-10}$.

By Lemma \ref{Lem2}, one has $m_{10}=4$.
Moreover, in this case, through Table 1, we have $\mathfrak{U}_2=\mathcal{O}2+\mathcal{O}\omega$ and $h(v_2)=1/2$.
Hence, $\omega\gamma_\ell=\omega(a_\ell+b_\ell\omega)=-10b_\ell+a_\ell\omega$ follows from Lemma \ref{Lem1} for each $1\leq\ell\leq m$, yielding $2|a_\ell$.
Therefore, one has
\begin{equation*}
h\big(v_2^{r_2}\big)=\frac{r_2}{2}=\sum_{\ell=1}^mN\Big(\frac{\gamma_\ell}{2}\Big)
=\sum_{\ell=1}^m\bigg(\frac{a_\ell^2}{4}+\frac{10b_\ell^2}{4}\bigg)=:\sum_{\ell=1}^mP_{10}(a_\ell,b_\ell),
\end{equation*}
where $a_\ell$ is an even integer and $b_\ell$ is an arbitrary integer.

When $r_2$ is a positive odd integer, $r_2/2$ is a positive half integer; so, $r_2$ is at least $5$.
We choose $a_1=2\tilde{a}_1,b_1=1$ and $a_\ell=2\tilde{a}_\ell,b_\ell=0$ for $2\leq\ell\leq4$.
Then,
\begin{equation*}
\sum_{\ell=1}^4P_{10}(a_\ell,b_\ell)=\frac{5}{2}+\tilde{a}_1^2+\tilde{a}_2^2+\tilde{a}_3^2+\tilde{a}_4^2.
\end{equation*}
Thus, $r_2/2$ is represented by $\sum_{\ell=1}^4P_{10}(a_\ell,b_\ell)$ for all odd integers $r_2\geq5$.
So, we have $\mathfrak{S}_{10}(1)=\big\{\mathcal{O}v_1^{r_1}:r_1\geq1\big\}\bigcup\big\{\mathfrak{U}_2v_2^{r_2}:r_2\geq2\,\,\text{and}\,\,r_2\neq3\big\}$ and $g_{10}(1)=4$.

\vskip8pt\noindent{\bf Case 4.} $E=\mathbb{Q}\big(\sqrt{-13}\big)$ with $\mathcal{O}=\mathbb{Z}+\mathbb{Z}\omega$ for $\omega=\sqrt{-13}$.

By Lemma \ref{Lem2}, one has $m_{13}=4$.
Moreover, in this case, through Table 1, we have $\mathfrak{U}_2=\mathcal{O}2+\mathcal{O}(1+\omega)$ and $h(v_2)=1/2$.
So, it follows from Lemma \ref{Lem1} that
\begin{equation*}
(1+\omega)\gamma_\ell=(1+\omega)(a_\ell+b_\ell\omega)=(a_\ell-13b_\ell)+(a_\ell+b_\ell)\omega
\end{equation*}
for each $1\leq\ell\leq m$, which leads to $2|(a_\ell+b_\ell)$.
Therefore, one has
\begin{equation*}
h\big(v_2^{r_2}\big)=\frac{r_2}{2}=\sum_{\ell=1}^mN\Big(\frac{\gamma_\ell}{2}\Big)
=\sum_{\ell=1}^m\bigg(\frac{a_\ell^2}{4}+\frac{13b_\ell^2}{4}\bigg)=:\sum_{\ell=1}^mP_{13}(a_\ell,b_\ell),
\end{equation*}
where $a_\ell,b_\ell$ are either both odd integers or are both even integers.

When $r_2$ is a positive odd integer, $r_2/2$ is a positive half integer; so, $r_2$ is at least $7$.
We choose $a_1=2\tilde{a}_1+1,b_1=1$ to be odd integers, and $a_\ell=2\tilde{a}_\ell,b_\ell=0$ to be even integers for $2\leq\ell\leq3$.
Then, $\sum_{\ell=1}^3P_{13}(a_\ell,b_\ell)$ has the form below
\begin{equation*}
\sum_{\ell=1}^3P_{13}(a_\ell,b_\ell)=\frac{7}{2}+\tilde{a}_1^2+\tilde{a}_1+\tilde{a}_2^2+\tilde{a}_3^2=\frac{7}{2}+2T_{\tilde{a}_1}+\tilde{a}_2^2+\tilde{a}_3^2.
\end{equation*}
Using \cite{zS15}, the above weighted mixed sum $2T_{\tilde{a}_1}+\tilde{a}_2^2+\tilde{a}_3^2$ is universal.
Thus, one sees $r_2/2$ is represented by $\sum_{\ell=1}^3P_{13}(a_\ell,b_\ell)$ for all odd integers $r_2\geq7$.
Therefore, we have $\mathfrak{S}_{13}(1)=\big\{\mathcal{O}v_1^{r_1}:r_1\geq1\big\}\bigcup\big\{\mathfrak{U}_2v_2^{r_2}:r_2\geq2\,\,\text{and}\,\,r_2\neq3,5\big\}$ and $g_{13}(1)=4$.

\vskip8pt\noindent{\bf Case 5.} $E=\mathbb{Q}\big(\sqrt{-15}\big)$ with $\mathcal{O}=\mathbb{Z}+\mathbb{Z}\omega$ for $\omega=\frac{1+\sqrt{-15}}{2}$.

By Lemma \ref{Lem2}, one has $m_{15}=3$.
Moreover, in this case, through Table 1, we have $\mathfrak{U}_2=\mathcal{O}2+\mathcal{O}(1+\omega)$ and $h(v_2)=1/2$.
So, it follows from Lemma \ref{Lem1} that
\begin{equation*}
(1+\omega)\gamma_\ell=(1+\omega)(a_\ell+b_\ell\omega)=(a_\ell-4b_\ell)+(a_\ell+2b_\ell)\omega
\end{equation*}
for each $1\leq\ell\leq m$, which leads to $2|a_\ell$.
Therefore, one has
\begin{equation*}
h\big(v_2^{r_2}\big)=\frac{r_2}{2}=\sum_{\ell=1}^mN\Big(\frac{\gamma_\ell}{2}\Big)
=\sum_{\ell=1}^m\bigg(\frac{a_\ell^2}{4}+\frac{a_\ell b_\ell}{4}+\frac{4b_\ell^2}{4}\bigg)=:\sum_{\ell=1}^mP_{15}(a_\ell,b_\ell),
\end{equation*}
where $a_\ell$ is an even integer and $b_\ell$ is an arbitrary integer.

When $r_2$ is a positive odd integer, $r_2/2$ is a positive half integer; so, $r_2$ is at least $3$.
We choose $a_1=2,b_1=-1$ and $a_\ell=2\tilde{a}_\ell,b_\ell=2\tilde{b}_\ell$ for $2\leq\ell\leq3$.
Then,
\begin{equation*}
\sum_{\ell=1}^3P_{15}(a_\ell,b_\ell)=\frac{3}{2}+\tilde{a}_2^2+\tilde{a}_2\tilde{b}_2+4\tilde{b}_2^2+\tilde{a}_3^2+\tilde{a}_3\tilde{b}_3+4\tilde{b}_3^2.
\end{equation*}
From Jagy, Kaplansky and Schiemann \cite[Table 1]{JKS}, the quadratic form $f\big(\tilde{a}_2,\tilde{b}_2,\tilde{a}_3\big)=\tilde{a}_2^2+\tilde{a}_2\tilde{b}_2+4\tilde{b}_2^2+\tilde{a}_3^2$ is regular and satisfies the local-global principle; in fact, locally, we have $f_p\big(\tilde{a}_2,\tilde{b}_2,\tilde{a}_3\big)\cong\langle1,1,15\rangle$ if $p\geq3$.
It is clear that $f_p$ is universal when $p\geq7$; when $p=5$, we get $1,2,5,15\to f_5$; when $p=2$, we get $1,11(\equiv3~(\mathrm{mod}~8)),5,7,2,6,$ $10,14\to f_2$; when $p=3$, $f_3$ represents all the integers in $\mathbb{Z}_3$ except for the square class containing $3$.
Thus, $f\big(\tilde{a}_2,\tilde{b}_2,\tilde{a}_3\big)$ represents every positive integer except for that of the form $3^{2a+1}(3b+1)$ for nonnegative integers $a,b$; as $5/2=P_{15}(2,1)$, for numbers of the form $3/2+3^{2a+1}(3b+1)$, we observe $3/2+3^{2a+1}(3b+1)=5/2+\big(3^{2a+1}(3b+1)-1\big)$, with $3^{2a+1}(3b+1)-1\equiv2~(\mathrm{mod}~3)$ being represented by $f\big(\tilde{a}_2,\tilde{b}_2,\tilde{a}_3\big)$.
Consequently, $r_2/2$ is represented by $\sum_{\ell=1}^3P_{15}(a_\ell,b_\ell)$ for all odd integers $r_2\geq3$.
Therefore, we have $\mathfrak{S}_{15}(1)=\big\{\mathcal{O}v_1^{r_1}:r_1\geq1\big\}\bigcup\big\{\mathfrak{U}_2v_2^{r_2}:r_2\geq2\big\}$ and $g_{15}(1)=3$.

\vskip8pt\noindent{\bf Case 6.} $E=\mathbb{Q}\big(\sqrt{-22}\big)$ with $\mathcal{O}=\mathbb{Z}+\mathbb{Z}\omega$ for $\omega=\sqrt{-22}$.

By Lemma \ref{Lem2}, one has $m_{22}=4$.
Moreover, in this case, through Table 1, we have $\mathfrak{U}_2=\mathcal{O}2+\mathcal{O}\omega$ and $h(v_2)=1/2$.
Thus, $\omega\gamma_\ell=\omega(a_\ell+b_\ell\omega)=-22b_\ell+a_\ell\omega$ follows from Lemma \ref{Lem1} for each $1\leq\ell\leq m$, yielding $2|a_\ell$.
Therefore, one has
\begin{equation*}
h\big(v_2^{r_2}\big)=\frac{r_2}{2}=\sum_{\ell=1}^mN\Big(\frac{\gamma_\ell}{2}\Big)
=\sum_{\ell=1}^m\bigg(\frac{a_\ell^2}{4}+\frac{22b_\ell^2}{4}\bigg)=:\sum_{\ell=1}^mP_{22}(a_\ell,b_\ell),
\end{equation*}
where $a_\ell$ is an even integer and $b_\ell$ is an arbitrary integer.

When $r_2$ is a positive odd integer, $r_2/2$ is a positive half integer; thus, $r_2$ is at least $11$.
We take $a_1=2\tilde{a}_1,b_1=1$ and $a_\ell=2\tilde{a}_\ell,b_\ell=0$ for $2\leq\ell\leq4$.
Then,
\begin{equation*}
\sum_{\ell=1}^4P_{22}(a_\ell,b_\ell)=\frac{11}{2}+\tilde{a}_1^2+\tilde{a}_2^2+\tilde{a}_3^2+\tilde{a}_4^2.
\end{equation*}
Thus, $r_2/2$ is represented by $\sum_{\ell=1}^4P_{22}(a_\ell,b_\ell)$ for all odd integers $r_2\geq11$.
So, we have $\mathfrak{S}_{22}(1)=\big\{\mathcal{O}v_1^{r_1}:r_1\geq1\big\}\bigcup\big\{\mathfrak{U}_2v_2^{r_2}:r_2\geq2\,\,\text{and}\,\,r_2\neq3,5,7,9\big\}$ and $g_{22}(1)=4$.

\vskip8pt\noindent{\bf Case 7.} $E=\mathbb{Q}\big(\sqrt{-35}\big)$ with $\mathcal{O}=\mathbb{Z}+\mathbb{Z}\omega$ for $\omega=\frac{1+\sqrt{-35}}{2}$.

By Lemma \ref{Lem2}, one has $m_{35}=4$.
Moreover, in this case, through Table 1, we have $\mathfrak{U}_2=\mathcal{O}5+\mathcal{O}(2+\omega)$ and $h(v_2)=1/5$.
So, it follows from Lemma \ref{Lem1} that
\begin{equation*}
(2+\omega)\gamma_\ell=(2+\omega)(a_\ell+b_\ell\omega)=(2a_\ell-9b_\ell)+(a_\ell+3b_\ell)\omega
\end{equation*}
for each $1\leq\ell\leq m$, which yields $5|(a_\ell+3b_\ell)$ and further $5|(2a_\ell+b_\ell)$.
So,
\begin{equation*}
h\big(v_2^{r_2}\big)=\frac{r_2}{5}=\sum_{\ell=1}^mN\Big(\frac{\gamma_\ell}{5}\Big)
=\sum_{\ell=1}^m\bigg(\frac{a_\ell^2}{25}+\frac{a_\ell b_\ell}{25}+\frac{9b_\ell^2}{25}\bigg)=:\sum_{\ell=1}^mP_{35}(a_\ell,b_\ell),
\end{equation*}
where $a_\ell$ and $b_\ell$ are integers satisfying $5|(2a_\ell+b_\ell)$.

For $1\leq\delta\leq4$, let $r(\delta)$ be the smallest positive integer such that $r(\delta)\equiv\delta~(\mathrm{mod}~5)$ and $r(\delta)/5$ is represented by $\sum_{\ell=1}^mP_{35}(a_\ell,b_\ell)$ for some positive integer $m$.
Below, we show that by choosing $a_\ell,b_\ell$ properly, $\sum_{\ell=1}^3P_{35}(a_\ell,b_\ell)$, or $\sum_{\ell=1}^4P_{35}(a_\ell,b_\ell)$, is equal to $r(\delta)/5$, plus an universal form.
Thus, $r_2/5$, with $r_2\geq r(\delta)$ and $r_2\equiv\delta~(\mathrm{mod}~5)$, can be represented by $\sum_{\ell=1}^3P_{35}(a_\ell,b_\ell)$ or $\sum_{\ell=1}^4P_{35}(a_\ell,b_\ell)$.
\begin{enumerate}
\item $r(1)=6$: choose $a_\ell=5\tilde{a}_\ell+2,b_\ell=1$ for $1\leq\ell\leq2$ and $a_3=5\tilde{a}_3,b_3=0$ to observe
\begin{equation*}
\sum_{\ell=1}^3P_{35}(a_\ell,b_\ell)=\frac{6}{5}+\tilde{a}_1^2+\tilde{a}_1+\tilde{a}_2^2+\tilde{a}_2+\tilde{a}_3^2=\frac{6}{5}+2T_{\tilde{a}_1}+2T_{\tilde{a}_2}+\tilde{a}_3^2.
\end{equation*}
\item $r(2)=7$: choose $a_1=5\tilde{a}_1-1,b_1=2$ and $a_\ell=5\tilde{a}_\ell,b_\ell=0$ for $2\leq\ell\leq4$ to observe
\begin{equation*}
\sum_{\ell=1}^4P_{35}(a_\ell,b_\ell)=\frac{7}{5}+\tilde{a}_1^2+\tilde{a}_2^2+\tilde{a}_3^2+\tilde{a}_4^2.
\end{equation*}
\item $r(3)=3$: choose $a_1=5\tilde{a}_1+2,b_1=1$ and $a_\ell=5\tilde{a}_\ell,b_\ell=0$ for $2\leq\ell\leq3$ to observe
\begin{equation*}
\sum_{\ell=1}^3P_{35}(a_\ell,b_\ell)=\frac{3}{5}+\tilde{a}_1^2+\tilde{a}_1+\tilde{a}_2^2+\tilde{a}_3^2=\frac{3}{5}+2T_{\tilde{a}_1}+\tilde{a}_2^2+\tilde{a}_3^2.
\end{equation*}
\item $r(4)=9$: choose $a_1=2,b_1=1$, $a_\ell=5\tilde{a}_\ell+2,b_\ell=1$ for $2\leq\ell\leq3$ and $a_4=5\tilde{a}_4,b_4=0$ to observe
\begin{equation*}
\sum_{\ell=1}^4P_{35}(a_\ell,b_\ell)=\frac{9}{5}+\tilde{a}_2^2+\tilde{a}_2+\tilde{a}_3^2+\tilde{a}_3+\tilde{a}_4^2=\frac{9}{5}+2T_{\tilde{a}_2}+2T_{\tilde{a}_3}+\tilde{a}_4^2.
\end{equation*}
\end{enumerate}
Recall that $2T_x+y^2+z^2$ and $2T_x+2T_y+z^2$ are universal by \cite{zS15}.
Therefore, we have $\mathfrak{S}_{35}(1)=\big\{\mathcal{O}v_1^{r_1}:r_1\geq1\big\}\bigcup\big\{\mathfrak{U}_2v_2^{r_2}:r_2\geq3\,\,\text{and}\,\,r_2\neq4\big\}$ and $g_{35}(1)=4$.

\vskip8pt\noindent{\bf Case 8.} $E=\mathbb{Q}\big(\sqrt{-37}\big)$ with $\mathcal{O}=\mathbb{Z}+\mathbb{Z}\omega$ for $\omega=\sqrt{-37}$.

By Lemma \ref{Lem2}, one has $m_{37}=4$.
Moreover, in this case, through Table 1, we have $\mathfrak{U}_2=\mathcal{O}2+\mathcal{O}(1+\omega)$ and $h(v_2)=1/2$.
So, it follows from Lemma \ref{Lem1} that
\begin{equation*}
(1+\omega)\gamma_\ell=(1+\omega)(a_\ell+b_\ell\omega)=(a_\ell-37b_\ell)+(a_\ell+b_\ell)\omega
\end{equation*}
for each $1\leq\ell\leq m$, which leads to $2|(a_\ell+b_\ell)$.
Therefore, one has
\begin{equation*}
h\big(v_2^{r_2}\big)=\frac{r_2}{2}=\sum_{\ell=1}^mN\Big(\frac{\gamma_\ell}{2}\Big)
=\sum_{\ell=1}^m\bigg(\frac{a_\ell^2}{4}+\frac{37b_\ell^2}{4}\bigg)=:\sum_{\ell=1}^mP_{37}(a_\ell,b_\ell),
\end{equation*}
where $a_\ell,b_\ell$ are either both odd integers or are both even integers.

When $r_2$ is a positive odd integer, $r_2/2$ is a positive half integer; thus, $r_2$ is at least $19$.
We take $a_1=2\tilde{a}_1+1,b_1=1$ to be odd integers, and $a_\ell=2\tilde{a}_\ell,b_\ell=0$ to be even integers for $2\leq\ell\leq3$.
Then, $\sum_{\ell=1}^3P_{37}(a_\ell,b_\ell)$ has the form below
\begin{equation*}
\sum_{\ell=1}^3P_{37}(a_\ell,b_\ell)=\frac{19}{2}+\tilde{a}_1^2+\tilde{a}_1+\tilde{a}_2^2+\tilde{a}_3^2=\frac{19}{2}+2T_{\tilde{a}_1}+\tilde{a}_2^2+\tilde{a}_3^2.
\end{equation*}
Since $2T_{\tilde{a}_1}+\tilde{a}_2^2+\tilde{a}_3^2$ is universal by \cite{zS15}, $r_2/2$ is represented by $\sum_{\ell=1}^3P_{37}(a_\ell,b_\ell)$ for all odd integers $r_2\geq19$.
Thus, we have $\mathfrak{S}_{37}(1)=\big\{\mathcal{O}v_1^{r_1}:r_1\geq1\big\}\bigcup\big\{\mathfrak{U}_2v_2^{r_2}:r_2\geq2$
$\text{and}\,\,r_2\neq3,5,7,9,11,13,15,17\big\}$ and $g_{37}(1)=4$.

\vskip8pt\noindent{\bf Case 9.} $E=\mathbb{Q}\big(\sqrt{-51}\big)$ with $\mathcal{O}=\mathbb{Z}+\mathbb{Z}\omega$ for $\omega=\frac{1+\sqrt{-51}}{2}$.

By Lemma \ref{Lem2}, one has $m_{51}=4$.
Moreover, in this case, through Table 1, we have $\mathfrak{U}_2=\mathcal{O}5+\mathcal{O}(1+\omega)$ and $h(v_2)=1/5$.
So, it follows from Lemma \ref{Lem1} that
\begin{equation*}
(1+\omega)\gamma_\ell=(1+\omega)(a_\ell+b_\ell\omega)=(a_\ell-13b_\ell)+(a_\ell+2b_\ell)\omega
\end{equation*}
for each $1\leq\ell\leq m$, which leads to $5|(a_\ell+2b_\ell)$.
Therefore, one has
\begin{equation*}
h\big(v_2^{r_2}\big)=\frac{r_2}{5}=\sum_{\ell=1}^mN\Big(\frac{\gamma_\ell}{5}\Big)
=\sum_{\ell=1}^m\bigg(\frac{a_\ell^2}{25}+\frac{a_\ell b_\ell}{25}+\frac{13b_\ell^2}{25}\bigg)=:\sum_{\ell=1}^mP_{51}(a_\ell,b_\ell),
\end{equation*}
where $a_\ell$ and $b_\ell$ are integers satisfying $5|(a_\ell+2b_\ell)$.

Notice $P_{51}(5\tilde{a}_2,5)+P_{51}(5\tilde{a}_3,0)+P_{51}(5\tilde{a}_4,0)=13+2T_{\tilde{a}_2}+\tilde{a}_3^2+\tilde{a}_4^2$ represents all positive integers $r\geq13$, while $P_{51}(5\tilde{a}_2,0)+P_{51}(5\tilde{a}_3,0)+P_{51}(5\tilde{a}_4,0)=\tilde{a}_2^2+\tilde{a}_3^2+\tilde{a}_4^2$ represents all positive integers $r\neq4^a(8b+7)$ for nonnegative integers $a,b$; moreover, notice $7=P_{51}(2,-1)+P_{51}(4,-2)+P_{51}(10,0)$.
Combining these cases together, one sees that $\sum_{\ell=2}^4P_{51}(a_\ell,b_\ell)$ represents all positive integers.

For $1\leq\delta\leq4$, let $r(\delta)$ be the smallest positive integer such that $r(\delta)\equiv\delta~(\mathrm{mod}~5)$ and $r(\delta)/5$ is represented by $P_{51}(a_1,b_1)$.
Then, as analyzed above, $r_2/5$ is represented by $\sum_{\ell=1}^4P_{51}(a_\ell,b_\ell)$ for every $r_2\geq r(\delta)$ and $r_2\equiv\delta~(\mathrm{mod}~5)$.
It is enough to check the representation of $r_2/5$, with $r_2<r(\delta)$ and $r_2\equiv\delta~(\mathrm{mod}~5)$, by $\sum_{\ell=1}^mP_{51}(a_\ell,b_\ell)$, and further by $\sum_{\ell=1}^4P_{51}(a_\ell,b_\ell)$, below
{\small\begin{center}
\setlength{\arrayrulewidth}{0.2mm}
\setlength{\tabcolsep}{2pt}
\renewcommand{\arraystretch}{0.81}
\begin{tabular}{|m{1.37cm}|m{5.5cm}|m{1.67cm}|}
\hline\vskip2pt $r(\delta)$ & \vskip2pt $r_2<r(\delta)$ and $r_2\equiv\delta~(\mathrm{mod}~5)$ & \vskip2pt $r_2=r(\delta)$ \\[0.5ex]

\hline\vskip2pt $r(1)=11$   & \vskip2pt \,\,\,$6=2P_{51}(2,-1)$                                & \vskip2pt $P_{51}(1,2)$   \\[0.5ex]
                            &           Others cannot be represented.                          &                           \\[0.5ex]

\hline\vskip2pt $r(2)=12$   & \vskip2pt None can be represented.                               & \vskip2pt $P_{51}(4,-2)$  \\[0.5ex]

\hline\vskip2pt $r(3)=3$    & \vskip2pt None can be represented.                               & \vskip2pt $P_{51}(2,-1)$  \\[0.5ex]

\hline\vskip2pt $r(4)=29$   & \vskip2pt \,\,\,$9=3P_{51}(2,-1)$                                & \vskip2pt $P_{51}(12,-1)$ \\[0.5ex]
                            &           $14=P_{51}(1,2)+P_{51}(2,-1)$                          &                           \\[0.5ex]
                            &           $19=P_{51}(1,2)+P_{51}(2,-1)+P_{51}(5,0)$              &                           \\[0.5ex]
                            &           $24=P_{51}(1,2)+P_{51}(2,-1)+2P_{51}(5,0)$             &                           \\[0.5ex]
                            &           Others cannot be represented.                          &                           \\[0.5ex]
\hline
\end{tabular}
\end{center}}
As the first three numbers that can be represented by $P_{51}(a_\ell,b_\ell)$ are $3/5,1,11/5$, it is easy to note $1/5,2/5,4/5,7/5$ cannot be represented by $\sum_{\ell=1}^mP_{51}(a_\ell,b_\ell)$ for any positive integer $m$.
So, we have $\mathfrak{S}_{51}(1)=\big\{\mathcal{O}v_1^{r_1}:r_1\geq1\big\}\bigcup\big\{\mathfrak{U}_2v_2^{r_2}:r_2\geq3\,\,\text{and}\,\,r_2\neq4,7\big\}$ and $g_{51}(1)=4$.

\vskip8pt\noindent{\bf Case 10.} $E=\mathbb{Q}\big(\sqrt{-58}\big)$ with $\mathcal{O}=\mathbb{Z}+\mathbb{Z}\omega$ for $\omega=\sqrt{-58}$.

By Lemma \ref{Lem2}, one has $m_{58}=4$.
Moreover, in this case, through Table 1, we have $\mathfrak{U}_2=\mathcal{O}2+\mathcal{O}\omega$ and $h(v_2)=1/2$.
Thus, $\omega\gamma_\ell=\omega(a_\ell+b_\ell\omega)=-58b_\ell+a_\ell\omega$ follows from Lemma \ref{Lem1} for each $1\leq\ell\leq m$, yielding $2|a_\ell$.
Therefore, one has
\begin{equation*}
h\big(v_2^{r_2}\big)=\frac{r_2}{2}=\sum_{\ell=1}^mN\Big(\frac{\gamma_\ell}{2}\Big)
=\sum_{\ell=1}^m\bigg(\frac{a_\ell^2}{4}+\frac{58b_\ell^2}{4}\bigg)=:\sum_{\ell=1}^mP_{58}(a_\ell,b_\ell),
\end{equation*}
where $a_\ell$ is an even integer and $b_\ell$ is an arbitrary integer.

When $r_2$ is a positive odd integer, $r_2/2$ is a positive half integer; thus, $r_2$ is at least $29$.
We take $a_1=2\tilde{a}_1, b_1=1$ and $a_\ell=2\tilde{a}_\ell,b_\ell=0$ for $2\leq\ell\leq4$.
Then,
\begin{equation*}
\sum_{\ell=1}^4P_{58}(a_\ell,b_\ell)=\frac{29}{2}+\tilde{a}_1^2+\tilde{a}_2^2+\tilde{a}_3^2+\tilde{a}_4^2.
\end{equation*}
So, $r_2/2$ is represented by $\sum_{\ell=1}^4P_{58}(a_\ell,b_\ell)$ for all odd integers $r_2\geq29$.
Therefore, we have $\mathfrak{S}_{58}(1)=\big\{\mathcal{O}v_1^{r_1}:r_1\geq1\big\}\bigcup\big\{\mathfrak{U}_2v_2^{r_2}:r_2\geq2\,\,\text{and}\,\,r_2\neq3,5,7,9,11,13,15,17,$ $19,21,23,25,27\big\}$ and $g_{58}(1)=4$.

\vskip8pt\noindent{\bf Case 11.} $E=\mathbb{Q}\big(\sqrt{-91}\big)$ with $\mathcal{O}=\mathbb{Z}+\mathbb{Z}\omega$ for $\omega=\frac{1+\sqrt{-91}}{2}$.

By Lemma \ref{Lem2}, one has $m_{91}=4$.
Moreover, in this case, through Table 1, we have $\mathfrak{U}_2=\mathcal{O}7+\mathcal{O}(3+\omega)$ and $h(v_2)=1/7$.
So, it follows from Lemma \ref{Lem1} that
\begin{equation*}
(3+\omega)\gamma_\ell=(3+\omega)(a_\ell+b_\ell\omega)=(3a_\ell-23b_\ell)+(a_\ell+4b_\ell)\omega
\end{equation*}
for each $1\leq\ell\leq m$, which yields $7|(a_\ell+4b_\ell)$ and further $7|(2a_\ell+b_\ell)$.
So,
\begin{equation*}
h\big(v_2^{r_2}\big)=\frac{r_2}{7}=\sum_{\ell=1}^mN\Big(\frac{\gamma_\ell}{7}\Big)
=\sum_{\ell=1}^m\bigg(\frac{a_\ell^2}{49}+\frac{a_\ell b_\ell}{49}+\frac{23b_\ell^2}{49}\bigg)=:\sum_{\ell=1}^mP_{91}(a_\ell,b_\ell),
\end{equation*}
where $a_\ell$ and $b_\ell$ are integers satisfying $7|(2a_\ell+b_\ell)$.

For $1\leq\delta\leq 6$, let $r(\delta)$ be the smallest positive integer such that $r(\delta)\equiv\delta~(\mathrm{mod}~7)$ and $r(\delta)/7$ is represented by $\sum_{\ell=1}^mP_{91}(a_\ell,b_\ell)$ for some positive integer $m$.
Below, we show that by choosing $a_\ell,b_\ell$ properly, $\sum_{\ell=1}^3P_{91}(a_\ell,b_\ell)$, or $\sum_{\ell=1}^4P_{91}(a_\ell,b_\ell)$, is equal to $r(\delta)/7$, plus an universal form.
Thus, $r_2/7$, with $r_2\geq r(\delta)$ and $r_2\equiv\delta~(\mathrm{mod}~7)$, can be represented by $\sum_{\ell=1}^3P_{91}(a_\ell,b_\ell)$ or $\sum_{\ell=1}^4P_{91}(a_\ell,b_\ell)$.
\begin{enumerate}
\item $r(1)=15$: choose $a_1=3,b_1=1$, $a_\ell=7\tilde{a}_\ell+3,b_\ell=1$ for $2\leq\ell\leq3$ and $a_4=7\tilde{a}_4,b_4=0$ to observe
\begin{equation*}
\sum_{\ell=1}^4P_{91}(a_\ell,b_\ell)=\frac{15}{7}+\tilde{a}_2^2+\tilde{a}_2+\tilde{a}_3^2+\tilde{a}_3+\tilde{a}_4^2=\frac{15}{7}+2T_{\tilde{a}_2}+2T_{\tilde{a}_3}+\tilde{a}_4^2.
\end{equation*}
\item $r(2)=23$: choose $a_1=-1,b_1=2$, $a_\ell=7\tilde{a}_\ell+3,b_\ell=1$ for $2\leq\ell\leq3$ and $a_4=7\tilde{a}_4,b_4=0$ to observe
\begin{equation*}
\sum_{\ell=1}^4P_{91}(a_\ell,b_\ell)=\frac{23}{7}+\tilde{a}_2^2+\tilde{a}_2+\tilde{a}_3^2+\tilde{a}_3+\tilde{a}_4^2=\frac{23}{7}+2T_{\tilde{a}_2}+2T_{\tilde{a}_3}+\tilde{a}_4^2.
\end{equation*}
\item $r(3)=10$: choose $a_\ell=7\tilde{a}_\ell+3,b_\ell=1$ for $1\leq\ell\leq2$ and $a_3=7\tilde{a}_3,b_3=0$ to observe
\begin{equation*}
\sum_{\ell=1}^3P_{91}(a_\ell,b_\ell)=\frac{10}{7}+\tilde{a}_1^2+\tilde{a}_1+\tilde{a}_2^2+\tilde{a}_2+\tilde{a}_3^2=\frac{10}{7}+2T_{\tilde{a}_1}+2T_{\tilde{a}_2}+\tilde{a}_3^2.
\end{equation*}
\item $r(4)=18$: choose $a_1=-1,b_1=2$, $a_2=7\tilde{a}_2+3,b_2=1$ and $a_\ell=7\tilde{a}_\ell,b_\ell=0$ for $3\leq\ell\leq4$ to observe
\begin{equation*}
\sum_{\ell=1}^4P_{91}(a_\ell,b_\ell)=\frac{18}{7}+\tilde{a}_2^2+\tilde{a}_2+\tilde{a}_3^2+\tilde{a}_4^2=\frac{18}{7}+2T_{\tilde{a}_2}+\tilde{a}_3^2+\tilde{a}_4^2.
\end{equation*}
\item $r(5)=5$: choose $a_1=7\tilde{a}_1+3,b_1=1$ and $a_\ell=7\tilde{a}_\ell,b_\ell=0$ for $2\leq\ell\leq3$ to observe
\begin{equation*}
\sum_{\ell=1}^3P_{91}(a_\ell,b_\ell)=\frac{5}{7}+\tilde{a}_1^2+\tilde{a}_1+\tilde{a}_2^2+\tilde{a}_3^2=\frac{5}{7}+2T_{\tilde{a}_1}+\tilde{a}_2^2+\tilde{a}_3^2.
\end{equation*}
\item $r(6)=13$: choose $a_1=7\tilde{a}_1-1,b_1=2$ and $a_\ell=7\tilde{a}_\ell,b_\ell=0$ for $2\leq\ell\leq4$ to observe
\begin{equation*}
\sum_{\ell=1}^4P_{91}(a_\ell,b_\ell)=\frac{13}{7}+\tilde{a}_1^2+\tilde{a}_2^2+\tilde{a}_3^2+\tilde{a}_4^2.
\end{equation*}
\end{enumerate}
Recall $2T_x+y^2+z^2$ and $2T_x+2T_y+z^2$ are universal by \cite{zS15}.
So, we have $\mathfrak{S}_{91}(1)=\big\{\mathcal{O}v_1^{r_1}:r_1\geq1\big\}\bigcup\big\{\mathfrak{U}_2v_2^{r_2}:r_2\geq5\,\,\text{and}\,\,r_2\neq6,8,9,11,16\big\}$ and $g_{91}(1)=4$.

\vskip8pt\noindent{\bf Case 12.} $E=\mathbb{Q}\big(\sqrt{-115}\big)$ with $\mathcal{O}=\mathbb{Z}+\mathbb{Z}\omega$ for $\omega=\frac{1+\sqrt{-115}}{2}$.

By Lemma \ref{Lem2}, one has $m_{115}=4$.
Moreover, in this case, through Table 1, we have $\mathfrak{U}_2=\mathcal{O}5+\mathcal{O}(-3+\omega)$ and $h(v_2)=1/5$.
So, it follows via Lemma \ref{Lem1} that
\begin{equation*}
(-3+\omega)\gamma_\ell=(-3+\omega)(a_\ell+b_\ell\omega)=(-3a_\ell-29b_\ell)+(a_\ell-2b_\ell)\omega
\end{equation*}
for each $1\leq\ell\leq m$, which yields $5|(a_\ell-2b_\ell)$ and further $5|(2a_\ell+b_\ell)$.
So,
\begin{equation*}
h\big(v_2^{r_2}\big)=\frac{r_2}{5}=\sum_{\ell=1}^mN\Big(\frac{\gamma_\ell}{5}\Big)
=\sum_{\ell=1}^m\bigg(\frac{a_\ell^2}{25}+\frac{a_\ell b_\ell}{25}+\frac{29b_\ell^2}{25}\bigg)=:\sum_{\ell=1}^mP_{115}(a_\ell,b_\ell),
\end{equation*}
where $a_\ell$ and $b_\ell$ are integers satisfying $5|(2a_\ell+b_\ell)$.

For $1\leq\delta\leq 4$, let $r(\delta)$ be the smallest positive integer such that $r(\delta)\equiv\delta~(\mathrm{mod}~5)$ and $r(\delta)/5$ is represented by $\sum_{\ell=1}^mP_{115}(a_\ell,b_\ell)$ for some positive integer $m$.
Below, we show that by picking $a_\ell,b_\ell$ properly, $\sum_{\ell=1}^3P_{115}(a_\ell,b_\ell)$, or $\sum_{\ell=1}^4P_{115}(a_\ell,b_\ell)$, is equal to $r(\delta)/5$, plus an universal form.
Thus, $r_2/5$, with $r_2\geq r(\delta)$ and $r_2\equiv\delta~(\mathrm{mod}~5)$, can be represented by $\sum_{\ell=1}^3P_{115}(a_\ell,b_\ell)$ or $\sum_{\ell=1}^4P_{115}(a_\ell,b_\ell)$.
\begin{enumerate}
\item $r(1)=21$: choose $a_1=2,b_1=1$, $a_\ell=5\tilde{a}_\ell+2,b_\ell=1$ for $2\leq\ell\leq3$ and $a_4=5\tilde{a}_4,b_4=0$ to observe
\begin{equation*}
\sum_{\ell=1}^4P_{115}(a_\ell,b_\ell)=\frac{21}{5}+\tilde{a}_2^2+\tilde{a}_2+\tilde{a}_3^2+\tilde{a}_3+\tilde{a}_4^2=\frac{21}{5}+2T_{\tilde{a}_2}+2T_{\tilde{a}_3}+\tilde{a}_4^2.
\end{equation*}
\item $r(2)=7$: choose $a_1=5\tilde{a}_1+2,b_1=1$ and $a_\ell=5\tilde{a}_\ell,b_\ell=0$ for $2\leq\ell\leq3$ to observe
\begin{equation*}
\sum_{\ell=1}^3P_{115}(a_\ell,b_\ell)=\frac{7}{5}+\tilde{a}_1^2+\tilde{a}_1+\tilde{a}_2^2+\tilde{a}_3^2=\frac{7}{5}+2T_{\tilde{a}_1}+\tilde{a}_2^2+\tilde{a}_3^2.
\end{equation*}
\item $r(3)=23$: choose $a_1=5\tilde{a}_1-1,b_1=2$ and $a_\ell=5\tilde{a}_\ell,b_\ell=0$ for $2\leq\ell\leq4$ to observe
\begin{equation*}
\sum_{\ell=1}^4P_{115}(a_\ell,b_\ell)=\frac{23}{5}+\tilde{a}_1^2+\tilde{a}_2^2+\tilde{a}_3^2+\tilde{a}_4^2.
\end{equation*}
\item $r(4)=14$: choose $a_\ell=5\tilde{a}_\ell+2,b_\ell=1$ for $1\leq\ell\leq2$ and $a_3=5\tilde{a}_3,b_3=0$ to observe
\begin{equation*}
\sum_{\ell=1}^3P_{115}(a_\ell,b_\ell)=\frac{14}{5}+\tilde{a}_1^2+\tilde{a}_1+\tilde{a}_2^2+\tilde{a}_2+\tilde{a}_3^2=\frac{14}{5}+2T_{\tilde{a}_1}+2T_{\tilde{a}_2}+\tilde{a}_3^2.
\end{equation*}
\end{enumerate}
As $2T_x+y^2+z^2$ and $2T_x+2T_y+z^2$ are universal by \cite{zS15}, we have $\mathfrak{S}_{115}(1)=\big\{\mathcal{O}v_1^{r_1}:r_1\geq1\big\}\bigcup\big\{\mathfrak{U}_2v_2^{r_2}:r_2\geq5\,\,\text{and}\,\,r_2\neq6,8,9,11,13,16,18\big\}$ and $g_{115}(1)=4$.

\vskip8pt\noindent{\bf Case 13.} $E=\mathbb{Q}\big(\sqrt{-123}\big)$ with $\mathcal{O}=\mathbb{Z}+\mathbb{Z}\omega$ for $\omega=\frac{1+\sqrt{-123}}{2}$.

By Lemma \ref{Lem2}, one has $m_{123}=4$.
Moreover, in this case, through Table 1, we have $\mathfrak{U}_2=\mathcal{O}3+\mathcal{O}(1+\omega)$ and $h(v_2)=1/3$.
So, it follows from Lemma \ref{Lem1} that
\begin{equation*}
(1+\omega)\gamma_\ell=(1+\omega)(a_\ell+b_\ell\omega)=(a_\ell-31b_\ell)+(a_\ell+2b_\ell)\omega
\end{equation*}
for each $1\leq\ell\leq m$, which leads to $3|(a_\ell+2b_\ell)$.
Therefore, one has
\begin{equation*}
h\big(v_2^{r_2}\big)=\frac{r_2}{3}=\sum_{\ell=1}^mN\Big(\frac{\gamma_\ell}{3}\Big)
=\sum_{\ell=1}^m\bigg(\frac{a_\ell^2}{9}+\frac{a_\ell b_\ell}{9}+\frac{31b_\ell^2}{9}\bigg)=:\sum_{\ell=1}^mP_{123}(a_\ell,b_\ell),
\end{equation*}
where $a_\ell$ and $b_\ell$ are integers satisfying $3|(a_\ell+2b_\ell)$.

For $1\leq\delta\leq2$, let $r(\delta)$ be the smallest positive integer such that $r(\delta)\equiv\delta~(\mathrm{mod}~3)$ and $r(\delta)/3$ is represented by $\sum_{\ell=1}^mP_{123}(a_\ell,b_\ell)$ for some positive integer $m$.
Below, we derive that by choosing $a_\ell,b_\ell$ properly, $\sum_{\ell=1}^3P_{123}(a_\ell,b_\ell)$ is equal to $r(\delta)/3$, plus an universal form.
Thus, $r_2/3$, with $r_2\geq r(\delta)$ and $r_2\equiv\delta~(\mathrm{mod}~3)$, can be represented by $\sum_{\ell=1}^3P_{123}(a_\ell,b_\ell)$.
\begin{enumerate}
\item $r(1)=22$: choose $a_\ell=3\tilde{a}_\ell+2,b_\ell=-1$ for $1\leq\ell\leq2$ and $a_3=3\tilde{a}_3,b_3=0$ to observe
\begin{equation*}
\sum_{\ell=1}^3P_{123}(a_\ell,b_\ell)=\frac{22}{3}+\tilde{a}_1^2+\tilde{a}_1+\tilde{a}_2^2+\tilde{a}_2+\tilde{a}_3^2=\frac{22}{3}+2T_{\tilde{a}_1}+2T_{\tilde{a}_2}+\tilde{a}_3^2.
\end{equation*}
\item $r(2)=11$: choose $a_1=3\tilde{a}_1+2,b_1=-1$ and $a_\ell=3\tilde{a}_\ell,b_\ell=0$ for $2\leq\ell\leq3$ to observe
\begin{equation*}
\sum_{\ell=1}^3P_{123}(a_\ell,b_\ell)=\frac{11}{3}+\tilde{a}_1^2+\tilde{a}_1+\tilde{a}_2^2+\tilde{a}_3^2=\frac{11}{3}+2T_{\tilde{a}_1}+\tilde{a}_2^2+\tilde{a}_3^2.
\end{equation*}
\end{enumerate}
As $2T_x+y^2+z^2$ and $2T_x+2T_y+z^2$ are universal by \cite{zS15}, we have $\mathfrak{S}_{123}(1)=\big\{\mathcal{O}v_1^{r_1}:r_1\geq1\big\}\bigcup\big\{\mathfrak{U}_2v_2^{r_2}:r_2\geq3\,\,\text{and}\,\,r_2\neq4,5,7,8,10,13,16,19\big\}$ and $g_{123}(1)=4$.

\vskip8pt\noindent{\bf Case 14.} $E=\mathbb{Q}\big(\sqrt{-187}\big)$ with $\mathcal{O}=\mathbb{Z}+\mathbb{Z}\omega$ for $\omega=\frac{1+\sqrt{-187}}{2}$.

By Lemma \ref{Lem2}, one has $m_{187}=4$.
Moreover, in this case, through Table 1, we have $\mathfrak{U}_2=\mathcal{O}7+\mathcal{O}(-2+\omega)$ and $h(v_2)=1/7$.
So, it follows via Lemma \ref{Lem1} that
\begin{equation*}
(-2+\omega)\gamma_\ell=(-2+\omega)(a_\ell+b_\ell\omega)=(-2a_\ell-47b_\ell)+(a_\ell-b_\ell)\omega
\end{equation*}
for each $1\leq\ell\leq m$, which leads to $7|(a_\ell-b_\ell)$.
Therefore, one has
\begin{equation*}
h\big(v_2^{r_2}\big)=\frac{r_2}{7}=\sum_{\ell=1}^mN\Big(\frac{\gamma_\ell}{7}\Big)
=\sum_{\ell=1}^m\bigg(\frac{a_\ell^2}{49}+\frac{a_\ell b_\ell}{49}+\frac{47b_\ell^2}{49}\bigg)=:\sum_{\ell=1}^mP_{187}(a_\ell,b_\ell),
\end{equation*}
where $a_\ell$ and $b_\ell$ are integers satisfying $7|(a_\ell-b_\ell)$.

Note $P_{187}(7\tilde{a}_2,7)+P_{187}(7\tilde{a}_3,0)+P_{187}(7\tilde{a}_4,0)=47+2T_{\tilde{a}_2}+\tilde{a}_3^2+\tilde{a}_4^2$ represents all positive integers $r\geq47$, and $P_{187}(7\tilde{a}_2,0)+P_{187}(7\tilde{a}_3,0)+P_{187}(7\tilde{a}_4,0)=\tilde{a}_2^2+\tilde{a}_3^2+\tilde{a}_4^2$ represents all positive integers $r\neq4^a(8b+7)$ for nonnegative integers $a,b$; moreover, note
$15=P_{187}(12,-2)+P_{187}(20,-1)$,
$23=P_{187}(8,1)+P_{187}(10,-4)+P_{187}(14,0)$,
$28=P_{187}(8,1)+P_{187}(10,-4)+P_{187}(21,0)$,
$31=P_{187}(12,-2)+P_{187}(20,-1)+P_{187}(28,0)$,
and $39=P_{187}(17,3)+P_{187}(18,4)$.
Combining these cases together, we see that $\sum_{\ell=2}^4P_{187}(a_\ell,b_\ell)$ represents all positive integers except for $7$.

For $1\leq\delta\leq6$, let $r(\delta)$ be the smallest positive integer such that $r(\delta)\equiv\delta~(\mathrm{mod}~7)$ and $r(\delta)/7$ is represented by $P_{187}(a_1,b_1)$.
Then, as analyzed above, $r_2/7$ is represented by $\sum_{\ell=1}^4P_{187}(a_\ell,b_\ell)$ for every $r_2\geq r(\delta)$ and $r_2\equiv\delta~(\mathrm{mod}~7)$ except for $r_2=r(\delta)+49$.
It is sufficient to check the representations of $r(\delta)/7+7$ and $r_2/7$, with $r_2<r(\delta)$ and $r_2\equiv\delta~(\mathrm{mod}~7)$, by $\sum_{\ell=1}^mP_{187}(a_\ell,b_\ell)$, and further by $\sum_{\ell=1}^4P_{187}(a_\ell,b_\ell)$, below
{\small\begin{center}
\setlength{\arrayrulewidth}{0.2mm}
\setlength{\tabcolsep}{2pt}
\renewcommand{\arraystretch}{0.81}
\begin{tabular}{|m{1.37cm}|m{6.26cm}|m{1.81cm}|m{2.08cm}|}
\hline\vskip2pt $r(\delta)$ & \vskip2pt $r_2<r(\delta)$ and $r_2\equiv\delta~(\mathrm{mod}~7)$ & \vskip2pt $r_2=r(\delta)$  & \vskip2pt $r_2=r(\delta)+49$   \\[0.5ex]

\hline\vskip2pt $r(1)=29$   & \vskip2pt $22=2P_{187}(6,-1)$                                    & \vskip2pt $P_{187}(13,-1)$ & \vskip2pt $P_{187}(8,1)$ $+$   \\[0.5ex]
                            &           Others cannot be represented.                          &                            &           $P_{187}(20,-1)$     \\[0.5ex]

\hline\vskip2pt $r(2)=44$   & \vskip2pt None can be represented.                               & \vskip2pt $P_{187}(12,-2)$ & \vskip2pt $P_{187}(7,0)$ $+$   \\[0.5ex]
                            &                                                                  &                            &           $P_{187}(8,1)$ $+$   \\[0.5ex]
                            &                                                                  &                            &           $P_{187}(14,0)$ $+$  \\[0.5ex]
                            &                                                                  &                            &           $P_{187}(15,1)$      \\[0.5ex]

\hline\vskip2pt $r(3)=17$   & \vskip2pt None can be represented.                               & \vskip2pt $P_{187}(8,1)$   & \vskip2pt $P_{187}(6,-1)$ $+$  \\[0.5ex]
                            &                                                                  &                            &           $2P_{187}(7,0)$ $+$  \\[0.5ex]
                            &                                                                  &                            &           $P_{187}(15,1)$      \\[0.5ex]

\hline\vskip2pt $r(4)=11$   & \vskip2pt None can be represented.                               & \vskip2pt $P_{187}(6,-1)$  & \vskip2pt $2P_{187}(7,0)$ $+$  \\[0.5ex]
                            &                                                                  &                            &           $P_{187}(8,1)$ $+$   \\[0.5ex]
                            &                                                                  &                            &           $P_{187}(13,-1)$     \\[0.5ex]

\hline\vskip2pt $r(5)=61$   & \vskip2pt $33=3P_{187}(6,-1)$                                    & \vskip2pt $P_{187}(20,-1)$ & \vskip2pt $2P_{187}(7,0)$ $+$  \\[0.5ex]
                            &           $40=P_{187}(6,-1)+P_{187}(13,-1)$                      &                            &           $P_{187}(8,1)$ $+$   \\[0.5ex]
                            &           $47=P_{187}(6,-1)+P_{187}(7,0)+P_{187}(13,-1)$         &                            &           $P_{187}(22,1)$      \\[0.5ex]
                            &           $54=P_{187}(6,-1)+2P_{187}(7,0)+P_{187}(13,-1)$        &                            &                                \\[0.5ex]
                            &           Others cannot be represented.                          &                            &                                \\[0.5ex]

\hline\vskip2pt $r(6)=41$   & \vskip2pt $34=2P_{187}(8,1)$                                     & \vskip2pt $P_{187}(15,1)$  & \vskip2pt $P_{187}(13,-1)$ $+$ \\[0.5ex]
                            &           Others cannot be represented.                          &                            &           $P_{187}(20,-1)$     \\[0.5ex]
\hline
\end{tabular}
\end{center}}
As the first five numbers that can be represented by $P_{187}(a_\ell,b_\ell)$ are $1,11/7,17/7,$
$4,29/7$, we see, for $r_2=1,2,3,4,5,6,8,9,10,12,13,15,16,19,20,23,26,27,30,37$, that $r_2/7$ cannot be represented by $\sum_{\ell=1}^mP_{187}(a_\ell,b_\ell)$ for any positive integer $m$.
Thus, we have $\mathfrak{S}_{187}(1)=\big\{\mathcal{O}v_1^{r_1}:r_1\geq1\big\}\bigcup\big\{\mathfrak{U}_2v_2^{r_2}:r_2\geq7\,\,\text{and}\,\,r_2\neq8,9,10,12,13,15,$
$16,19,20,23,26,27,30,37\big\}$ and $g_{187}(1)=4$.

\vskip8pt\noindent{\bf Case 15.} $E=\mathbb{Q}\big(\sqrt{-235}\big)$ with $\mathcal{O}=\mathbb{Z}+\mathbb{Z}\omega$ for $\omega=\frac{1+\sqrt{-235}}{2}$.

By Lemma \ref{Lem2}, one has $m_{235}=4$.
Moreover, in this case, through Table 1, we have $\mathfrak{U}_2=\mathcal{O}5+\mathcal{O}(2+\omega)$ and $h(v_2)=1/5$.
So, it follows from Lemma \ref{Lem1} that
\begin{equation*}
(2+\omega)\gamma_\ell=(2+\omega)(a_\ell+b_\ell\omega)=(2a_\ell-59b_\ell)+(a_\ell+3b_\ell)\omega
\end{equation*}
for each $1\leq\ell\leq m$, which yields $5|(a_\ell+3b_\ell)$ and further $5|(2a_\ell+b_\ell)$.
So,
\begin{equation*}
h\big(v_2^{r_2}\big)=\frac{r_2}{5}=\sum_{\ell=1}^mN\Big(\frac{\gamma_\ell}{5}\Big)
=\sum_{\ell=1}^m\bigg(\frac{a_\ell^2}{25}+\frac{a_\ell b_\ell}{25}+\frac{59b_\ell^2}{25}\bigg)=:\sum_{\ell=1}^mP_{235}(a_\ell,b_\ell),
\end{equation*}
where $a_\ell$ and $b_\ell$ are integers satisfying $5|(2a_\ell+b_\ell)$.

For $1\leq\delta\leq4$, let $r(\delta)$ be the smallest positive integer such that $r(\delta)\equiv\delta~(\mathrm{mod}~5)$ and $r(\delta)/5$ is represented by $\sum_{\ell=1}^mP_{235}(a_\ell,b_\ell)$ for some positive integer $m$.
Below, we show that by picking $a_\ell,b_\ell$ properly, $\sum_{\ell=1}^3P_{235}(a_\ell,b_\ell)$, or $\sum_{\ell=1}^4P_{235}(a_\ell,b_\ell)$, is equal to $r(\delta)/5$, plus an universal form.
Thus, $r_2/5$, with $r_2\geq r(\delta)$ and $r_2\equiv\delta~(\mathrm{mod}~5)$, can be represented by $\sum_{\ell=1}^3P_{235}(a_\ell,b_\ell)$ or $\sum_{\ell=1}^4P_{235}(a_\ell,b_\ell)$.
\begin{enumerate}
\item $r(1)=26$: choose $a_\ell=5\tilde{a}_\ell+2,b_\ell=1$ for $1\leq\ell\leq2$ and $a_3=5\tilde{a}_3,b_3=0$ to observe
\begin{equation*}
\sum_{\ell=1}^3P_{235}(a_\ell,b_\ell)=\frac{26}{5}+\tilde{a}_1^2+\tilde{a}_1+\tilde{a}_2^2+\tilde{a}_2+\tilde{a}_3^2=\frac{26}{5}+2T_{\tilde{a}_1}+2T_{\tilde{a}_2}+\tilde{a}_3^2.
\end{equation*}
\item $r(2)=47$: choose $a_1=5\tilde{a}_1-1,b_1=2$ and $a_\ell=5\tilde{a}_\ell,b_\ell=0$ for $2\leq\ell\leq4$ to observe
\begin{equation*}
\sum_{\ell=1}^4P_{235}(a_\ell,b_\ell)=\frac{47}{5}+\tilde{a}_1^2+\tilde{a}_2^2+\tilde{a}_3^2+\tilde{a}_4^2.
\end{equation*}
\item $r(3)=13$: choose $a_1=5\tilde{a}_1+2,b_1=1$ and $a_\ell=5\tilde{a}_\ell,b_\ell=0$ for $2\leq\ell\leq3$ to observe
\begin{equation*}
\sum_{\ell=1}^3P_{235}(a_\ell,b_\ell)=\frac{13}{5}+\tilde{a}_1^2+\tilde{a}_1+\tilde{a}_2^2+\tilde{a}_3^2=\frac{13}{5}+2T_{\tilde{a}_1}+\tilde{a}_2^2+\tilde{a}_3^2.
\end{equation*}
\item $r(4)=39$: choose $a_1=2,b_1=1$, $a_\ell=5\tilde{a}_\ell+2,b_\ell=1$ for $2\leq\ell\leq3$ and $a_4=5\tilde{a}_4,b_4=0$ to observe
\begin{equation*}
\sum_{\ell=1}^4P_{235}(a_\ell,b_\ell)=\frac{39}{5}+\tilde{a}_2^2+\tilde{a}_2+\tilde{a}_3^2+\tilde{a}_3+\tilde{a}_4^2=\frac{39}{5}+2T_{\tilde{a}_2}+2T_{\tilde{a}_3}+\tilde{a}_4^2.
\end{equation*}
\end{enumerate}
As $2T_x+y^2+z^2$ and $2T_x+2T_y+z^2$ are universal by \cite{zS15}, we have $\mathfrak{S}_{235}(1)=\big\{\mathcal{O}v_1^{r_1}:r_1\geq1\big\}\bigcup\big\{\mathfrak{U}_2v_2^{r_2}:r_2\geq5\,\,\text{and}\,\,r_2\neq6,7,8,9,11,12,14,16,17,19,21,22,24,27,29,$ $32,34,37,42\big\}$ and $g_{235}(1)=4$.

\vskip8pt\noindent{\bf Case 16.} $E=\mathbb{Q}\big(\sqrt{-267}\big)$ with $\mathcal{O}=\mathbb{Z}+\mathbb{Z}\omega$ for $\omega=\frac{1+\sqrt{-267}}{2}$.

By Lemma \ref{Lem2}, one has $m_{267}=4$.
Moreover, in this case, through Table 1, we have $\mathfrak{U}_2=\mathcal{O}3+\mathcal{O}(1+\omega)$ and $h(v_2)=1/3$.
So, it follows from Lemma \ref{Lem1} that
\begin{equation*}
(1+\omega)\gamma_\ell=(1+\omega)(a_\ell+b_\ell\omega)=(a_\ell-67b_\ell)+(a_\ell+2b_\ell)\omega
\end{equation*}
for each $1\leq\ell\leq m$, which leads to $3|(a_\ell+2b_\ell)$.
Therefore, one has
\begin{equation*}
h\big(v_2^{r_2}\big)=\frac{r_2}{3}=\sum_{\ell=1}^mN\Big(\frac{\gamma_\ell}{3}\Big)
=\sum_{\ell=1}^m\bigg(\frac{a_\ell^2}{9}+\frac{a_\ell b_\ell}{9}+\frac{67b_\ell^2}{9}\bigg)=:\sum_{\ell=1}^mP_{267}(a_\ell,b_\ell),
\end{equation*}
where $a_\ell$ and $b_\ell$ are integers satisfying $3|(a_\ell+2b_\ell)$.

For $1\leq\delta\leq2$, let $r(\delta)$ be the smallest positive integer such that $r(\delta)\equiv\delta~(\mathrm{mod}~3)$ and $r(\delta)/3$ is represented by $\sum_{\ell=1}^mP_{267}(a_\ell,b_\ell)$ for some positive integer $m$.
Below, we derive that by choosing $a_\ell,b_\ell$ properly, $\sum_{\ell=1}^3P_{267}(a_\ell,b_\ell)$ is equal to $r(\delta)/3$, plus an universal form.
Thus, $r_2/3$, with $r_2\geq r(\delta)$ and $r_2\equiv\delta~(\mathrm{mod}~3)$, can be represented by $\sum_{\ell=1}^3P_{267}(a_\ell,b_\ell)$.
\begin{enumerate}
\item $r(1)=46$: choose $a_\ell=3\tilde{a}_\ell+2,b_\ell=-1$ for $1\leq\ell\leq2$ and $a_3=3\tilde{a}_3,b_3=0$ to observe
\begin{equation*}
\sum_{\ell=1}^3P_{267}(a_\ell,b_\ell)=\frac{46}{3}+\tilde{a}_1^2+\tilde{a}_1+\tilde{a}_2^2+\tilde{a}_2+\tilde{a}_3^2=\frac{46}{3}+2T_{\tilde{a}_1}+2T_{\tilde{a}_2}+\tilde{a}_3^2.
\end{equation*}
\item $r(2)=23$: choose $a_1=3\tilde{a}_1+2,b_1=-1$ and $a_\ell=3\tilde{a}_\ell,b_\ell=0$ for $2\leq\ell\leq3$ to observe
\begin{equation*}
\sum_{\ell=1}^3P_{267}(a_\ell,b_\ell)=\frac{23}{3}+\tilde{a}_1^2+\tilde{a}_1+\tilde{a}_2^2+\tilde{a}_3^2=\frac{23}{3}+2T_{\tilde{a}_1}+\tilde{a}_2^2+\tilde{a}_3^2.
\end{equation*}
\end{enumerate}
As $2T_x+y^2+z^2$ and $2T_x+2T_y+z^2$ are universal by \cite{zS15}, we have $\mathfrak{S}_{267}(1)=\big\{\mathcal{O}v_1^{r_1}:r_1\geq1\big\}\bigcup\big\{\mathfrak{U}_2v_2^{r_2}:r_2\geq3\,\,\text{and}\,\,r_2\neq4,5,7,8,10,11,13,14,16,17,19,20,22,25,28,$ $31,34,37,40,43\big\}$ and $g_{267}(1)=4$.

\vskip8pt\noindent{\bf Case 17.} $E=\mathbb{Q}\big(\sqrt{-403}\big)$ with $\mathcal{O}=\mathbb{Z}+\mathbb{Z}\omega$ for $\omega=\frac{1+\sqrt{-403}}{2}$.

By Lemma \ref{Lem2}, one has $m_{403}=4$.
Moreover, in this case, through Table 1, we have $\mathfrak{U}_2=\mathcal{O}11+\mathcal{O}(6+\omega)$ and $h(v_2)=1/11$.
So, it follows by Lemma \ref{Lem1} that
\begin{equation*}
(6+\omega)\gamma_\ell=(6+\omega)(a_\ell+b_\ell\omega)=(6a_\ell-101b_\ell)+(a_\ell+7b_\ell)\omega
\end{equation*}
for each $1\leq\ell\leq m$, which leads to $11|(a_\ell+7b_\ell)$.
Therefore, one has
\begin{equation*}
h\big(v_2^{r_2}\big)=\frac{r_2}{11}=\sum_{\ell=1}^mN\Big(\frac{\gamma_\ell}{11}\Big)
=\sum_{\ell=1}^m\bigg(\frac{a_\ell^2}{121}+\frac{a_\ell b_\ell}{121}+\frac{101b_\ell^2}{121}\bigg)=:\sum_{\ell=1}^mP_{403}(a_\ell,b_\ell),
\end{equation*}
where $a_\ell$ and $b_\ell$ are integers satisfying $11|(a_\ell+7b_\ell)$.

Notice $P_{403}(11\tilde{a}_2,11)+P_{403}(11\tilde{a}_3,0)+P_{403}(11\tilde{a}_4,0)=101+2T_{\tilde{a}_2}+\tilde{a}_3^2+\tilde{a}_4^2$ represents all positive integers $r\geq101$, and $P_{403}(11\tilde{a}_2,0)+P_{403}(11\tilde{a}_3,0)+P_{403}(11\tilde{a}_4,0)=\tilde{a}_2^2+\tilde{a}_3^2+\tilde{a}_4^2$ represents all positive integers $r\neq4^a(8b+7)$ for nonnegative integers $a,b$; in addition, one can observe that
$15=P_{403}(3,-2)+P_{403}(11,0)+P_{403}(21,-3)$,
$23=P_{403}(2,-5)+P_{403}(7,-1)+P_{403}(11,0)$,
$28=P_{403}(11,0)+P_{403}(25,-2)+P_{403}(28,-4)$,
$31=P_{403}(22,0)+P_{403}(25,-2)+P_{403}(28,-4)$,
$39=P_{403}(3,-2)+P_{403}(21,-3)+P_{403}(55,0)$,
$47=P_{403}(2,-5)+P_{403}(7,-1)+P_{403}(55,0)$,
$55=P_{403}(7,-1)$ $+$\\ $P_{403}(72,-4)$,
$60=P_{403}(26,1)+P_{403}(76,-3)$,
$63=P_{403}(3,-2)+P_{403}(21,-3)+P_{403}(77,0)$,
$71=P_{403}(2,-5)+P_{403}(7,-1)+P_{403}(77,0)$,
$79=P_{403}(18,-1)+P_{403}(56,-8)$,
$87=P_{403}(62,-1)+P_{403}(80,-2)$,
$92=P_{403}(29,-1)+P_{403}(98,-3)$,
and $95=P_{403}(18,-1)+P_{403}(44,0)+P_{403}(56,-8)$.
Merging these cases together, we see that $\sum_{\ell=2}^4P_{403}(a_\ell,b_\ell)$ represents all positive integers except for $7$.

For $1\leq\delta\leq10$, let $r(\delta)$ be the smallest positive integer such that $r(\delta)\equiv\delta~(\mathrm{mod}~11)$ and $r(\delta)/11$ is represented by $P_{403}(a_1,b_1)$.
Then, as shown above, $r_2/11$ is represented by $\sum_{\ell=1}^4P_{403}(a_\ell,b_\ell)$ for every $r_2\geq r(\delta)$ and $r_2\equiv\delta~(\mathrm{mod}~11)$ except for $r_2=r(\delta)+77$.
It is enough to check the representations of $r(\delta)/11+7$ and $r_2/11$, with $r_2<r(\delta)$ and $r_2\equiv\delta~(\mathrm{mod}~11)$, by $\sum_{\ell=1}^mP_{403}(a_\ell,b_\ell)$, and further by $\sum_{\ell=1}^4P_{403}(a_\ell,b_\ell)$, below
{\small\begin{center}
\setlength{\arrayrulewidth}{0.2mm}
\setlength{\tabcolsep}{2pt}
\renewcommand{\arraystretch}{0.81}
\begin{tabular}{|m{1.71cm}|m{6.6cm}|m{1.81cm}|m{2.08cm}|}
\hline\vskip2pt $r(\delta)$ & \vskip2pt $r_2<r(\delta)$ and $r_2\equiv\delta~(\mathrm{mod}~11)$ & \vskip2pt $r_2=r(\delta)$  & \vskip2pt $r_2=r(\delta)+77$   \\[0.5ex]

\hline\vskip2pt $r(1)=89$   & \vskip2pt \,\,\,$78=2P_{403}(7,-1)+P_{403}(14,-2)$                & \vskip2pt $P_{403}(10,-3)$ & \vskip2pt 2$P_{403}(29,-1)$    \\[0.5ex]
                            &           Others cannot be represented.                           &                            &                                \\[0.5ex]

\hline\vskip2pt $r(2)=13$   & \vskip2pt None can be represented.                                & \vskip2pt $P_{403}(7,-1)$  & \vskip2pt $P_{403}(3, -2)$ $+$ \\[0.5ex]
                            &                                                                   &                            &           2$P_{403}(11,0)$ $+$ \\[0.5ex]
                            &                                                                   &                            &           $P_{403}(15,1)$      \\[0.5ex]

\hline\vskip2pt $r(3)=124$  & \vskip2pt \,\,\,$91=3P_{403}(7,-1)+P_{403}(14,-2)$                & \vskip2pt $P_{403}(30,2)$  & \vskip2pt $P_{403}(11,0)$ $+$  \\[0.5ex]
                            &           $102=P_{403}(7,-1)+P_{403}(10,-3)$                      &                            &           $P_{403}(21,-3)$ $+$ \\[0.5ex]
                            &           $113=P_{403}(7,-1)+P_{403}(10,-3)+P_{403}(11,0)$        &                            &           $P_{403}(26,1)$      \\[0.5ex]
                            &           Others cannot be represented.                           &                            &                                \\[0.5ex]

\hline\vskip2pt $r(4)=37$   & \vskip2pt \,\,\,$26=2P_{403}(7,-1)$                               & \vskip2pt $P_{403}(18,-1)$ & \vskip2pt $P_{403}(15,1)$ $+$  \\[0.5ex]
                            &           Others cannot be represented.                           &                            &           $P_{403}(29,-1)$     \\[0.5ex]

\hline\vskip2pt $r(5)=137$  & \vskip2pt \,\,\,$93=3P_{403}(15,1)$                               & \vskip2pt $P_{403}(37,1)$  & \vskip2pt $P_{403}(5,4)$ $+$   \\[0.5ex]
                            &           $104=P_{403}(15,1)+P_{403}(19,2)$                       &                            &           $P_{403}(11,0)$ $+$  \\[0.5ex]
                            &           $115=P_{403}(11,0)+P_{403}(15,1)+P_{403}(19,2)$         &                            &           $P_{403}(14,-2)$     \\[0.5ex]
                            &           $126=2P_{403}(11,0)+P_{403}(15,1)+P_{403}(19,2)$        &                            &                                \\[0.5ex]
                            &           Others cannot be represented.                           &                            &                                \\[0.5ex]

\hline\vskip2pt $r(6)=83$   & \vskip2pt \,\,\,$39=3P_{403}(7,-1)$                               & \vskip2pt $P_{403}(29,-1)$ & \vskip2pt $P_{403}(3,-2)$ $+$  \\[0.5ex]
                            &           \,\,\,$50=P_{403}(3,-2)+P_{403}(7,-1)$                  &                            &           $P_{403}(7,-1)$ $+$  \\[0.5ex]
                            &           \,\,\,$61=P_{403}(3,-2)+P_{403}(7,-1)+P_{403}(11,0)$    &                            &           $P_{403}(11,0)$ $+$  \\[0.5ex]
                            &           \,\,\,$72=P_{403}(3,-2)+P_{403}(7,-1)+2P_{403}(11,0)$   &                            &           $P_{403}(12,3)$      \\[0.5ex]
                            &           Others cannot be represented.                           &                            &                                \\[0.5ex]

\hline\vskip2pt $r(7)=73$   & \vskip2pt \,\,\,$62=2P_{403}(15,1)$                               & \vskip2pt $P_{403}(26,1)$  & \vskip2pt $P_{403}(7,-1)$ $+$  \\[0.5ex]
                            &           Others cannot be represented.                           &                            &           $P_{403}(37,1)$      \\[0.5ex]

\hline\vskip2pt $r(8)=52$   & \vskip2pt None can be represented.                                & \vskip2pt $P_{403}(14,-2)$ & \vskip2pt $P_{403}(3,-2)$ $+$  \\[0.5ex]
                            &                                                                   &                            &           $P_{403}(8,2)$ $+$   \\[0.5ex]
                            &                                                                   &                            &           $P_{403}(11,0)$ $+$  \\[0.5ex]
                            &                                                                   &                            &           $P_{403}(18,-1)$     \\[0.5ex]
\hline
\end{tabular}
\end{center}}

\newpage\noindent{\color{blue}table continued}

{\small\begin{center}
\setlength{\arrayrulewidth}{0.2mm}
\setlength{\tabcolsep}{2pt}
\renewcommand{\arraystretch}{0.81}
\begin{tabular}{|m{1.71cm}|m{6.6cm}|m{1.81cm}|m{2.08cm}|}
\hline\vskip2pt $r(9)=31$   & \vskip2pt None can be represented.                                & \vskip2pt $P_{403}(15,1)$  & \vskip2pt $P_{403}(4,1)$ $+$   \\[0.5ex]
                            &                                                                   &                            &           $P_{403}(7,-1)$ $+$  \\[0.5ex]
                            &                                                                   &                            &           $P_{403}(11,0)$ $+$  \\[0.5ex]
                            &                                                                   &                            &           $P_{403}(19,2)$      \\[0.5ex]

\hline\vskip2pt $r(10)=197$ & \vskip2pt \,\,\,$65=P_{403}(7,-1)+P_{403}(14,-2)$                 & \vskip2pt $P_{403}(34,3)$  & \vskip2pt $P_{403}(10,-3)$ $+$ \\[0.5ex]
                            &           \,\,\,$76=P_{403}(7,-1)+P_{403}(11,0)+P_{403}(14,-2)$   &                            &           $P_{403}(18,-1)$ $+$ \\[0.5ex]
                            &           \,\,\,$87=P_{403}(7,-1)+2P_{403}(11,0)+P_{403}(14,-2)$  &                            &           $P_{403}(36,-2)$     \\[0.5ex]
                            &           \,\,\,$98=2P_{403}(3,-2)+P_{403}(7,-1)+P_{403}(11,0)$   &                            &                                \\[0.5ex]
                            &           $109=P_{403}(7,-1)+P_{403}(14,-2)+P_{403}(22,0)$        &                            &                                \\[0.5ex]
                            &           $120=P_{403}(1,3)+P_{403}(3,-2)$                        &                            &                                \\[0.5ex]
                            &           $131=P_{403}(1,3)+P_{403}(3,-2)+P_{403}(11,0)$          &                            &                                \\[0.5ex]
                            &           $142=P_{403}(1,3)+P_{403}(3,-2)+2P_{403}(11,0)$         &                            &                                \\[0.5ex]
                            &           $153=P_{403}(1,3)+2P_{403}(7,-1)+P_{403}(22,0)$         &                            &                                \\[0.5ex]
                            &           $164=P_{403}(7,-1)+P_{403}(40,-1)$                      &                            &                                \\[0.5ex]
                            &           $175=P_{403}(7,-1)+P_{403}(11,0)+P_{403}(40,-1)$        &                            &                                \\[0.5ex]
                            &           $186=P_{403}(7,-1)+2P_{403}(11,0)+P_{403}(40,-1)$       &                            &                                \\[0.5ex]
                            &           Others cannot be represented.                           &                            &                                \\[0.5ex]
\hline
\end{tabular}
\end{center}}
Since the first eight numbers that can be represented by $P_{403}(a_\ell,b_\ell)$ are $1,13/11,$
$31/11,37/11,4,52/11,73/11,83/11$, we see, for $r_2=1,2,3,4,5,6,7,8,9,10,12,14,$
$15,16,17,18,19,20,21,23,25,27,28,29,30,32,34,36,38,40,41,43,45,47,49,51,$\newline
$54,56,58,60,67,69,71,80,82$, that $r_2/11$ cannot be represented by $\sum_{\ell=1}^mP_{403}(a_\ell,b_\ell)$ for any positive integer $m$.
Thus, $\mathfrak{S}_{403}(1)=\big\{\mathcal{O}v_1^{r_1}:r_1\geq1\big\}\bigcup\big\{\mathfrak{U}_2v_2^{r_2}:r_2\geq11\,\,\text{and}$
$\,\,r_2\neq12,14,15,16,17,18,19,20,21,23,25,27,28,29,30,32,34,36,38,40,41,43,45,$
$47,49,51,54,56,58,60,67,69,71,80,82\big\}$ and $g_{403}(1)=4$.

\vskip8pt\noindent{\bf Case 18.} $E=\mathbb{Q}\big(\sqrt{-427}\big)$ with $\mathcal{O}=\mathbb{Z}+\mathbb{Z}\omega$ for $\omega=\frac{1+\sqrt{-427}}{2}$.

By Lemma \ref{Lem2}, one has $m_{427}=4$.
Moreover, in this case, through Table 1, we have $\mathfrak{U}_2=\mathcal{O}7+\mathcal{O}(3+\omega)$ and $h(v_2)=1/7$.
So, it follows from Lemma \ref{Lem1} that
\begin{equation*}
(3+\omega)\gamma_\ell=(3+\omega)(a_\ell+b_\ell\omega)=(3a_\ell-107b_\ell)+(a_\ell+4b_\ell)\omega
\end{equation*}
for each $1\leq\ell\leq m$, which yields $7|(a_\ell+4b_\ell)$ and further $7|(2a_\ell+b_\ell)$.
So,
\begin{equation*}
h\big(v_2^{r_2}\big)=\frac{r_2}{7}=\sum_{\ell=1}^mN\Big(\frac{\gamma_\ell}{7}\Big)
=\sum_{\ell=1}^m\bigg(\frac{a_\ell^2}{49}+\frac{a_\ell b_\ell}{49}+\frac{107b_\ell^2}{49}\bigg)=:\sum_{\ell=1}^mP_{427}(a_\ell,b_\ell),
\end{equation*}
where $a_\ell$ and $b_\ell$ are integers satisfying $7|(2a_\ell+b_\ell)$.

For $1\leq\delta\leq 6$, let $r(\delta)$ be the smallest positive integer such that $r(\delta)\equiv\delta~(\mathrm{mod}~7)$ and $r(\delta)/7$ is represented by $\sum_{\ell=1}^mP_{427}(a_\ell,b_\ell)$ for some positive integer $m$.
Below, we show that by picking $a_\ell,b_\ell$ properly, $\sum_{\ell=1}^3P_{427}(a_\ell,b_\ell)$, or $\sum_{\ell=1}^4P_{427}(a_\ell,b_\ell)$, is equal to $r(\delta)/7$, plus an universal form.
Thus, $r_2/7$, with $r_2\geq r(\delta)$ and $r_2\equiv\delta~(\mathrm{mod}~7)$, can be represented by $\sum_{\ell=1}^3P_{427}(a_\ell,b_\ell)$ or $\sum_{\ell=1}^4P_{427}(a_\ell,b_\ell)$.
\begin{enumerate}
\item $r(1)=78$: choose $a_1=1,b_1=-2$, $a_2=7\tilde{a}_2+3,b_2=1$ and $a_\ell=7\tilde{a}_\ell,b_\ell=0$ for $3\leq\ell\leq4$ to observe
\begin{equation*}
\sum_{\ell=1}^4P_{427}(a_\ell,b_\ell)=\frac{78}{7}+\tilde{a}_2^2+\tilde{a}_2+\tilde{a}_3^2+\tilde{a}_4^2=\frac{78}{7}+2T_{\tilde{a}_2}+\tilde{a}_3^2+\tilde{a}_4^2.
\end{equation*}
\item $r(2)=51$: choose $a_1=3,b_1=1$, $a_\ell=7\tilde{a}_\ell+3,b_\ell=1$ for $2\leq\ell\leq3$ and $a_4=7\tilde{a}_4,b_4=0$ to observe
\begin{equation*}
\sum_{\ell=1}^4P_{427}(a_\ell,b_\ell)=\frac{51}{7}+\tilde{a}_2^2+\tilde{a}_2+\tilde{a}_3^2+\tilde{a}_3+\tilde{a}_4^2=\frac{51}{7}+2T_{\tilde{a}_2}+2T_{\tilde{a}_3}+\tilde{a}_4^2.
\end{equation*}
\item $r(3)=17$: choose $a_1=7\tilde{a}_1+3,b_1=1$ and $a_\ell=7\tilde{a}_\ell,b_\ell=0$ for $2\leq\ell\leq3$ to observe
\begin{equation*}
\sum_{\ell=1}^3P_{427}(a_\ell,b_\ell)=\frac{17}{7}+\tilde{a}_1^2+\tilde{a}_1+\tilde{a}_2^2+\tilde{a}_3^2=\frac{17}{7}+2T_{\tilde{a}_1}+\tilde{a}_2^2+\tilde{a}_3^2.
\end{equation*}
\item $r(4)=95$: choose $a_1=1,b_1=-2$, $a_\ell=7\tilde{a}_\ell+3,b_\ell=1$ for $2\leq\ell\leq3$ and $a_4=7\tilde{a}_4,b_4=0$ to observe
\begin{equation*}
\sum_{\ell=1}^4P_{427}(a_\ell,b_\ell)=\frac{95}{7}+\tilde{a}_2^2+\tilde{a}_2+\tilde{a}_3^2+\tilde{a}_3+\tilde{a}_4^2=\frac{95}{7}+2T_{\tilde{a}_2}+2T_{\tilde{a}_3}+\tilde{a}_4^2.
\end{equation*}
\item $r(5)=61$: choose $a_1=7\tilde{a}_1+1,b_1=-2$ and $a_\ell=7\tilde{a}_\ell,b_\ell=0$ for $2\leq\ell\leq4$ to observe
\begin{equation*}
\sum_{\ell=1}^4P_{427}(a_\ell,b_\ell)=\frac{61}{7}+\tilde{a}_1^2+\tilde{a}_2^2+\tilde{a}_3^2+\tilde{a}_4^2.
\end{equation*}
\item $r(6)=34$: choose $a_\ell=7\tilde{a}_\ell+3,b_\ell=1$ for $1\leq\ell\leq2$ and $a_3=7\tilde{a}_3,b_3=0$ to observe
\begin{equation*}
\sum_{\ell=1}^3P_{427}(a_\ell,b_\ell)=\frac{34}{7}+\tilde{a}_1^2+\tilde{a}_1+\tilde{a}_2^2+\tilde{a}_2+\tilde{a}_3^2=\frac{34}{7}+2T_{\tilde{a}_1}+2T_{\tilde{a}_2}+\tilde{a}_3^2.
\end{equation*}
\end{enumerate}
Recall that $2T_x+y^2+z^2$ and $2T_x+2T_y+z^2$ are universal by \cite{zS15}.
Therefore, we have $\mathfrak{S}_{427}(1)=\big\{\mathcal{O}v_1^{r_1}:r_1\geq1\big\}\bigcup\big\{\mathfrak{U}_2v_2^{r_2}:r_2\geq7\,\,\text{and}\,\,r\neq8,9,10,11,12,13,15,16,18,$ $19,20,22,23,25,26,27,29,30,32,33,36,37,39,40,43,44,46,47,50,53,54,57,60,$\newline
$64,67,71,74,81,88\big\}$ and $g_{427}(1)=4$.
\end{proof}

\section{Computations in the case of class number 3}\label{Sec:Cl3}
In this section, we discuss the imaginary quadratic fields $E$ with class number $3$.
Now, assume that $\mathfrak{U}_2$ and $\mathfrak{U}_3$ are representatives of the two non-principle ideal classes, and let $L_2=\mathfrak{U}_2v_2$ and $L_3=\mathfrak{U}_3v_3$ be associated unimodular lattices as shown in Table 2, Appendix 2.
As we observed in the introduction, each positive definite integral unary Hermitian lattice is in the isometry class of the lattice $L_j^{r_j}=\mathfrak{U}_jv_j^{r_j}$, with $r_j$ a positive integer and $v_j^{r_j}$ a vector such that $h\big(v_j^{r_j}\big)=r_jh(v_j)$, for $j=2,3$.

Before to proceed to our discussions, we first observe that $r_2$ and $r_3$ take the same set of values for which $L_2^{r_2}$ and $L_3^{r_3}$ are in $\mathfrak{S}_d(1)$, and there is a transformation between their representations by $I_m$.
Therefore, in the proof of Theorem \ref{Thm2}, one only needs to determine the values of, say, $r_2$ in each individual case.

\begin{lem}\label{Lem4}
Let $E=\mathbb{Q}\big(\sqrt{-d}\big)$ be an imaginary quadratic field with class number $3$, and let $\mathfrak{U}_2$ and $\mathfrak{U}_3$ be representatives of the two non-principle ideal classes.
Then, $r_2,r_3$ assume the same set of values for which $L_2^{r_2},L_3^{r_3}\in\mathfrak{S}_d(1)$ and there is a transformation between their representations by $I_m$.
\end{lem}

\begin{proof}
Note that $\mathfrak{U}_2=\big(k,\frac{n-1}{2}+\omega\big)$ and $\mathfrak{U}_3=\big(k,-\frac{n+1}{2}+\omega\big)$ in Table 2, where $n$ is the smallest positive odd integer such that $-d\equiv n^2~(\mathrm{mod}~k)$.
Recall here $d\equiv 3~(\mathrm{mod}~4)$ throughout this section.
We have $k|\big(a_\ell+\frac{n+1}{2}b_\ell\big)$ and $k|\big(\frac{n-1}{2}a_\ell-\frac{1+d}{4}b_\ell\big)$ for $\mathfrak{U}_2$, and $k|\big(a_\ell-\frac{n-1}{2}b_\ell\big)$ and $k|\big(\frac{n+1}{2}a_\ell+\frac{1+d}{4}b_\ell\big)$ for $\mathfrak{U}_3$, using Lemma \ref{Lem1}.
It is easy to verify $k|\big(a_\ell+\frac{n+1}{2}b_\ell\big),k|\big(a_\ell-\frac{n-1}{2}b_\ell\big)$ lead to $k|\big(\frac{n-1}{2}a_\ell-\frac{1+d}{4}b_\ell\big),k|\big(\frac{n+1}{2}a_\ell+\frac{1+d}{4}b_\ell\big)$, resp-ectively.
So, when calculating the values of $r_2,r_3$ for which $L_2^{r_2},L_3^{r_3}\in\mathfrak{S}_d(1)$, we only need conditions $k|\big(a_\ell+\frac{n+1}{2}b_\ell\big)$ for $\mathfrak{U}_2$ and $k|\big(a_\ell-\frac{n-1}{2}b_\ell\big)$ for $\mathfrak{U}_3$.

Next, we show for each pair $a_\ell,b_\ell$ with $k|\big(a_\ell+\frac{n+1}{2}b_\ell\big)$, there is a pair $a'_\ell,b'_\ell$ with $k|\big(a'_\ell-\frac{n-1}{2}b'_\ell\big)$ and $P_d\big(a'_\ell,b'_\ell\big)=P_d(a_\ell,b_\ell)$, and vice versa.
First, recall
\begin{equation*}
P_d(a,b)=\frac{a^2}{k^2}+\frac{ab}{k^2}+\frac{(1+d)b^2}{4k^2}.
\end{equation*}
Now, set $a'_\ell=a_\ell+b_\ell$ and $b'_\ell=-b_\ell$ for each pair of integers $a_\ell,b_\ell$ with $k|\big(a_\ell+\frac{n+1}{2}b_\ell\big)$ to see
\begin{equation*}
a'_\ell-\frac{n-1}{2}b'_\ell=a_\ell+b_\ell+\frac{n-1}{2}b_\ell=a_\ell+\frac{n+1}{2}b_\ell,
\end{equation*}
so that $k|\big(a'_\ell-\frac{n-1}{2}b'_\ell\big)$; further, one has
\begin{equation*}
\begin{aligned}
P_d\big(a'_\ell,b'_\ell\big)&=\frac{(a_\ell+b_\ell)^2}{k^2}+\frac{(a_\ell+b_\ell)(-b_\ell)}{k^2}+\frac{(1+d)(-b_\ell)^2}{4k^2}\\
&=\frac{a_\ell^2}{k^2}+\frac{a_\ell b_\ell}{k^2}+\frac{(1+d)b_\ell^2}{4k^2}=P_d(a_\ell,b_\ell).
\end{aligned}
\end{equation*}
Conversely, choose $a_\ell=a'_\ell+b'_\ell$ and $b_\ell=-b'_\ell$ for each pair of integers $a'_\ell,b'_\ell$ with $k|\big(a'_\ell-\frac{n-1}{2}b'_\ell\big)$ to see $k|\big(a_\ell+\frac{n+1}{2}b_\ell\big)$ and $P_d(a_\ell,b_\ell)=P_d\big(a'_\ell,b'_\ell\big)$.
Thus, $r_2$ and $r_3$ assume the same set of values for which $L_2^{r_2}$ and $L_3^{r_3}$ are in $\mathfrak{S}_d(1)$.
\end{proof}

\begin{thm}\label{Thm2}
Assume $E=\mathbb{Q}\big(\sqrt{-d}\big)$ is an imaginary quadratic field with class number $3$, and $\mathcal{O}$ is its ring of integers.
Then, the set $\mathfrak{S}_d(1)$ and the value of $g_d(1)$ are listed in the table below
\begin{center}
\setlength{\arrayrulewidth}{0.3mm}
\setlength{\tabcolsep}{3pt}
\renewcommand{\arraystretch}{1.0}
\begin{tabular}{|m{1.68cm}|m{8.96cm}|m{0.96cm}|}
\hline\vskip2pt $\mathbb{Q}\big(\sqrt{-d}\big)$   & \vskip2pt $\mathfrak{S}_d(1)$; $j=2,3$                                          & \vskip2pt $g_d(1)$               \\[0.6ex]

\hline\vskip2pt $\mathbb{Q}\big(\sqrt{-23}\big)$  & \vskip2pt $L_1^{r_1}=\mathcal{O}v_1^{r_1}$ with $h\big(v_1^{r_1}\big)=r_1\geq1$ & \vskip2pt $3$                    \\[0.8ex]
                                                  &       $L_j^{r_j}=\mathfrak{U}_jv_j^{r_j}$ with $h\big(v_j^{r_j}\big)=r_j/2$ for $r_j\geq2$                       & \\[0.6ex]

\hline\vskip2pt $\mathbb{Q}\big(\sqrt{-31}\big)$  & \vskip2pt $L_1^{r_1}=\mathcal{O}v_1^{r_1}$ with $h\big(v_1^{r_1}\big)=r_1\geq1$ & \vskip2pt $4$                    \\[0.8ex]
                                                  &       $L_j^{r_j}=\mathfrak{U}_jv_j^{r_j}$ with $h\big(v_j^{r_j}\big)=r_j/2$ for $r_j\geq2$ and $r_j\neq3$        & \\[0.6ex]

\hline\vskip2pt $\mathbb{Q}\big(\sqrt{-59}\big)$  & \vskip2pt $L_1^{r_1}=\mathcal{O}v_1^{r_1}$ with $h\big(v_1^{r_1}\big)=r_1\geq1$ & \vskip2pt $4$                    \\[0.8ex]
                                                  &       $L_j^{r_j}=\mathfrak{U}_jv_j^{r_j}$ with $h\big(v_j^{r_j}\big)=r_j/3$ for $r_j\geq3$ and $r_j\neq4$        & \\[0.6ex]

\hline\vskip2pt $\mathbb{Q}\big(\sqrt{-83}\big)$  & \vskip2pt $L_1^{r_1}=\mathcal{O}v_1^{r_1}$ with $h\big(v_1^{r_1}\big)=r_1\geq1$ & \vskip2pt $4$                    \\[0.8ex]
                                                  &       $L_j^{r_j}=\mathfrak{U}_jv_j^{r_j}$ with $h\big(v_j^{r_j}\big)=r_j/3$ for $r_j\geq3$ and $r_j\neq4,5,8$    & \\[0.6ex]

\hline\vskip2pt $\mathbb{Q}\big(\sqrt{-107}\big)$ & \vskip2pt $L_1^{r_1}=\mathcal{O}v_1^{r_1}$ with $h\big(v_1^{r_1}\big)=r_1\geq1$ & \vskip2pt $4$                    \\[0.8ex]
                                                  &       $L_j^{r_j}=\mathfrak{U}_jv_j^{r_j}$ with $h\big(v_j^{r_j}\big)=r_j/3$ for $r_j\geq3$ and $r_j\neq4,5,7,8$, & \\[0.8ex]
                                                  &       $10$                                                                                                       & \\[0.6ex]

\hline\vskip2pt $\mathbb{Q}\big(\sqrt{-139}\big)$ & \vskip2pt $L_1^{r_1}=\mathcal{O}v_1^{r_1}$ with $h\big(v_1^{r_1}\big)=r_1\geq1$ & \vskip2pt $4$                    \\[0.8ex]
                                                  &       $L_j^{r_j}=\mathfrak{U}_jv_j^{r_j}$ with $h\big(v_j^{r_j}\big)=r_j/5$ for $r_j\geq5$ and $r_j\neq6,8,9$    & \\[0.6ex]

\hline\vskip2pt $\mathbb{Q}\big(\sqrt{-211}\big)$ & \vskip2pt $L_1^{r_1}=\mathcal{O}v_1^{r_1}$ with $h\big(v_1^{r_1}\big)=r_1\geq1$ & \vskip2pt $4$                    \\[0.8ex]
                                                  &       $L_j^{r_j}=\mathfrak{U}_jv_j^{r_j}$ with $h\big(v_j^{r_j}\big)=r_j/5$ for $r_j\geq5$ and $r_j\neq6,7,8,9,$ & \\[0.8ex]
                                                  &       $12,14,17$                                                                                                 & \\[0.6ex]

\hline\vskip2pt $\mathbb{Q}\big(\sqrt{-283}\big)$ & \vskip2pt $L_1^{r_1}=\mathcal{O}v_1^{r_1}$ with $h\big(v_1^{r_1}\big)=r_1\geq1$ & \vskip2pt $4$                    \\[0.8ex]
                                                  &       $L_j^{r_j}=\mathfrak{U}_jv_j^{r_j}$ with $h\big(v_j^{r_j}\big)=r_j/7$ for $r_j\geq7$ and $r_j\neq8,9,10,$  & \\[0.8ex]
                                                  &       $12,15,16,17,19$                                                                                           & \\[0.6ex]

\hline\vskip2pt $\mathbb{Q}\big(\sqrt{-307}\big)$ & \vskip2pt $L_1^{r_1}=\mathcal{O}v_1^{r_1}$ with $h\big(v_1^{r_1}\big)=r_1\geq1$ & \vskip2pt $4$                    \\[0.8ex]
                                                  &       $L_j^{r_j}=\mathfrak{U}_jv_j^{r_j}$ with $h\big(v_j^{r_j}\big)=r_j/7$ for $r_j\geq7$ and $r_j\neq8,9,10,$  & \\[0.8ex]
                                                  &       $12,13,15,16,20,23,27$                                                                                     & \\[0.6ex]

\hline\vskip2pt $\mathbb{Q}\big(\sqrt{-331}\big)$ & \vskip2pt $L_1^{r_1}=\mathcal{O}v_1^{r_1}$ with $h\big(v_1^{r_1}\big)=r_1\geq1$ & \vskip2pt $4$                    \\[0.8ex]
                                                  &       $L_j^{r_j}=\mathfrak{U}_jv_j^{r_j}$ with $h\big(v_j^{r_j}\big)=r_j/5$ for $r_j\geq5$ and $r_j\neq6,7,8,9,$ & \\[0.8ex]
                                                  &       $11,12,13,14,16,18,21,23,26,28,33$                                                                         & \\[0.6ex]

\hline\vskip2pt $\mathbb{Q}\big(\sqrt{-379}\big)$ & \vskip2pt $L_1^{r_1}=\mathcal{O}v_1^{r_1}$ with $h\big(v_1^{r_1}\big)=r_1\geq1$ & \vskip2pt $4$                    \\[0.8ex]
                                                  &       $L_j^{r_j}=\mathfrak{U}_jv_j^{r_j}$ with $h\big(v_j^{r_j}\big)=r_j/5$ for $r_j\geq5$ and $r_j\neq6,7,8,9,$ & \\[0.8ex]
                                                  &       $11,12,13,14,16,17,18,21,22,26,27,31,32,36$                                                                & \\[0.6ex]
\hline
\end{tabular}
\end{center}

\newpage\noindent{\color{blue}table continued}

\begin{center}
\setlength{\arrayrulewidth}{0.3mm}
\setlength{\tabcolsep}{3pt}
\renewcommand{\arraystretch}{1.0}
\begin{tabular}{|m{1.68cm}|m{8.96cm}|m{0.96cm}|}
\hline\vskip2pt $\mathbb{Q}\big(\sqrt{-499}\big)$ & \vskip2pt $L_1^{r_1}=\mathcal{O}v_1^{r_1}$ with $h\big(v_1^{r_1}\big)=r_1\geq1$ & \vskip2pt $4$                    \\[0.8ex]
                                                  &       $L_j^{r_j}=\mathfrak{U}_jv_j^{r_j}$ with $h\big(v_j^{r_j}\big)=r_j/5$ for $r_j\geq5$ and $r_j\neq6,7,8,9,$ & \\[0.8ex]
                                                  &       $11,12,13,14,16,17,18,19,21,22,23,24,26,27,28,32,33,$                                                      & \\[0.8ex]
                                                  &       $37,38,42$                                                                                                 & \\[0.6ex]

\hline\vskip2pt $\mathbb{Q}\big(\sqrt{-547}\big)$ & \vskip2pt $L_1^{r_1}=\mathcal{O}v_1^{r_1}$ with $h\big(v_1^{r_1}\big)=r_1\geq1$ & \vskip2pt $4$                    \\[0.8ex]
                                                  &       $L_j^{r_j}=\mathfrak{U}_jv_j^{r_j}$ with $h\big(v_j^{r_j}\big)=r_j/11$ for $r_j\geq11$ and $r_j\neq12,14,$ & \\[0.8ex]
                                                  &       $15,16,17,18,20,21,23,25,27,28,31,34,36$                                                                   & \\[0.6ex]

\hline\vskip2pt $\mathbb{Q}\big(\sqrt{-643}\big)$ & \vskip2pt $L_1^{r_1}=\mathcal{O}v_1^{r_1}$ with $h\big(v_1^{r_1}\big)=r_1\geq1$ & \vskip2pt $4$                    \\[0.8ex]
                                                  &       $L_j^{r_j}=\mathfrak{U}_jv_j^{r_j}$ with $h\big(v_j^{r_j}\big)=r_j/7$ for $r_j\geq7$ and $r_j\neq8,9,10,$  & \\[0.8ex]
                                                  &       $11,12,13,15,16,17,18,19,20,22,24,25,26,27,32,33,34,$                                                      & \\[0.8ex]
                                                  &       $39,40,41,47,48,55$                                                                                        & \\[0.6ex]

\hline\vskip2pt $\mathbb{Q}\big(\sqrt{-883}\big)$ & \vskip2pt $L_1^{r_1}=\mathcal{O}v_1^{r_1}$ with $h\big(v_1^{r_1}\big)=r_1\geq1$ & \vskip2pt $4$                    \\[0.8ex]
                                                  &       $L_j^{r_j}=\mathfrak{U}_jv_j^{r_j}$ with $h\big(v_j^{r_j}\big)=r_j/13$ for $r_j\geq13$ and $r_j\neq14,15,$ & \\[0.8ex]
                                                  &       $16,18,19,20,21,22,23,24,25,27,28,32,33,35,36,37,38,$                                                      & \\[0.8ex]
                                                  &       $40,41,45,49,50,53,54,66$                                                                                  & \\[0.6ex]

\hline\vskip2pt $\mathbb{Q}\big(\sqrt{-907}\big)$ & \vskip2pt $L_1^{r_1}=\mathcal{O}v_1^{r_1}$ with $h\big(v_1^{r_1}\big)=r_1\geq1$ & \vskip2pt $5$                    \\[0.8ex]
                                                  &       $L_j^{r_j}=\mathfrak{U}_jv_j^{r_j}$ with $h\big(v_j^{r_j}\big)=r_j/13$ for $r_j\geq13$ and $r_j\neq14,15,$ & \\[0.8ex]
                                                  &       $16,17,18,20,21,22,24,25,27,28,29,30,31,33,34,35,37,$                                                      & \\[0.8ex]
                                                  &       $40,43,44,47,48,50,56,63$                                                                                  & \\[0.6ex]
\hline
\end{tabular}
\end{center}
\end{thm}

\begin{proof}
Notice again by Lemmata \ref{Lem2} and \ref{Lem3}, there is the smallest positive integer $m_d$ such that $L_1^{r_1}\to I_{m_d}$ for all positive integers $r_1$ and $L_j^{r_j}\to I_{m_d}$ for all positive integers $r_j\equiv0~(\mathrm{mod}~k_j)$ with $j=2,3$; so, these lattices are in $\mathfrak{S}_d(1)$ and $g_d(1)\geq m_d$.
Thus, it suffices to consider only the cases where $r_j\not\equiv0~(\mathrm{mod}~k_j)$ for $j=2,3$.
We will show that $g_d(1)=m_d$ for all the cases except for $d=907$ where $g_d(1)=5$.
By Lemma \ref{Lem4}, we only need to determine the value of $r_2$ in each individual case.

\vskip8pt\noindent{\bf Case 1.} $E=\mathbb{Q}\big(\sqrt{-23}\big)$ with $\mathcal{O}=\mathbb{Z}+\mathbb{Z}\omega$ for $\omega=\frac{1+\sqrt{-23}}{2}$.

By Lemma \ref{Lem2}, one has $m_{23}=3$.
Moreover, in this case, through Table 2, we have $\mathfrak{U}_2=\mathcal{O}2+\mathcal{O}\omega$ and $h(v_2)=1/2$.
So, it follows from Lemma \ref{Lem1} that
\begin{equation*}
\omega\gamma_\ell=\omega(a_\ell+b_\ell\omega)=-6b_\ell+(a_\ell+b_\ell)\omega
\end{equation*}
for each $1\leq\ell\leq m$, which leads to $2|(a_\ell+b_\ell)$.
Therefore, one has
\begin{equation*}
h\big(v_2^{r_2}\big)=\frac{r_2}{2}=\sum_{\ell=1}^mN\Big(\frac{\gamma_\ell}{2}\Big)
=\sum_{\ell=1}^m\bigg(\frac{a_\ell^2}{4}+\frac{a_\ell b_\ell}{4}+\frac{6b_\ell^2}{4}\bigg)=:\sum_{\ell=1}^mP_{23}(a_\ell,b_\ell),
\end{equation*}
where $a_\ell,b_\ell$ are either both odd integers or are both even integers.

When $r_2$ is a positive odd integer, $r_2/2$ is a positive half integer; so, $r_2$ is at least $3$.
We choose $a_1=2\tilde{a}_1+1,b_1=-1$ and $a_\ell=-2-b_\ell$ for $2\leq\ell\leq3$.
Then,
\begin{equation*}
\sum_{\ell=1}^3P_{23}(a_\ell,b_\ell)=\frac{3}{2}+2+\frac{1}{2}\big(2\tilde{a}_1^2+\tilde{a}_1+3b_2^2+b_2+3b_3^2+b_3\big).
\end{equation*}
By Sun \cite[Theorem 1.2]{zS17}, we know $2\tilde{a}_1^2+\tilde{a}_1+3b_2^2+b_2+3b_3^2+b_3$ is universal.
Also, notice $3/2=P_{23}(1,-1)$ and $5/2=P_{23}(1,-1)+P_{23}(2,0)$.
Thus, $r_2/2$ is represented by $\sum_{\ell=1}^3P_{23}(a_\ell,b_\ell)$ for all odd integers $r_2\geq3$.

Therefore, by Lemma \ref{Lem4}, we have $\mathfrak{S}_{23}(1)=\big\{\mathcal{O}v_1^{r_1}:r_1\geq1\big\}\bigcup\big[\bigcup^3_{j=2}\big\{\mathfrak{U}_jv_j^{r_j}:r_j\geq2\big\}\big]$ and $g_{23}(1)=3$.

\vskip8pt\noindent{\bf Case 2.} $E=\mathbb{Q}\big(\sqrt{-31}\big)$ with $\mathcal{O}=\mathbb{Z}+\mathbb{Z}\omega$ for $\omega=\frac{1+\sqrt{-31}}{2}$.

By Lemma \ref{Lem2}, one has $m_{31}=4$.
Moreover, in this case, through Table 2, we have $\mathfrak{U}_2=\mathcal{O}2+\mathcal{O}\omega$ and $h(v_2)=1/2$.
So, it follows from Lemma \ref{Lem1} that
\begin{equation*}
\omega\gamma_\ell=\omega(a_\ell+b_\ell\omega)=-8b_\ell+(a_\ell+b_\ell)\omega
\end{equation*}
for each $1\leq\ell\leq m$, which leads to $2|(a_\ell+b_\ell)$.
Therefore, one has
\begin{equation*}
h\big(v_2^{r_2}\big)=\frac{r_2}{2}=\sum_{\ell=1}^mN\Big(\frac{\gamma_\ell}{2}\Big)
=\sum_{\ell=1}^m\bigg(\frac{a_\ell^2}{4}+\frac{a_\ell b_\ell}{4}+\frac{8b_\ell^2}{4}\bigg)=:\sum_{\ell=1}^mP_{31}(a_\ell,b_\ell),
\end{equation*}
where $a_\ell,b_\ell$ are either both odd integers or are both even integers.

When $r_2$ is a positive odd integer, $r_2/2$ is a positive half integer; so, $r_2$ is at least $5$.
Note that $P_{31}(2\tilde{a}_2,2)+P_{31}(2\tilde{a}_3,0)+P_{31}(2\tilde{a}_4,0)=8+2T_{\tilde{a}_2}+\tilde{a}_3^2+\tilde{a}_4^2$ represents all positive integers $r\geq8$, while $P_{31}(2\tilde{a}_2,0)+P_{31}(2\tilde{a}_3,0)+P_{31}(2\tilde{a}_4,0)=\tilde{a}_2^2+\tilde{a}_3^2+\tilde{a}_4^2$ represents all positive integers $r\neq4^a(8b+7)$ for nonnegative integers $a,b$; moreover, note $7=P_{31}(1,1)+P_{31}(2,0)+P_{31}(3,-1)$.
Combining these cases together yields that $\sum_{\ell=2}^4P_{31}(a_\ell,b_\ell)$ represents all positive integers.
In addition, notice $5/2=P_{31}(1,1)$ is the smallest positive half integer that can be represented by $P_{31}(a_1,b_1)$.
Thus, $r_2/2$ is represented by $\sum_{\ell=1}^4P_{31}(a_\ell,b_\ell)$ for all odd integers $r_2\geq5$.

Therefore, by Lemma \ref{Lem4}, we have $\mathfrak{S}_{31}(1)=\big\{\mathcal{O}v_1^{r_1}:r_1\geq1\big\}\bigcup\big[\bigcup^3_{j=2}\big\{\mathfrak{U}_jv_j^{r_j}:r_j\geq2\,\,\text{and}\,\,r_j\neq3\big\}\big]$ and $g_{31}(1)=4$.

\vskip8pt\noindent{\bf Case 3.} $E=\mathbb{Q}\big(\sqrt{-59}\big)$ with $\mathcal{O}=\mathbb{Z}+\mathbb{Z}\omega$ for $\omega=\frac{1+\sqrt{-59}}{2}$.

By Lemma \ref{Lem2}, one has $m_{59}=4$.
Moreover, in this case, through Table 2, we have $\mathfrak{U}_2=\mathcal{O}3+\mathcal{O}\omega$ and $h(v_2)=1/3$.
So, it follows from Lemma \ref{Lem1} that
\begin{equation*}
\omega\gamma_\ell=\omega(a_\ell+b_\ell\omega)=-15b_\ell+(a_\ell+b_\ell)\omega
\end{equation*}
for each $1\leq\ell\leq m$, which leads to $3|(a_\ell+b_\ell)$.
Therefore, one has
\begin{equation*}
h\big(v_2^{r_2}\big)=\frac{r_2}{3}=\sum_{\ell=1}^mN\Big(\frac{\gamma_\ell}{3}\Big)
=\sum_{\ell=1}^m\bigg(\frac{a_\ell^2}{9}+\frac{a_\ell b_\ell}{9}+\frac{15b_\ell^2}{9}\bigg)=:\sum_{\ell=1}^mP_{59}(a_\ell,b_\ell),
\end{equation*}
where $a_\ell$ and $b_\ell$ are integers satisfying $3|(a_\ell+b_\ell)$.

Notice $P_{59}(3\tilde{a}_2,3)+P_{59}(3\tilde{a}_3,0)+P_{59}(3\tilde{a}_4,0)=15+2T_{\tilde{a}_2}+\tilde{a}_3^2+\tilde{a}_4^2$ represents all positive integers $r\geq15$, while $P_{59}(3\tilde{a}_2,0)+P_{59}(3\tilde{a}_3,0)+P_{59}(3\tilde{a}_4,0)=\tilde{a}_2^2+\tilde{a}_3^2+\tilde{a}_4^2$ represents all positive integers $r\neq4^a(8b+7)$ for nonnegative integers $a,b$; moreover, note $7=P_{59}(4,-1)+P_{59}(6,0)$.
Merging these cases together yields $\sum_{\ell=2}^4P_{59}(a_\ell,b_\ell)$ represents all positive integers.

For $1\leq\delta\leq2$, let $r(\delta)$ be the smallest positive integer such that $r(\delta)\equiv\delta~(\mathrm{mod}~3)$ and $r(\delta)/3$ is represented by $P_{59}(a_1,b_1)$. Then, as analyzed above, $r_2/3$ is represented by $\sum_{\ell=1}^4P_{59}(a_\ell,b_\ell)$ for every $r_2\geq r(\delta)$ and $r_2\equiv\delta~(\mathrm{mod}~3)$.
It is enough to check the representation of $r_2/3$, with $r_2<r(\delta)$ and $r_2\equiv\delta~(\mathrm{mod}~3)$, by $\sum_{\ell=1}^mP_{59}(a_\ell,b_\ell)$, and further by $\sum_{\ell=1}^4P_{59}(a_\ell,b_\ell)$, below
{\small\begin{center}
\setlength{\arrayrulewidth}{0.2mm}
\setlength{\tabcolsep}{2pt}
\renewcommand{\arraystretch}{0.81}
\begin{tabular}{|m{1.21cm}|m{4.03cm}|m{1.51cm}|}
\hline\vskip2pt $r(\delta)$ & \vskip2pt $r_2<r(\delta)$ and $r_2\equiv\delta~(\mathrm{mod}~3)$ & \vskip2pt $r_2=r(\delta)$ \\[0.5ex]

\hline\vskip2pt $r(1)=7$    & \vskip2pt None can be represented.                               & \vskip2pt $P_{59}(2,1)$   \\[0.5ex]

\hline\vskip2pt $r(2)=5$    & \vskip2pt None can be represented.                               & \vskip2pt $P_{59}(1,-1)$  \\[0.5ex]
\hline
\end{tabular}
\end{center}}

Therefore, by Lemma \ref{Lem4}, we have $\mathfrak{S}_{59}(1)=\big\{\mathcal{O}v_1^{r_1}:r_1\geq1\big\}\bigcup\big[\bigcup^3_{j=2}\big\{\mathfrak{U}_jv_j^{r_j}:r_j\geq3\,\,\text{and}\,\,r_j\neq4\big\}\big]$ and $g_{59}(1)=4$.

It is noteworthy that the arguments in the remaining cases are very similar to those already applied in {\bf Case 3}.
As a matter of fact, one has
\begin{equation*}
h\big(v_2^{r_2}\big)=\frac{r_2}{k}=\sum_{\ell=1}^mN\Big(\frac{\gamma_\ell}{k}\Big)
=\sum_{\ell=1}^m\bigg(\frac{a_\ell^2}{k^2}+\frac{a_\ell b_\ell}{k^2}+\frac{(1+d)b_\ell^2}{4k^2}\bigg)=:\sum_{\ell=1}^mP_d(a_\ell,b_\ell),
\end{equation*}
where $a_\ell,b_\ell$ are integers satisfying $k|\big(a_\ell+\frac{n+1}{2}b_\ell\big)$ by Lemma \ref{Lem4} and $n$ is the smallest positive odd integer such that $-d\equiv n^2~(\mathrm{mod}~k)$.

Note $P_d(k\tilde{a}_2,k)+P_d(k\tilde{a}_3,0)+P_d(k\tilde{a}_4,0)=\frac{d+1}{4}+2T_{\tilde{a}_2}+\tilde{a}_3^2+\tilde{a}_4^2$ represents all positive integers $r\geq\frac{d+1}{4}$, while $P_d(k\tilde{a}_2,0)+P_d(k\tilde{a}_3,0)+P_d(k\tilde{a}_4,0)=\tilde{a}_2^2+\tilde{a}_3^2+\tilde{a}_4^2$ represents all positive integers $r\neq4^a(8b+7)$ for nonnegative integers $a,b$.
One can corroborate that $\sum_{\ell=2}^4P_d(a_\ell,b_\ell)$ represents all the remaining positive integers for all imaginary quadratic fields $\mathbb{Q}\big(\sqrt{-d}\big)$ with class number $3$ case by case \footnote{Interested reader can find associated computational details at
\vskip2pt\hspace{2.6mm}{\url{https://sites.google.com/view/jingbos-number-theory/home?authuser=1}}}.
Consequently, $\sum_{\ell=2}^4P_d(a_\ell,b_\ell)$ is an universal form that represents all positive integers.

For $1\leq\delta\leq k$, let $r(\delta)$ be the smallest positive integer such that $r(\delta)\equiv\delta~(\mathrm{mod}~k)$ and $r(\delta)/k$ is represented by $P_d(a_1,b_1)$. Then, as analyzed above, $r_2/k$ is represented by $\sum_{\ell=1}^4P_d(a_\ell,b_\ell)$ for every $r_2\geq r(\delta)$ and $r_2\equiv\delta~(\mathrm{mod}~k)$.
It is enough to check the representation of $r_2/k$, with $r_2<r(\delta)$ and $r_2\equiv\delta~(\mathrm{mod}~k)$, by $\sum_{\ell=1}^mP_d(a_\ell,b_\ell)$, and then determine the explicit form of $\mathfrak{S}_d(1)$ and the exact value of $g_d(1)$.

In order to avoid tedious repetitions, in the sequel, we shall only provide the table which includes the values of $r(\delta)$ and $r_2$, with $r_2<r(\delta)$ and $r_2\equiv\delta~(\mathrm{mod}~k)$, such that $r_2/k$ is represented by some sum $\sum_{\ell=1}^mP_d(a_\ell,b_\ell)$ in each remaining case.

\vskip8pt\noindent{\bf Case 4.} $E=\mathbb{Q}\big(\sqrt{-83}\big)$ with $\mathcal{O}=\mathbb{Z}+\mathbb{Z}\omega$ for $\omega=\frac{1+\sqrt{-83}}{2}$.

In this case, one has $m_{83}=4$, $k=3$, $n=1$ and
\begin{equation*}
h\big(v_2^{r_2}\big)=\frac{r_2}{3}=\sum_{\ell=1}^mN\Big(\frac{\gamma_\ell}{3}\Big)
=\sum_{\ell=1}^m\bigg(\frac{a_\ell^2}{9}+\frac{a_\ell b_\ell}{9}+\frac{21b_\ell^2}{9}\bigg)=:\sum_{\ell=1}^mP_{83}(a_\ell,b_\ell),
\end{equation*}
where $a_\ell$ and $b_\ell$ are integers satisfying $3|(a_\ell+b_\ell)$.
The numbers $r$, with $r<21$ and $r=4^a(8b+7)$ for $a,b\in\mathbb{Z}_{\geq0}$, are $7,15$, whose detailed representations are available in the link at the footnote of {\bf Case 3}; others are given in the table below
{\small\begin{center}
\setlength{\arrayrulewidth}{0.2mm}
\setlength{\tabcolsep}{2pt}
\renewcommand{\arraystretch}{0.81}
\begin{tabular}{|m{1.36cm}|m{4.03cm}|m{1.51cm}|}
\hline\vskip2pt $r(\delta)$ & \vskip2pt $r_2<r(\delta)$ and $r_2\equiv\delta~(\mathrm{mod}~3)$ & \vskip2pt $r_2=r(\delta)$ \\[0.5ex]

\hline\vskip2pt $r(1)=7$    & \vskip2pt None can be represented.                               & \vskip2pt $P_{83}(1,-1)$  \\[0.5ex]

\hline\vskip2pt $r(2)=11$   & \vskip2pt None can be represented.                               & \vskip2pt $P_{83}(4,-1)$  \\[0.5ex]
\hline
\end{tabular}
\end{center}}

Therefore, by Lemma \ref{Lem4}, we have
$\mathfrak{S}_{83}(1)=\big\{\mathcal{O}v_1^{r_1}:r_1\geq1\big\}\bigcup\big[\bigcup^3_{j=2}\big\{\mathfrak{U}_jv_j^{r_j}:r_j\geq3\,\,\text{and}\,\,r_j\neq4,5,8\big\}\big]$ and $g_{83}(1)=4$.

\vskip8pt\noindent{\bf Case 5.} $E=\mathbb{Q}\big(\sqrt{-107}\big)$ with $\mathcal{O}=\mathbb{Z}+\mathbb{Z}\omega$ for $\omega=\frac{1+\sqrt{-107}}{2}$.

In this case, one has $m_{107}=4$, $k=3$, $n=1$ and
\begin{equation*}
h\big(v_2^{r_2}\big)=\frac{r_2}{3}=\sum_{\ell=1}^mN\Big(\frac{\gamma_\ell}{3}\Big)
=\sum_{\ell=1}^m\bigg(\frac{a_\ell^2}{9}+\frac{a_\ell b_\ell}{9}+\frac{27b_\ell^2}{9}\bigg)=:\sum_{\ell=1}^mP_{107}(a_\ell,b_\ell),
\end{equation*}
where $a_\ell$ and $b_\ell$ are integers satisfying $3|(a_\ell+b_\ell)$.
The numbers $r$, with $r<27$ and $r=4^a(8b+7)$ for $a,b\in\mathbb{Z}_{\geq0}$, are $7,15,23$, whose detailed representations are given in the link at the footnote of {\bf Case 3}; others are given in the table below
{\small\begin{center}
\setlength{\arrayrulewidth}{0.2mm}
\setlength{\tabcolsep}{2pt}
\renewcommand{\arraystretch}{0.81}
\begin{tabular}{|m{1.37cm}|m{4.03cm}|m{1.63cm}|}
\hline\vskip2pt $r(\delta)$ & \vskip2pt $r_2<r(\delta)$ and $r_2\equiv\delta~(\mathrm{mod}~3)$ & \vskip2pt $r_2=r(\delta)$ \\[0.5ex]

\hline\vskip2pt $r(1)=13$   & \vskip2pt None can be represented.                               & \vskip2pt $P_{107}(4,-1)$ \\[0.5ex]

\hline\vskip2pt $r(2)=11$   & \vskip2pt None can be represented.                               & \vskip2pt $P_{107}(2,1)$  \\[0.5ex]
\hline
\end{tabular}
\end{center}}

Therefore, by Lemma \ref{Lem4}, we have $\mathfrak{S}_{107}(1)=\big\{\mathcal{O}v_1^{r_1}:r_1\geq1\big\}\bigcup\big[\bigcup^3_{j=2}\big\{\mathfrak{U}_jv_j^{r_j}:r_j\geq3\,\,\text{and}\,\,r_j\neq4,5,7,8,10\big\}\big]$ and $g_{107}(1)=4$.

\vskip8pt\noindent{\bf Case 6.} $E=\mathbb{Q}\big(\sqrt{-139}\big)$ with $\mathcal{O}=\mathbb{Z}+\mathbb{Z}\omega$ for $\omega=\frac{1+\sqrt{-139}}{2}$.

In this case, one has $m_{139}=4$, $k=5$, $n=1$ and
\begin{equation*}
h\big(v_2^{r_2}\big)=\frac{r_2}{5}=\sum_{\ell=1}^mN\Big(\frac{\gamma_\ell}{5}\Big)
=\sum_{\ell=1}^m\bigg(\frac{a_\ell^2}{25}+\frac{a_\ell b_\ell}{25}+\frac{35b_\ell^2}{25}\bigg)=:\sum_{\ell=1}^mP_{139}(a_\ell,b_\ell),
\end{equation*}
where $a_\ell$ and $b_\ell$ are integers satisfying $5|(a_\ell+b_\ell)$.
The numbers $r$, with $r<35$ and $r=4^a(8b+7)$ for $a,b\in\mathbb{Z}_{\geq0}$, are $7,15,23,28,31$, whose detailed representations are given in the link mentioned earlier; others are given in the table below
{\small\begin{center}
\setlength{\arrayrulewidth}{0.2mm}
\setlength{\tabcolsep}{2pt}
\renewcommand{\arraystretch}{0.81}
\begin{tabular}{|m{1.37cm}|m{4.35cm}|m{1.81cm}|}
\hline\vskip2pt $r(\delta)$ & \vskip2pt $r_2<r(\delta)$ and $r_2\equiv\delta~(\mathrm{mod}~5)$ & \vskip2pt $r_2=r(\delta)$  \\[0.5ex]

\hline\vskip2pt $r(1)=11$   & \vskip2pt None can be represented.                               & \vskip2pt $P_{139}(4,1)$   \\[0.5ex]

\hline\vskip2pt $r(2)=7$    & \vskip2pt None can be represented.                               & \vskip2pt $P_{139}(1,-1)$  \\[0.5ex]

\hline\vskip2pt $r(3)=13$   & \vskip2pt None can be represented.                               & \vskip2pt $P_{139}(6,-1)$  \\[0.5ex]

\hline\vskip2pt $r(4)=29$   & \vskip2pt $14=2P_{139}(1,-1)$                                    & \vskip2pt $P_{139}(11,-1)$ \\[0.5ex]
                            &           $19=2P_{139}(1,-1)+P_{139}(5,0)$                       &                            \\[0.5ex]
                            &           $24=2P_{139}(1,-1)+2P_{139}(5,0)$                      &                            \\[0.5ex]
                            &           Others cannot be represented.                          &                            \\[0.5ex]
\hline
\end{tabular}
\end{center}}

Therefore, by Lemma \ref{Lem4}, we have $\mathfrak{S}_{139}(1)=\big\{\mathcal{O}v_1^{r_1}:r_1\geq1\big\}\bigcup\big[\bigcup^3_{j=2}\big\{\mathfrak{U}_jv_j^{r_j}:r_j\geq5\,\,\text{and}\,\,r_j\neq6,8,9\big\}\big]$ and $g_{139}(1)=4$.

\vskip8pt\noindent{\bf Case 7.} $E=\mathbb{Q}\big(\sqrt{-211}\big)$ with $\mathcal{O}=\mathbb{Z}+\mathbb{Z}\omega$ for $\omega=\frac{1+\sqrt{-211}}{2}$.

In this case, one has $m_{211}=4$, $k=5$, $n=3$ and
\begin{equation*}
h\big(v_2^{r_2}\big)=\frac{r_2}{5}=\sum_{\ell=1}^mN\Big(\frac{\gamma_\ell}{5}\Big)
=\sum_{\ell=1}^m\bigg(\frac{a_\ell^2}{25}+\frac{a_\ell b_\ell}{25}+\frac{53b_\ell^2}{25}\bigg)=:\sum_{\ell=1}^mP_{211}(a_\ell,b_\ell),
\end{equation*}
where $a_\ell$ and $b_\ell$ are integers satisfying $5|(a_\ell+2b_\ell)$.
The numbers $r$, with $r<53$ and $r=4^a(8b+7)$ for $a,b\in\mathbb{Z}_{\geq0}$, are $7,15,23,28,31,39,47$, with representation details given in the link mentioned earlier; others are given in the table below
{\small\begin{center}
\setlength{\arrayrulewidth}{0.2mm}
\setlength{\tabcolsep}{2pt}
\renewcommand{\arraystretch}{0.81}
\begin{tabular}{|m{1.37cm}|m{4.35cm}|m{1.81cm}|}
\hline\vskip2pt $r(\delta)$ & \vskip2pt $r_2<r(\delta)$ and $r_2\equiv\delta~(\mathrm{mod}~5)$ & \vskip2pt $r_2=r(\delta)$  \\[0.5ex]

\hline\vskip2pt $r(1)=11$   & \vskip2pt None can be represented.                               & \vskip2pt $P_{211}(2,-1)$  \\[0.5ex]

\hline\vskip2pt $r(2)=37$   & \vskip2pt $22=2P_{211}(2,-1)$                                    & \vskip2pt $P_{211}(12,-1)$ \\[0.5ex]
                            &           $27=2P_{211}(2,-1)+P_{211}(5,0)$                       &                            \\[0.5ex]
                            &           $32=2P_{211}(2,-1)+2P_{211}(5,0)$                      &                            \\[0.5ex]
                            &           Others cannot be represented.                          &                            \\[0.5ex]

\hline\vskip2pt $r(3)=13$   & \vskip2pt None can be represented.                               & \vskip2pt $P_{211}(3,1)$   \\[0.5ex]

\hline\vskip2pt $r(4)=19$   & \vskip2pt None can be represented.                               & \vskip2pt $P_{211}(7,-1)$  \\[0.5ex]
\hline
\end{tabular}
\end{center}}

Therefore, by Lemma \ref{Lem4}, we have $\mathfrak{S}_{211}(1)=\big\{\mathcal{O}v_1^{r_1}:r_1\geq1\big\}\bigcup\big[\bigcup^3_{j=2}\big\{\mathfrak{U}_jv_j^{r_j}:r_j\geq5\,\,\text{and}\,\,r_j\neq6,7,8,9,12,14,17\big\}\big]$ and $g_{211}(1)=4$.

\vskip8pt\noindent{\bf Case 8.} $E=\mathbb{Q}\big(\sqrt{-283}\big)$ with $\mathcal{O}=\mathbb{Z}+\mathbb{Z}\omega$ for $\omega=\frac{1+\sqrt{-283}}{2}$.

In this case, one has $m_{283}=4$, $k=7$, $n=5$ and
\begin{equation*}
h\big(v_2^{r_2}\big)=\frac{r_2}{7}=\sum_{\ell=1}^mN\Big(\frac{\gamma_\ell}{7}\Big)
=\sum_{\ell=1}^m\bigg(\frac{a_\ell^2}{49}+\frac{a_\ell b_\ell}{49}+\frac{71b_\ell^2}{49}\bigg)=:\sum_{\ell=1}^mP_{283}(a_\ell,b_\ell),
\end{equation*}
where $a_\ell$ and $b_\ell$ are integers satisfying $7|(a_\ell+3b_\ell)$.
The numbers $r$, with $r<71$ and $r=4^a(8b+7)$ for $a,b\in\mathbb{Z}_{\geq0}$, are $7,15,23,28,31,39,47,55,60,63$, whose detailed representations are given as before; others are given in the table below
{\small\begin{center}
\setlength{\arrayrulewidth}{0.2mm}
\setlength{\tabcolsep}{2pt}
\renewcommand{\arraystretch}{0.81}
\begin{tabular}{|m{1.37cm}|m{7.76cm}|m{1.81cm}|}
\hline\vskip2pt $r(\delta)$ & \vskip2pt $r_2<r(\delta)$ and $r_2\equiv\delta~(\mathrm{mod}~7)$      & \vskip2pt $r_2=r(\delta)$  \\[0.5ex]

\hline\vskip2pt $r(1)=29$   & \vskip2pt $22=2P_{283}(3,-1)$                                         & \vskip2pt $P_{283}(11,1)$  \\[0.5ex]
                            &           Others cannot be represented.                               &                            \\[0.5ex]

\hline\vskip2pt $r(2)=23$   & \vskip2pt None can be represented.                                    & \vskip2pt $P_{283}(10,-1)$ \\[0.5ex]

\hline\vskip2pt $r(3)=52$   & \vskip2pt $24=P_{283}(3,-1)+P_{283}(4,1)$                             & \vskip2pt $P_{283}(8,2)$   \\[0.5ex]
                            &           $31=P_{283}(3,-1)+P_{283}(4,1)+P_{283}(7,0)$                &                            \\[0.5ex]
                            &           $38=P_{283}(3,-1)+P_{283}(4,1)+2P_{283}(7,0)$               &                            \\[0.5ex]
                            &           $45=2P_{283}(3,-1)+P_{283}(10,-1)$                          &                            \\[0.5ex]
                            &           Others cannot be represented.                               &                            \\[0.5ex]

\hline\vskip2pt $r(4)=11$   & \vskip2pt None can be represented.                                    & \vskip2pt $P_{283}(3,-1)$  \\[0.5ex]

\hline\vskip2pt $r(5)=61$   & \vskip2pt $26=2P_{283}(4,1)$                                          & \vskip2pt $P_{283}(13,-2)$ \\[0.5ex]
                            &           $33=2P_{283}(4,1)+P_{283}(7,0)$                             &                            \\[0.5ex]
                            &           $40=2P_{283}(4,1)+2P_{283}(7,0)$                            &                            \\[0.5ex]
                            &           $47=P_{283}(3,-1)+P_{283}(4,1)+P_{283}(10,-1)$              &                            \\[0.5ex]
                            &           $54=P_{283}(3,-1)+P_{283}(4,1)+P_{283}(7,0)+P_{283}(10,-1)$ &                            \\[0.5ex]
                            &           Others cannot be represented.                               &                            \\[0.5ex]

\hline\vskip2pt $r(6)=13$   & \vskip2pt None can be represented.                                    & \vskip2pt $P_{283}(4,1)$   \\[0.5ex]
\hline
\end{tabular}
\end{center}}

Therefore, by Lemma \ref{Lem4}, we have $\mathfrak{S}_{283}(1)=\big\{\mathcal{O}v_1^{r_1}:r_1\geq1\big\}
\bigcup\big[\bigcup^3_{j=2}\big\{\mathfrak{U}_jv_j^{r_j}:r_j\geq7\,\,\text{and}\,\,r_j\neq8,9,10,12,15,16,17,19\big\}\big]$ and $g_{283}(1)=4$.

\vskip8pt\noindent{\bf Case 9.} $E=\mathbb{Q}\big(\sqrt{-307}\big)$ with $\mathcal{O}=\mathbb{Z}+\mathbb{Z}\omega$ for $\omega=\frac{1+\sqrt{-307}}{2}$.

In this case, one has $m_{307}=4$, $k=7$, $n=1$ and
\begin{equation*}
h\big(v_2^{r_2}\big)=\frac{r_2}{7}=\sum_{\ell=1}^mN\Big(\frac{\gamma_\ell}{7}\Big)
=\sum_{\ell=1}^m\bigg(\frac{a_\ell^2}{49}+\frac{a_\ell b_\ell}{49}+\frac{77b_\ell^2}{49}\bigg)=:\sum_{\ell=1}^mP_{307}(a_\ell,b_\ell),
\end{equation*}
where $a_\ell$ and $b_\ell$ are integers satisfying $7|(a_\ell+b_\ell)$.
The numbers $r$, with $r<77$ and $r=4^a(8b+7)$ for $a,b\in\mathbb{Z}_{\geq0}$, are $7,15,23,28,31,39,47,55,60,63,71$, with detailed representations available as before; others are given in the table below
{\small\begin{center}
\setlength{\arrayrulewidth}{0.2mm}
\setlength{\tabcolsep}{2pt}
\renewcommand{\arraystretch}{0.81}
\begin{tabular}{|m{1.37cm}|m{6.04cm}|m{1.81cm}|}
\hline\vskip2pt $r(\delta)$ & \vskip2pt $r_2<r(\delta)$ and $r_2\equiv\delta~(\mathrm{mod}~7)$ & \vskip2pt $r_2=r(\delta)$ \\[0.5ex]

\hline\vskip2pt $r(1)=71$   & \vskip2pt $22=2P_{307}(1,-1)$                                    & \vskip2pt $P_{307}(20,1)$  \\[0.5ex]
                            &           $29=2P_{307}(1,-1)+P_{307}(7,0)$                       &                            \\[0.5ex]
                            &           $36=2P_{307}(1,-1)+2P_{307}(7,0)$                      &                            \\[0.5ex]
                            &           $43=P_{307}(6,1)+P_{307}(7,0)+P_{307}(8,-1)$           &                            \\[0.5ex]
                            &           $50=P_{307}(6,1)+2P_{307}(7,0)+P_{307}(8,-1)$          &                            \\[0.5ex]
                            &           $57=2P_{307}(1,-1)+P_{307}(7,0)+P_{307}(14,0)$         &                            \\[0.5ex]
                            &           $64=P_{307}(1,-1)+P_{307}(9,-2)$                       &                            \\[0.5ex]
                            &           Others cannot be represented.                          &                            \\[0.5ex]

\hline\vskip2pt $r(2)=37$   & \vskip2pt $30=P_{307}(1,-1)+P_{307}(8,-1)$                       & \vskip2pt $P_{307}(13,1)$  \\[0.5ex]
                            &           Others cannot be represented.                          &                            \\[0.5ex]

\hline\vskip2pt $r(3)=17$   & \vskip2pt None can be represented.                               & \vskip2pt $P_{307}(6,1)$   \\[0.5ex]

\hline\vskip2pt $r(4)=11$   & \vskip2pt None can be represented.                               & \vskip2pt $P_{307}(1,-1)$  \\[0.5ex]

\hline\vskip2pt $r(5)=19$   & \vskip2pt None can be represented.                               & \vskip2pt $P_{307}(8,-1)$  \\[0.5ex]
\hline
\end{tabular}
\end{center}}

\newpage\noindent{\color{blue}table continued}

{\small\begin{center}
\setlength{\arrayrulewidth}{0.2mm}
\setlength{\tabcolsep}{2pt}
\renewcommand{\arraystretch}{0.81}
\begin{tabular}{|m{1.37cm}|m{6.06cm}|m{1.81cm}|}
\hline\vskip2pt $r(6)=41$   & \vskip2pt $34=2P_{307}(6,1)$                                     & \vskip2pt $P_{307}(15,-1)$ \\[0.5ex]
                            &           Others cannot be represented.                          &                            \\[0.5ex]
\hline
\end{tabular}
\end{center}}

Therefore, by Lemma \ref{Lem4}, we have $\mathfrak{S}_{307}(1)=\big\{\mathcal{O}v_1^{r_1}:r_1\geq1\big\}
\bigcup\big[\bigcup^3_{j=2}\big\{\mathfrak{U}_jv_j^{r_j}:r_j\geq7\,\,\text{and}\,\,r_j\neq8,9,10,12,13,15,16,20,23,27\big\}\big]$ and $g_{307}(1)=4$.

\vskip8pt\noindent{\bf Case 10.} $E=\mathbb{Q}\big(\sqrt{-331}\big)$ with $\mathcal{O}=\mathbb{Z}+\mathbb{Z}\omega$ for $\omega=\frac{1+\sqrt{-331}}{2}$.

In this case, one has $m_{331}=4$, $k=5$, $n=3$ and
\begin{equation*}
h\big(v_2^{r_2}\big)=\frac{r_2}{5}=\sum_{\ell=1}^mN\Big(\frac{\gamma_\ell}{5}\Big)
=\sum_{\ell=1}^m\bigg(\frac{a_\ell^2}{25}+\frac{a_\ell b_\ell}{25}+\frac{83b_\ell^2}{25}\bigg)=:\sum_{\ell=1}^mP_{331}(a_\ell,b_\ell),
\end{equation*}
where $a_\ell$ and $b_\ell$ are integers satisfying $5|(a_\ell+2b_\ell)$.
For $r<83$ and $r=4^a(8b+7)$ with $a,b\in\mathbb{Z}_{\geq0}$, one has $r=7,15,23,28,31,39,47,55,60,63,71,79$, whose detailed representations are given as before; others are given in the table below
{\small\begin{center}
\setlength{\arrayrulewidth}{0.2mm}
\setlength{\tabcolsep}{2pt}
\renewcommand{\arraystretch}{0.81}
\begin{tabular}{|m{1.37cm}|m{4.03cm}|m{1.81cm}|}
\hline\vskip2pt $r(\delta)$ & \vskip2pt $r_2<r(\delta)$ and $r_2\equiv\delta~(\mathrm{mod}~5)$ & \vskip2pt $r_2=r(\delta)$  \\[0.5ex]

\hline\vskip2pt $r(1)=31$   & \vskip2pt None can be represented.                               & \vskip2pt $P_{331}(8,1)$   \\[0.5ex]

\hline\vskip2pt $r(2)=17$   & \vskip2pt None can be represented.                               & \vskip2pt $P_{331}(2,-1)$  \\[0.5ex]

\hline\vskip2pt $r(3)=43$   & \vskip2pt $38=2P_{331}(3,1)$                                     & \vskip2pt $P_{331}(12,-1)$ \\[0.5ex]
                            &           Others cannot be represented.                          &                            \\[0.5ex]

\hline\vskip2pt $r(4)=19$   & \vskip2pt None can be represented.                               & \vskip2pt $P_{331}(3,1)$   \\[0.5ex]
\hline
\end{tabular}
\end{center}}

So, by Lemma \ref{Lem4}, $\mathfrak{S}_{331}(1)=\big\{\mathcal{O}v_1^{r_1}:r_1\geq1\big\}
\bigcup\big[\bigcup^3_{j=2}\big\{\mathfrak{U}_jv_j^{r_j}:r_j\geq5\,\,\text{and}\,\,r_j\neq6,7,8,9,11,12,13,14,16,18,21,23,26,28,33\big\}\big]$ and $g_{331}(1)=4$.

\vskip8pt\noindent{\bf Case 11.} $E=\mathbb{Q}\big(\sqrt{-379}\big)$ with $\mathcal{O}=\mathbb{Z}+\mathbb{Z}\omega$ for $\omega=\frac{1+\sqrt{-379}}{2}$.

In this case, one has $m_{379}=4$, $k=5$, $n=1$ and
\begin{equation*}
h\big(v_2^{r_2}\big)=\frac{r_2}{5}=\sum_{\ell=1}^mN\Big(\frac{\gamma_\ell}{5}\Big)
=\sum_{\ell=1}^m\bigg(\frac{a_\ell^2}{25}+\frac{a_\ell b_\ell}{25}+\frac{95b_\ell^2}{25}\bigg)=:\sum_{\ell=1}^mP_{379}(a_\ell,b_\ell),
\end{equation*}
where $a_\ell$ and $b_\ell$ are integers satisfying $5|(a_\ell+b_\ell)$.
For $r<95$ and $r=4^a(8b+7)$ with $a,b\in\mathbb{Z}_{\geq0}$, $r=7,15,23,28,31,39,47,55,60,63,71,79,87,92$, whose detailed representations are given as before; others are given in the table below
{\small\begin{center}
\setlength{\arrayrulewidth}{0.2mm}
\setlength{\tabcolsep}{2pt}
\renewcommand{\arraystretch}{0.81}
\begin{tabular}{|m{1.37cm}|m{4.03cm}|m{1.81cm}|}
\hline\vskip2pt $r(\delta)$ & \vskip2pt $r_2<r(\delta)$ and $r_2\equiv\delta~(\mathrm{mod}~5)$ & \vskip2pt $r_2=r(\delta)$  \\[0.5ex]

\hline\vskip2pt $r(1)=41$   & \vskip2pt None can be represented.                               & \vskip2pt $P_{379}(11,-1)$ \\[0.5ex]

\hline\vskip2pt $r(2)=37$   & \vskip2pt None can be represented.                               & \vskip2pt $P_{379}(9,1)$   \\[0.5ex]

\hline\vskip2pt $r(3)=23$   & \vskip2pt None can be represented.                               & \vskip2pt $P_{379}(4,1)$   \\[0.5ex]

\hline\vskip2pt $r(4)=19$   & \vskip2pt None can be represented.                               & \vskip2pt $P_{379}(1,-1)$  \\[0.5ex]
\hline
\end{tabular}
\end{center}}

So, by Lemma \ref{Lem4}, $\mathfrak{S}_{379}(1)=\big\{\mathcal{O}v_1^{r_1}:r_1\geq1\big\}
\bigcup\big[\bigcup^3_{j=2}\big\{\mathfrak{U}_jv_j^{r_j}:r_j\geq5\,\,\text{and}\,\,r_j\neq6,7,8,9,11,12,13,14,16,17,18,21,22,26,27,31,32,36\big\}\big]$ and $g_{379}(1)=4$.

\vskip8pt\noindent{\bf Case 12.} $E=\mathbb{Q}\big(\sqrt{-499}\big)$ with $\mathcal{O}=\mathbb{Z}+\mathbb{Z}\omega$ for $\omega=\frac{1+\sqrt{-499}}{2}$.

In this case, one has $m_{499}=4$, $k=5$, $n=1$ and
\begin{equation*}
h\big(v_2^{r_2}\big)=\frac{r_2}{5}=\sum_{\ell=1}^mN\Big(\frac{\gamma_\ell}{5}\Big)
=\sum_{\ell=1}^m\bigg(\frac{a_\ell^2}{25}+\frac{a_\ell b_\ell}{25}+\frac{125b_\ell^2}{25}\bigg)=:\sum_{\ell=1}^mP_{499}(a_\ell,b_\ell),
\end{equation*}
where $a_\ell$ and $b_\ell$ are integers satisfying $5|(a_\ell+b_\ell)$.
For $r<125$ and $r=4^a(8b+7)$ with $a,b\in\mathbb{Z}_{\geq0}$, one has $r=7,15,23,28,31,39,47,55,60,63,71,79,87,92,95,103,$
$111,112,119,124$, whose detailed representations are given as before; others are given in the table below
{\small\begin{center}
\setlength{\arrayrulewidth}{0.2mm}
\setlength{\tabcolsep}{2pt}
\renewcommand{\arraystretch}{0.81}
\begin{tabular}{|m{1.37cm}|m{4.03cm}|m{1.81cm}|}
\hline\vskip2pt $r(\delta)$ & \vskip2pt $r_2<r(\delta)$ and $r_2\equiv\delta~(\mathrm{mod}~5)$ & \vskip2pt $r_2=r(\delta)$  \\[0.5ex]

\hline\vskip2pt $r(1)=31$   & \vskip2pt None can be represented.                               & \vskip2pt $P_{499}(6,-1)$  \\[0.5ex]

\hline\vskip2pt $r(2)=47$   & \vskip2pt None can be represented.                               & \vskip2pt $P_{499}(11,-1)$ \\[0.5ex]

\hline\vskip2pt $r(3)=43$   & \vskip2pt None can be represented.                               & \vskip2pt $P_{499}(9,1)$   \\[0.5ex]

\hline\vskip2pt $r(4)=29$   & \vskip2pt None can be represented.                               & \vskip2pt $P_{499}(4,1)$   \\[0.5ex]
\hline
\end{tabular}
\end{center}}

So, by Lemma \ref{Lem4}, $\mathfrak{S}_{499}(1)=\big\{\mathcal{O}v_1^{r_1}:r_1\geq1\big\}\bigcup
\big[\bigcup^3_{j=2}\big\{\mathfrak{U}_jv_j^{r_j}:r_j\geq5\,\,\text{and}\,\,r_j\neq6,7,8,9,11,12,13,14,16,17,18,19,21,22,23,24,26,27,28,32,33,37,38,42\big\}\big]$ and $g_{499}(1)=4$.

\vskip8pt\noindent{\bf Case 13.} $E=\mathbb{Q}\big(\sqrt{-547}\big)$ with $\mathcal{O}=\mathbb{Z}+\mathbb{Z}\omega$ for $\omega=\frac{1+\sqrt{-547}}{2}$.

In this case, one has $m_{547}=4$, $k=11$, $n=5$ and
\begin{equation*}
h\big(v_2^{r_2}\big)=\frac{r_2}{11}=\sum_{\ell=1}^mN\Big(\frac{\gamma_\ell}{11}\Big)
=\sum_{\ell=1}^m\bigg(\frac{a_\ell^2}{121}+\frac{a_\ell b_\ell}{121}+\frac{137b_\ell^2}{121}\bigg)=:\sum_{\ell=1}^mP_{547}(a_\ell,b_\ell),
\end{equation*}
where $a_\ell$ and $b_\ell$ are integers satisfying $11|(a_\ell+3b_\ell)$.
For $r<137$ and $r=4^a(8b+7)$ with $a,b\in\mathbb{Z}_{\geq0}$, one has $r=7,15,23,28,31,39,47,55,60,63,71,79,87,92,95,103,$
$111,112,119,124,127,135$, whose detailed representations are given as before; others are given in the table below
{\small\begin{center}
\setlength{\arrayrulewidth}{0.2mm}
\setlength{\tabcolsep}{2pt}
\renewcommand{\arraystretch}{0.81}
\begin{tabular}{|m{1.61cm}|m{8.1cm}|m{1.81cm}|}
\hline\vskip2pt $r(\delta)$ & \vskip2pt $r_2<r(\delta)$ and $r_2\equiv\delta~(\mathrm{mod}~11)$       & \vskip2pt $r_2=r(\delta)$  \\[0.5ex]

\hline\vskip2pt $r(1)=67$   & \vskip2pt \,\,\,$45=2P_{547}(3,-1)+P_{547}(8,1)$                        & \vskip2pt $P_{547}(25,-1)$ \\[0.5ex]
                            &           \,\,\,$56=2P_{547}(3,-1)+P_{547}(8,1)+P_{547}(11,0)$          &                            \\[0.5ex]
                            &           Others cannot be represented.                                 &                            \\[0.5ex]

\hline\vskip2pt $r(2)=13$   & \vskip2pt None can be represented.                                      & \vskip2pt $P_{547}(3,-1)$  \\[0.5ex]

\hline\vskip2pt $r(3)=47$   & \vskip2pt None can be represented.                                      & \vskip2pt $P_{547}(19,1)$  \\[0.5ex]

\hline\vskip2pt $r(4)=169$  & \vskip2pt \,\,\,$26=2P_{547}(3,-1)$                                     & \vskip2pt $P_{547}(41,1)$  \\[0.5ex]
                            &           \,\,\,$37=2P_{547}(3,-1)+P_{547}(11,0)$                       &                            \\[0.5ex]
                            &           \,\,\,$48=2P_{547}(3,-1)+2P_{547}(11,0)$                      &                            \\[0.5ex]
                            &           \,\,\,$59=P_{547}(8,1)+P_{547}(11,0)+P_{547}(14,-1)$          &                            \\[0.5ex]
                            &           \,\,\,$70=P_{547}(8,1)+2P_{547}(11,0)+P_{547}(14,-1)$         &                            \\[0.5ex]
                            &           \,\,\,$81=P_{547}(6,-2)+P_{547}(14,-1)$                       &                            \\[0.5ex]

\hline
\end{tabular}
\end{center}}

\newpage\noindent{\color{blue}table continued}

{\small\begin{center}
\setlength{\arrayrulewidth}{0.2mm}
\setlength{\tabcolsep}{2pt}
\renewcommand{\arraystretch}{0.81}
\begin{tabular}{|m{1.61cm}|m{8.1cm}|m{1.81cm}|}
\hline\vskip2pt $r(4)=169$  & \vskip2pt \,\,\,$92=P_{547}(6,-2)+P_{547}(11,0)+P_{547}(14,-1)$         &                            \\[0.5ex]
{\color{blue}(continued)}   &           $103=P_{547}(6,-2)+2P_{547}(11,0)+P_{547}(14,-1)$             &                            \\[0.5ex]
                            &           $114=P_{547}(8,1)+2P_{547}(11,0)+P_{547}(17,-2)$              &                            \\[0.5ex]
                            &           $125=P_{547}(6,-2)+P_{547}(14,-1)+P_{547}(22,0)$              &                            \\[0.5ex]
                            &           $136=P_{547}(8,1)+P_{547}(17,-2)+P_{547}(22,0)$               &                            \\[0.5ex]
                            &           $147=P_{547}(8,1)+P_{547}(9,-3)+P_{547}(11,0)$                &                            \\[0.5ex]
                            &           $158=P_{547}(8,1)+P_{547}(9,-3)+2P_{547}(11,0)$               &                            \\[0.5ex]
                            &           Others cannot be represented.                                 &                            \\[0.5ex]

\hline\vskip2pt $r(5)=181$  & \vskip2pt \,\,\,$38=2P_{547}(8,1)$                                      & \vskip2pt $P_{547}(39,-2)$ \\[0.5ex]
                            &           \,\,\,$49=2P_{547}(8,1)+P_{547}(11,0)$                        &                            \\[0.5ex]
                            &           \,\,\,$60=2P_{547}(8,1)+2P_{547}(11,0)$                       &                            \\[0.5ex]
                            &           \,\,\,$71=P_{547}(6,-2)+P_{547}(8,1)$                         &                            \\[0.5ex]
                            &           \,\,\,$82=P_{547}(6,-2)+P_{547}(8,1)+P_{547}(11,0)$           &                            \\[0.5ex]
                            &           \,\,\,$93=P_{547}(6,-2)+P_{547}(8,1)+2P_{547}(11,0)$          &                            \\[0.5ex]
                            &           $104=P_{547}(5,2)+2P_{547}(11,0)+P_{547}(14,-1)$              &                            \\[0.5ex]
                            &           $115=P_{547}(6,-2)+P_{547}(8,1)+P_{547}(22,0)$                &                            \\[0.5ex]
                            &           $126=P_{547}(5,2)+P_{547}(14,-1)+P_{547}(22,0)$               &                            \\[0.5ex]
                            &           $137=2P_{547}(8,1)+P_{547}(33,0)$                             &                            \\[0.5ex]
                            &           $148=2P_{547}(8,1)+P_{547}(11,0)+P_{547}(33,0)$               &                            \\[0.5ex]
                            &           $159=P_{547}(6,-2)+P_{547}(8,1)+2P_{547}(22,0)$               &                            \\[0.5ex]
                            &           $170=P_{547}(5,2)+P_{547}(14,-1)+2P_{547}(22,0)$              &                            \\[0.5ex]
                            &           Others cannot be represented.                                 &                            \\[0.5ex]

\hline\vskip2pt $r(6)=116$  & \vskip2pt \,\,\,$39=3P_{547}(3,-1)$                                     & \vskip2pt $P_{547}(28,-2)$ \\[0.5ex]
                            &           \,\,\,$50=3P_{547}(3,-1)+P_{547}(11,0)$                       &                            \\[0.5ex]
                            &           \,\,\,$61=P_{547}(3,-1)+P_{547}(8,1)+P_{547}(14,-1)$          &                            \\[0.5ex]
                            &           \,\,\,$72=P_{547}(5,2)+P_{547}(8,1)$                          &                            \\[0.5ex]
                            &           \,\,\,$83=P_{547}(5,2)+P_{547}(8,1)+P_{547}(11,0)$            &                            \\[0.5ex]
                            &           \,\,\,$94=P_{547}(5,2)+P_{547}(8,1)+2P_{547}(11,0)$           &                            \\[0.5ex]
                            &           $105=P_{547}(3,-1)+P_{547}(8,1)+P_{547}(14,-1)+P_{547}(22,0)$ &                            \\[0.5ex]
                            &           Others cannot be represented.                                 &                            \\[0.5ex]

\hline\vskip2pt $r(7)=29$   & \vskip2pt None can be represented.                                      & \vskip2pt $P_{547}(14,-1)$ \\[0.5ex]

\hline\vskip2pt $r(8)=19$   & \vskip2pt None can be represented.                                      & \vskip2pt $P_{547}(8,1)$   \\[0.5ex]

\hline\vskip2pt $r(9)=53$   & \vskip2pt \,\,\,$42=P_{547}(3,-1)+P_{547}(14,-1)$                       & \vskip2pt $P_{547}(5,2)$   \\[0.5ex]
                            &           Others cannot be represented.                                 &                            \\[0.5ex]

\hline\vskip2pt $r(10)=76$  & \vskip2pt \,\,\,$32=P_{547}(3,-1)+P_{547}(8,1)$                         & \vskip2pt $P_{547}(16,2)$  \\[0.5ex]
                            &           \,\,\,$43=P_{547}(3,-1)+P_{547}(8,1)+P_{547}(11,0)$           &                            \\[0.5ex]
                            &           \,\,\,$54=P_{547}(3,-1)+P_{547}(8,1)+2P_{547}(11,0)$          &                            \\[0.5ex]
                            &           \,\,\,$65=P_{547}(3,-1)+P_{547}(6,-2)$                        &                            \\[0.5ex]
                            &           Others cannot be represented.                                 &                            \\[0.5ex]
\hline
\end{tabular}
\end{center}}

So, by Lemma \ref{Lem4}, $\mathfrak{S}_{547}(1)=\big\{\mathcal{O}v_1^{r_1}:r_1\geq1\big\}
\bigcup\big[\bigcup^3_{j=2}\big\{\mathfrak{U}_jv_j^{r_j}:r_j\geq11\,\,\text{and}\,\,r_j\neq12,14,15,16,17,18,20,21,23,25,27,28,31,34,36\big\}\big]$ and $g_{547}(1)=4$.

\vskip8pt\noindent{\bf Case 14.} $E=\mathbb{Q}\big(\sqrt{-643}\big)$ with $\mathcal{O}=\mathbb{Z}+\mathbb{Z}\omega$ for $\omega=\frac{1+\sqrt{-643}}{2}$.

In this case, one has $m_{643}=4$, $k=7$, $n=1$ and
\begin{equation*}
h\big(v_2^{r_2}\big)=\frac{r_2}{7}=\sum_{\ell=1}^mN\Big(\frac{\gamma_\ell}{7}\Big)
=\sum_{\ell=1}^m\bigg(\frac{a_\ell^2}{49}+\frac{a_\ell b_\ell}{49}+\frac{161b_\ell^2}{49}\bigg)=:\sum_{\ell=1}^mP_{643}(a_\ell,b_\ell),
\end{equation*}
where $a_\ell$ and $b_\ell$ are integers satisfying $7|(a_\ell+b_\ell)$.
For $r<161$ and $r=4^a(8b+7)$ with $a,b\in\mathbb{Z}_{\geq0}$, one has $r=7,15,23,28,31,39,47,55,60,63,71,79,87,92,95,103,$
$111,112,119,124,127,135,143,151,156,159$, with detailed representations given as before; others are given in the table below
{\small\begin{center}
\setlength{\arrayrulewidth}{0.2mm}
\setlength{\tabcolsep}{2pt}
\renewcommand{\arraystretch}{0.81}
\begin{tabular}{|m{1.37cm}|m{6.1cm}|m{1.81cm}|}
\hline\vskip2pt $r(\delta)$ & \vskip2pt $r_2<r(\delta)$ and $r_2\equiv\delta~(\mathrm{mod}~7)$ & \vskip2pt $r_2=r(\delta)$  \\[0.5ex]

\hline\vskip2pt $r(1)=29$   & \vskip2pt None can be represented.                               & \vskip2pt $P_{643}(6,1)$   \\[0.5ex]

\hline\vskip2pt $r(2)=23$   & \vskip2pt None can be represented.                               & \vskip2pt $P_{643}(1,-1)$  \\[0.5ex]

\hline\vskip2pt $r(3)=31$   & \vskip2pt None can be represented.                               & \vskip2pt $P_{643}(8,-1)$  \\[0.5ex]

\hline\vskip2pt $r(4)=53$   & \vskip2pt $46=2P_{643}(1,-1)$                                    & \vskip2pt $P_{643}(15,-1)$ \\[0.5ex]

\hline\vskip2pt $r(5)=89$   & \vskip2pt $54=P_{643}(1,-1)+P_{643}(8,-1)$                       & \vskip2pt $P_{643}(22,-1)$ \\[0.5ex]
                            &           $61=P_{643}(1,-1)+P_{643}(7,0)+P_{643}(8,-1)$          &                            \\[0.5ex]
                            &           $68=P_{643}(1,-1)+2P_{643}(7,0)+P_{643}(8,-1)$         &                            \\[0.5ex]
                            &           $75=2P_{643}(1,-1)+P_{643}(6,1)$                       &                            \\[0.5ex]
                            &           $82=P_{643}(6,1)+P_{643}(15,-1)$                       &                            \\[0.5ex]
                            &           Others cannot be represented.                          &                            \\[0.5ex]

\hline\vskip2pt $r(6)=83$   & \vskip2pt $62=2P_{643}(8,-1)$                                    & \vskip2pt $P_{643}(20,1)$  \\[0.5ex]
                            &           $69=P_{643}(7,0)+2P_{643}(8,-1)$                       &                            \\[0.5ex]
                            &           $76=2P_{643}(7,0)+2P_{643}(8,-1)$                      &                            \\[0.5ex]
                            &           Others cannot be represented.                          &                            \\[0.5ex]
\hline
\end{tabular}
\end{center}}

So, by Lemma \ref{Lem4}, $\mathfrak{S}_{643}(1)=\big\{\mathcal{O}v_1^{r_1}:r_1\geq1\big\}
\bigcup\big[\bigcup^3_{j=2}\big\{\mathfrak{U}_jv_j^{r_j}:r_j\geq7\,\,\text{and}\,\,r_j\neq8,9,10,11,12,13,15,16,17,18,19,20,22,24,25,26,27,32,33,34,39,40,41,47,48,$
$55\big\}\big]$ and $g_{643}(1)=4$.

\vskip8pt\noindent{\bf Case 15.} $E=\mathbb{Q}\big(\sqrt{-883}\big)$ with $\mathcal{O}=\mathbb{Z}+\mathbb{Z}\omega$ for $\omega=\frac{1+\sqrt{-883}}{2}$.

In this case, one has $m_{883}=4$, $k=13$, $n=1$ and
\begin{equation*}
h\big(v_2^{r_2}\big)=\frac{r_2}{13}=\sum_{\ell=1}^mN\Big(\frac{\gamma_\ell}{13}\Big)
=\sum_{\ell=1}^m\bigg(\frac{a_\ell^2}{169}+\frac{a_\ell b_\ell}{169}+\frac{221b_\ell^2}{169}\bigg)=:\sum_{\ell=1}^mP_{883}(a_\ell,b_\ell),
\end{equation*}
where $a_\ell$ and $b_\ell$ are integers satisfying $13|(a_\ell+b_\ell)$.
For $r<221$ and $r=4^a(8b+7)$ with $a,b\in\mathbb{Z}_{\geq0}$, one has $r=7,15,23,28,31,39,47,55,60,63,71,79,87,92,95,103,$
$111,112,119,124,127,135,143,151,156,159,167,175,183,188,191,199,207,215,$
$220$, with detailed representations given as before; others are given in the table below
{\small\begin{center}
\setlength{\arrayrulewidth}{0.2mm}
\setlength{\tabcolsep}{2pt}
\renewcommand{\arraystretch}{0.81}
\begin{tabular}{|m{1.7cm}|m{8.44cm}|m{1.81cm}|}
\hline\vskip2pt $r(\delta)$ & \vskip2pt $r_2<r(\delta)$ and $r_2\equiv\delta~(\mathrm{mod}~13)$             & \vskip2pt $r_2=r(\delta)$  \\[0.5ex]

\hline\vskip2pt $r(1)=79$   & \vskip2pt None can be represented.                                            & \vskip2pt $P_{883}(11,2)$  \\[0.5ex]

\hline\vskip2pt $r(2)=67$   & \vskip2pt None can be represented.                                            & \vskip2pt $P_{883}(25,1)$  \\[0.5ex]
\hline
\end{tabular}
\end{center}}

\newpage\noindent{\color{blue}table continued}

{\small\begin{center}
\setlength{\arrayrulewidth}{0.2mm}
\setlength{\tabcolsep}{2pt}
\renewcommand{\arraystretch}{0.81}
\begin{tabular}{|m{1.7cm}|m{8.44cm}|m{1.81cm}|}
\hline\vskip2pt $r(3)=29$   & \vskip2pt None can be represented.                                            & \vskip2pt $P_{883}(12,1)$  \\[0.5ex]

\hline\vskip2pt $r(4)=17$   & \vskip2pt None can be represented.                                            & \vskip2pt $P_{883}(1,-1)$  \\[0.5ex]

\hline\vskip2pt $r(5)=31$   & \vskip2pt None can be represented.                                            & \vskip2pt $P_{883}(14,-1)$ \\[0.5ex]

\hline\vskip2pt $r(6)=279$  & \vskip2pt \,\,\,$58=2P_{883}(12,1)$                                           & \vskip2pt $P_{883}(42,-3)$ \\[0.5ex]
                            &           \,\,\,$71=2P_{883}(12,1)+P_{883}(13,0)$                             &                            \\[0.5ex]
                            &           \,\,\,$84=2P_{883}(12,1)+2P_{883}(13,0)$                            &                            \\[0.5ex]
                            &           \,\,\,$97=2P_{883}(13,0)+P_{883}(27,-1)$                            &                            \\[0.5ex]
                            &           $110=2P_{883}(12,1)+P_{883}(26,0)$                                  &                            \\[0.5ex]
                            &           $123=P_{883}(26,0)+P_{883}(27,-1)$                                  &                            \\[0.5ex]
                            &           $136=P_{883}(13,0)+P_{883}(26,0)+P_{883}(27,-1)$                    &                            \\[0.5ex]
                            &           $149=2P_{883}(13,0)+P_{883}(26,0)+P_{883}(27,-1)$                   &                            \\[0.5ex]
                            &           $162=2P_{883}(12,1)+2P_{883}(26,0)$                                 &                            \\[0.5ex]
                            &           $175=2P_{883}(12,1)+P_{883}(39,0)$                                  &                            \\[0.5ex]
                            &           $188=2P_{883}(12,1)+P_{883}(13,0)+P_{883}(39,0)$                    &                            \\[0.5ex]
                            &           $201=P_{883}(13,0)+P_{883}(27,-1)+P_{883}(39,0)$                    &                            \\[0.5ex]
                            &           $214=2P_{883}(13,0)+P_{883}(27,-1)+P_{883}(39,0)$                   &                            \\[0.5ex]
                            &           $227=3P_{883}(26,0)+P_{883}(27,-1)$                                 &                            \\[0.5ex]
                            &           $240=P_{883}(26,0)+P_{883}(27,-1)+P_{883}(39,0)$                    &                            \\[0.5ex]
                            &           $253=P_{883}(13,0)+P_{883}(26,0)+P_{883}(27,-1)+P_{883}(39,0)$      &                            \\[0.5ex]
                            &           $266=2P_{883}(12,1)+P_{883}(52,0)$                                  &                            \\[0.5ex]
                            &           Others cannot be represented.                                       &                            \\[0.5ex]

\hline\vskip2pt $r(7)=124$  & \vskip2pt \,\,\,$46=P_{883}(1,-1)+P_{883}(12,1)$                              & \vskip2pt $P_{883}(28,-2)$ \\[0.5ex]
                            &           \,\,\,$59=P_{883}(1,-1)+P_{883}(12,1)+P_{883}(13,0)$                &                            \\[0.5ex]
                            &           \,\,\,$72=P_{883}(1,-1)+P_{883}(12,1)+2P_{883}(13,0)$               &                            \\[0.5ex]
                            &           \,\,\,$85=P_{883}(1,-1)+P_{883}(2,-2)$                              &                            \\[0.5ex]
                            &           \,\,\,$98=P_{883}(1,-1)+P_{883}(2,-2)+P_{883}(13,0)$                &                            \\[0.5ex]
                            &           $111=P_{883}(1,-1)+P_{883}(2,-2)+2P_{883}(13,0)$                    &                            \\[0.5ex]
                            &           Others cannot be represented.                                       &                            \\[0.5ex]

\hline\vskip2pt $r(8)=229$  & \vskip2pt \,\,\,$34=2P_{883}(1,-1)$                                           & \vskip2pt $P_{883}(53,-1)$ \\[0.5ex]
                            &           \,\,\,$47=2P_{883}(1,-1)+P_{883}(13,0)$                             &                            \\[0.5ex]
                            &           \,\,\,$60=2P_{883}(1,-1)+2P_{883}(13,0)$                            &                            \\[0.5ex]
                            &           \,\,\,$73=P_{883}(12,1)+P_{883}(13,0)+P_{883}(14,-1)$               &                            \\[0.5ex]
                            &           \,\,\,$86=2P_{883}(1,-1)+P_{883}(26,0)$                             &                            \\[0.5ex]
                            &           \,\,\,$99=2P_{883}(1,-1)+P_{883}(13,0)+P_{883}(26,0)$               &                            \\[0.5ex]
                            &           $112=P_{883}(12,1)+P_{883}(14,-1)+P_{883}(26,0)$                    &                            \\[0.5ex]
                            &           $125=P_{883}(12,1)+P_{883}(14,-1)+P_{883}(13,0)+P_{883}(26,0)$      &                            \\[0.5ex]
                            &           $138=2P_{883}(1,-1)+2P_{883}(26,0)$                                 &                            \\[0.5ex]
                            &           $151=2P_{883}(1,-1)+P_{883}(39,0)$                                  &                            \\[0.5ex]
                            &           $164=2P_{883}(1,-1)+P_{883}(13,0)+P_{883}(39,0)$                    &                            \\[0.5ex]
                            &           $177=P_{883}(11,2)+P_{883}(14,-1)+P_{883}(25,1)$                    &                            \\[0.5ex]
                            &           $190=P_{883}(25,1)+P_{883}(26,0)+P_{883}(27,-1)$                    &                            \\[0.5ex]
                            &           $203=P_{883}(11,2)+P_{883}(28,-2)$                                  &                            \\[0.5ex]
                            &           $216=P_{883}(11,2)+P_{883}(13,0)+P_{883}(28,-2)$                    &                            \\[0.5ex]
                            &           Others cannot be represented.                                       &                            \\[0.5ex]
\hline
\end{tabular}
\end{center}}

\newpage\noindent{\color{blue}table continued}

{\small\begin{center}
\setlength{\arrayrulewidth}{0.2mm}
\setlength{\tabcolsep}{2pt}
\renewcommand{\arraystretch}{0.81}
\begin{tabular}{|m{1.7cm}|m{8.44cm}|m{1.81cm}|}
\hline\vskip2pt $r(9)=191$  & \vskip2pt \,\,\,$48=P_{883}(1,-1)+P_{883}(14,-1)$                             & \vskip2pt $P_{883}(41,-2)$ \\[0.5ex]
                            &           \,\,\,$61=P_{883}(1,-1)+P_{883}(13,0)+P_{883}(14,-1)$               &                            \\[0.5ex]
                            &           \,\,\,$74=P_{883}(1,-1)+2P_{883}(13,0)+P_{883}(14,-1)$              &                            \\[0.5ex]
                            &           \,\,\,$87=3P_{883}(12,1)$                                           &                            \\[0.5ex]
                            &           $100=P_{883}(12,1)+P_{883}(27,-1)$                                  &                            \\[0.5ex]
                            &           $113=P_{883}(12,1)+P_{883}(13,0)+P_{883}(27,-1)$                    &                            \\[0.5ex]
                            &           $126=P_{883}(12,1)+2P_{883}(13,0)+P_{883}(27,-1)$                   &                            \\[0.5ex]
                            &           $139=3P_{883}(12,1)+P_{883}(26,0)$                                  &                            \\[0.5ex]
                            &           $152=P_{883}(12,1)+P_{883}(26,0)+P_{883}(27,-1)$                    &                            \\[0.5ex]
                            &           $165=P_{883}(1,-1)+P_{883}(14,-1)+P_{883}(39,0)$                    &                            \\[0.5ex]
                            &           $178=P_{883}(1,-1)+P_{883}(13,0)+P_{883}(14,-1)+P_{883}(39,0)$      &                            \\[0.5ex]
                            &           Others cannot be represented.                                       &                            \\[0.5ex]

\hline\vskip2pt $r(10)=153$ & \vskip2pt \,\,\,$62=2P_{883}(14,-1)$                                          & \vskip2pt $P_{883}(3,-3)$  \\[0.5ex]
                            &           \,\,\,$75=P_{883}(13,0)+2P_{883}(14,-1)$                            &                            \\[0.5ex]
                            &           \,\,\,$88=2P_{883}(13,0)+2P_{883}(14,-1)$                           &                            \\[0.5ex]
                            &           $101=P_{883}(1,-1)+P_{883}(13,0)+P_{883}(27,-1)$                    &                            \\[0.5ex]
                            &           $114=2P_{883}(14,-1)+P_{883}(26,0)$                                 &                            \\[0.5ex]
                            &           $127=P_{883}(13,0)+2P_{883}(14,-1)+P_{883}(26,0)$                   &                            \\[0.5ex]
                            &           $140=P_{883}(1,-1)+P_{883}(26,0)+P_{883}(27,-1)$                    &                            \\[0.5ex]
                            &           Others cannot be represented.                                       &                            \\[0.5ex]

\hline\vskip2pt $r(11)=284$ & \vskip2pt \,\,\,$63=2P_{883}(1,-1)+P_{883}(12,1)$                             & \vskip2pt $P_{883}(54,-2)$ \\[0.5ex]
                            &           \,\,\,$76=2P_{883}(1,-1)+P_{883}(12,1)+P_{883}(13,0)$               &                            \\[0.5ex]
                            &           \,\,\,$89=2P_{883}(12,1)+P_{883}(14,-1)$                            &                            \\[0.5ex]
                            &           $102=2P_{883}(12,1)+P_{883}(13,0)+P_{883}(14,-1)$                   &                            \\[0.5ex]
                            &           $115=P_{883}(13,0)+P_{883}(14,-1)+P_{883}(27,-1)$                   &                            \\[0.5ex]
                            &           $128=2P_{883}(13,0)+P_{883}(14,-1)+P_{883}(27,-1)$                  &                            \\[0.5ex]
                            &           $141=P_{883}(1,-1)+P_{883}(28,-2)$                                  &                            \\[0.5ex]
                            &           $154=P_{883}(1,-1)+P_{883}(13,0)+P_{883}(28,-2)$                    &                            \\[0.5ex]
                            &           $167=P_{883}(1,-1)+2P_{883}(13,0)+P_{883}(28,-2)$                   &                            \\[0.5ex]
                            &           $180=P_{883}(1,-1)+P_{883}(10,3)$                                   &                            \\[0.5ex]
                            &           $193=P_{883}(1,-1)+P_{883}(10,3)+P_{883}(13,0)$                     &                            \\[0.5ex]
                            &           $206=P_{883}(1,-1)+P_{883}(10,3)+2P_{883}(13,0)$                    &                            \\[0.5ex]
                            &           $219=P_{883}(13,0)+P_{883}(15,-2)+P_{883}(26,0)+P_{883}(27,-1)$     &                            \\[0.5ex]
                            &           $232=P_{883}(1,-1)+P_{883}(10,3)+P_{883}(26,0)$                     &                            \\[0.5ex]
                            &           $245=P_{883}(1,-1)+P_{883}(10,3)+P_{883}(13,0)+P_{883}(26,0)$       &                            \\[0.5ex]
                            &           $258=P_{883}(1,-1)+P_{883}(28,-2)+P_{883}(39,0)$                    &                            \\[0.5ex]
                            &           $271=P_{883}(1,-1)+P_{883}(13,0)+P_{883}(28,-2)+P_{883}(39,0)$      &                            \\[0.5ex]
                            &           Others cannot be represented.                                       &                            \\[0.5ex]

\hline\vskip2pt $r(12)=116$ & \vskip2pt \,\,\,$51=3P_{883}(1,-1)$                                           & \vskip2pt $P_{883}(24,2)$  \\[0.5ex]
                            &           \,\,\,$64=3P_{883}(1,-1)+P_{883}(13,0)$                             &                            \\[0.5ex]
                            &           \,\,\,$77=P_{883}(1,-1)+P_{883}(12,1)+P_{883}(14,-1)$               &                            \\[0.5ex]
                            &           \,\,\,$90=P_{883}(1,-1)+P_{883}(12,1)+P_{883}(13,0)+P_{883}(14,-1)$ &                            \\[0.5ex]
                            &           $103=3P_{883}(1,-1)+P_{883}(13,0)$                                  &                            \\[0.5ex]
                            &           Others cannot be represented.                                       &                            \\[0.5ex]
\hline
\end{tabular}
\end{center}}

So, by Lemma \ref{Lem4}, $\mathfrak{S}_{883}(1)=\big\{\mathcal{O}v_1^{r_1}:r_1\geq1\big\}\bigcup
\big[\bigcup^3_{j=2}\big\{\mathfrak{U}_jv_j^{r_j}:r_j\geq13\,\,\text{and}\,\,r_j\neq14,15,16,18,19,20,21,22,23,24,25,27,28,32,33,35,36,37,38,40,41,45,49,50,$\newline
$53,54,66\big\}\big]$ and $g_{883}(1)=4$.

\vskip8pt\noindent{\bf Case 16.} $E=\mathbb{Q}\big(\sqrt{-907}\big)$ with $\mathcal{O}=\mathbb{Z}+\mathbb{Z}\omega$ for $\omega=\frac{1+\sqrt{-907}}{2}$.

In this case, one has $m_{907}=4$, $k=13$, $n=9$ and
\begin{equation*}
h\big(v_2^{r_2}\big)=\frac{r_2}{13}=\sum_{\ell=1}^mN\Big(\frac{\gamma_\ell}{13}\Big)
=\sum_{\ell=1}^m\bigg(\frac{a_\ell^2}{169}+\frac{a_\ell b_\ell}{169}+\frac{227b_\ell^2}{169}\bigg)=:\sum_{\ell=1}^mP_{907}(a_\ell,b_\ell),
\end{equation*}
where $a_\ell$ and $b_\ell$ are integers satisfying $13|(a_\ell+5b_\ell)$.
For $r<227$ and $r=4^a(8b+7)$ with $a,b\in\mathbb{Z}_{\geq0}$, $r=7,15,23,28,31,39,47,55,60,63,71,79,87,92,95,103,111,112,$
$119,124,127,135,143,151,156,159,167,175,183,188,191,199,207,215,220,223$, with detailed representations available as before; others are given in the table below
{\small\begin{center}
\setlength{\arrayrulewidth}{0.2mm}
\setlength{\tabcolsep}{2pt}
\renewcommand{\arraystretch}{0.81}
\begin{tabular}{|m{1.72cm}|m{8.28cm}|m{1.81cm}|}
\hline\vskip2pt $r(\delta)$ & \vskip2pt $r_2<r(\delta)$ and $r_2\equiv\delta~(\mathrm{mod}~13)$            & \vskip2pt $r_2=r(\delta)$  \\[0.5ex]

\hline\vskip2pt $r(1)=53$   & \vskip2pt None can be represented.                                           & \vskip2pt $P_{907}(21,1)$  \\[0.5ex]

\hline\vskip2pt $r(2)=41$   & \vskip2pt None can be represented.                                           & \vskip2pt $P_{907}(18,-1)$ \\[0.5ex]

\hline\vskip2pt $r(3)=107$  & \vskip2pt \,\,\,$42=P_{907}(5,-1)+P_{907}(8,1)$                              & \vskip2pt $P_{907}(23,-2)$ \\[0.5ex]
                            &           \,\,\,$55=P_{907}(5,-1)+P_{907}(8,1)+P_{907}(13,0)$                &                            \\[0.5ex]
                            &           \,\,\,$68=P_{907}(5,-1)+P_{907}(8,1)+2P_{907}(13,0)$               &                            \\[0.5ex]
                            &           \,\,\,$81=P_{907}(5,-1)+P_{907}(8,1)+3P_{907}(13,0)^{\,2}$         &                            \\[0.5ex]
                            &           \,\,\,$94=P_{907}(5,-1)+P_{907}(8,1)+P_{907}(26,0)$                &                            \\[0.5ex]
                            &           Others cannot be represented.                                      &                            \\[0.5ex]

\hline\vskip2pt $r(4)=212$  & \vskip2pt \,\,\,$69=3P_{907}(8,1)$                                           & \vskip2pt $P_{907}(42,2)$  \\[0.5ex]
                            &           \,\,\,$82=2P_{907}(18,-1)$                                         &                            \\[0.5ex]
                            &           \,\,\,$95=P_{907}(13,0)+2P_{907}(18,-1)$                           &                            \\[0.5ex]
                            &           $108=2P_{907}(13,0)+2P_{907}(18,-1)$                               &                            \\[0.5ex]
                            &           $121=P_{907}(5,-1)+P_{907}(10,-2)+2P_{907}(13,0)$                  &                            \\[0.5ex]
                            &           $134=2P_{907}(18,-1)+P_{907}(26,0)$                                &                            \\[0.5ex]
                            &           $147=P_{907}(13,0)+2P_{907}(18,-1)+P_{907}(26,0)$                  &                            \\[0.5ex]
                            &           $160=P_{907}(5,-1)+P_{907}(10,-2)+P_{907}(13,0)+P_{907}(26,0)$     &                            \\[0.5ex]
                            &           $173=P_{907}(3,2)+P_{907}(13,0)+P_{907}(31,-1)$                    &                            \\[0.5ex]
                            &           $186=P_{907}(3,2)+2P_{907}(13,0)+P_{907}(31,-1)$                   &                            \\[0.5ex]
                            &           $199=P_{907}(5,-1)+P_{907}(10,-2)+2P_{907}(26,0)$                  &                            \\[0.5ex]
                            &           Others cannot be represented.                                      &                            \\[0.5ex]

\hline\vskip2pt $r(5)=109$  & \vskip2pt \,\,\,$57=3P_{907}(5,-1)$                                          & \vskip2pt $P_{907}(34,1)$  \\[0.5ex]
                            &           \,\,\,$70=3P_{907}(5,-1)+P_{907}(13,0)$                            &                            \\[0.5ex]
                            &           \,\,\,$83=P_{907}(5,-1)+P_{907}(8,1)+P_{907}(18,-1)$               &                            \\[0.5ex]
                            &           \,\,\,$96=P_{907}(5,-1)+P_{907}(8,1)+P_{907}(13,0)+P_{907}(18,-1)$ &                            \\[0.5ex]
                            &           Others cannot be represented.                                      &                            \\[0.5ex]

\hline\vskip2pt $r(6)=19$   & \vskip2pt None can be represented.                                           & \vskip2pt $P_{907}(5,-1)$  \\[0.5ex]

\hline\vskip2pt $r(7)=163$  & \vskip2pt \,\,\,$46=2P_{907}(8,1)$                                           & \vskip2pt $P_{907}(44,-1)$ \\[0.5ex]
                            &           \,\,\,$59=2P_{907}(8,1)+P_{907}(13,0)$                             &                            \\[0.5ex]
\hline
\end{tabular}
\end{center}}

\noindent\rule{142pt}{0.5pt}
\vskip2pt\noindent{\footnotesize $^2$This is where one has $g_{907}(1)=5$.}
\newpage\noindent{\color{blue}table continued}

{\small\begin{center}
\setlength{\arrayrulewidth}{0.2mm}
\setlength{\tabcolsep}{2pt}
\renewcommand{\arraystretch}{0.81}
\begin{tabular}{|m{1.72cm}|m{8.28cm}|m{1.81cm}|}
\hline\vskip2pt $r(7)=163$  & \vskip2pt \,\,\,$72=2P_{907}(8,1)+2P_{907}(13,0)$                            &                            \\[0.5ex]
{\color{blue}(continued)}   &           \,\,\,$85=P_{907}(5,-1)+P_{907}(13,0)+P_{907}(21,1)$               &                            \\[0.5ex]
                            &           \,\,\,$98=2P_{907}(8,1)+P_{907}(26,0)$                             &                            \\[0.5ex]
                            &           $111=2P_{907}(8,1)+P_{907}(13,0)+P_{907}(26,0)$                    &                            \\[0.5ex]
                            &           $124=P_{907}(5,-1)+P_{907}(21,1)+P_{907}(26,0)$                    &                            \\[0.5ex]
                            &           $137=P_{907}(5,-1)+P_{907}(13,0)+P_{907}(21,1)+P_{907}(26,0)$      &                            \\[0.5ex]
                            &           $150=2P_{907}(8,1)+2P_{907}(26,0)$                                 &                            \\[0.5ex]
                            &           Others cannot be represented.                                      &                            \\[0.5ex]

\hline\vskip2pt $r(8)=164$  & \vskip2pt \,\,\,$60=P_{907}(5,-1)+P_{907}(18,-1)$                            & \vskip2pt $P_{907}(36,-2)$ \\[0.5ex]
                            &           \,\,\,$73=P_{907}(5,-1)+P_{907}(13,0)+P_{907}(18,-1)$              &                            \\[0.5ex]
                            &           \,\,\,$86=P_{907}(5,-1)+2P_{907}(13,0)+P_{907}(18,-1)$             &                            \\[0.5ex]
                            &           \,\,\,$99=P_{907}(8,1)+P_{907}(10,-2)$                             &                            \\[0.5ex]
                            &           $112=P_{907}(8,1)+P_{907}(10,-2)+P_{907}(13,0)$                    &                            \\[0.5ex]
                            &           $125=P_{907}(8,1)+P_{907}(10,-2)+2P_{907}(13,0)$                   &                            \\[0.5ex]
                            &           $138=P_{907}(8,1)+2P_{907}(13,0)+P_{907}(31,-1)$                   &                            \\[0.5ex]
                            &           $151=P_{907}(8,1)+P_{907}(10,-2)+P_{907}(26,0)$                    &                            \\[0.5ex]
                            &           Others cannot be represented.                                      &                            \\[0.5ex]

\hline\vskip2pt $r(9)=139$  & \vskip2pt \,\,\,$61=2P_{907}(5,-1)+P_{907}(8,1)$                             & \vskip2pt $P_{907}(29,2)$  \\[0.5ex]
                            &           \,\,\,$74=2P_{907}(5,-1)+P_{907}(8,1)+P_{907}(13,0)$               &                            \\[0.5ex]
                            &           \,\,\,$87=2P_{907}(8,1)+P_{907}(18,-1)$                            &                            \\[0.5ex]
                            &           $100=2P_{907}(8,1)+P_{907}(13,0)+P_{907}(18,-1)$                   &                            \\[0.5ex]
                            &           $113=2P_{907}(5,-1)+P_{907}(8,1)+P_{907}(26,0)$                    &                            \\[0.5ex]
                            &           $126=P_{907}(5,-1)+P_{907}(23,-2)$                                 &                            \\[0.5ex]
                            &           Others cannot be represented.                                      &                            \\[0.5ex]

\hline\vskip2pt $r(10)=23$  & \vskip2pt None can be represented.                                           & \vskip2pt $P_{907}(8,1)$   \\[0.5ex]

\hline\vskip2pt $r(11)=76$  & \vskip2pt None can be represented.                                           & \vskip2pt $P_{907}(10,-2)$ \\[0.5ex]

\hline\vskip2pt $r(12)=207$ & \vskip2pt \,\,\,$38=2P_{907}(5,-1)$                                          & \vskip2pt $P_{907}(24,3)$  \\[0.5ex]
                            &           \,\,\,$51=2P_{907}(5,-1)+P_{907}(13,0)$                            &                            \\[0.5ex]
                            &           \,\,\,$64=2P_{907}(5,-1)+2P_{907}(13,0)$                           &                            \\[0.5ex]
                            &           \,\,\,$77=P_{907}(8,1)+P_{907}(13,0)+P_{907}(18,-1)$               &                            \\[0.5ex]
                            &           \,\,\,$90=2P_{907}(5,-1)+P_{907}(26,0)$                            &                            \\[0.5ex]
                            &           $103=2P_{907}(5,-1)+P_{907}(13,0)+P_{907}(26,0)$                   &                            \\[0.5ex]
                            &           $116=P_{907}(8,1)+P_{907}(18,-1)+P_{907}(26,0)$                    &                            \\[0.5ex]
                            &           $129=P_{907}(10,-2)+P_{907}(21,1)$                                 &                            \\[0.5ex]
                            &           $142=P_{907}(10,-2)+P_{907}(13,0)+P_{907}(21,1)$                   &                            \\[0.5ex]
                            &           $155=P_{907}(10,-2)+2P_{907}(13,0)+P_{907}(21,1)$                  &                            \\[0.5ex]
                            &           $168=P_{907}(8,1)+P_{907}(18,-1)+2P_{907}(26,0)$                   &                            \\[0.5ex]
                            &           $181=P_{907}(10,-2)+P_{907}(21,1)+P_{907}(26,0)$                   &                            \\[0.5ex]
                            &           $194=P_{907}(10,-2)+P_{907}(13,0)+P_{907}(21,1)+P_{907}(26,0)$     &                            \\[0.5ex]
                            &           Others cannot be represented.                                      &                            \\[0.5ex]
\hline
\end{tabular}
\end{center}}

So, by Lemma \ref{Lem4}, $\mathfrak{S}_{907}(1)=\big\{\mathcal{O}v_1^{r_1}:r_1\geq1\big\}\bigcup
\big[\bigcup^3_{j=2}\big\{\mathfrak{U}_jv_j^{r_j}:r_j\geq13\,\,\text{and}\,\,r_j\neq14,15,16,17,18,20,21,22,24,25,27,28,29,30,31,33,34,35,37,40,43,44,47,48,$\newline
$50,56,63\big\}\big]$ and $g_{907}(1)=5$.
\end{proof}

\section{Appendix 1}\label{Sec:App1}
Below, we record a representative from each different isometry class of non-free unary unimodular Hermitian lattices over $E=\mathbb{Q}\big(\sqrt{-d}\big)$, with class number $2$.
One can find details in the classical work of Baker \cite{aB} and Stark \cite{hmS}.
{\small\begin{center}
\setlength{\arrayrulewidth}{0.2mm}
\setlength{\tabcolsep}{2pt}
\renewcommand{\arraystretch}{0.88}
\begin{tabular}{|m{1.12cm}|m{1.52cm}|m{6.92cm}|}
\hline\vskip2pt Cases   & \vskip2pt $\mathbb{Q}\big(\sqrt{-d}\big)$   & \vskip2pt Isometry class representative                                            \\[0.5ex]
\hline\vskip2pt Case 1  & \vskip2pt $\mathbb{Q}\big(\sqrt{-5}\big)$   & \vskip2pt $L_2=\mathfrak{U}_2v_2$ with $h(v_2)=1/2$ and $\mathfrak{U}_2=(2,1+w)$   \\[0.5ex]
\hline\vskip2pt Case 2  & \vskip2pt $\mathbb{Q}\big(\sqrt{-6}\big)$   & \vskip2pt $L_2=\mathfrak{U}_2v_2$ with $h(v_2)=1/2$ and $\mathfrak{U}_2=(2,w)$     \\[0.5ex]
\hline\vskip2pt Case 3  & \vskip2pt $\mathbb{Q}\big(\sqrt{-10}\big)$  & \vskip2pt $L_2=\mathfrak{U}_2v_2$ with $h(v_2)=1/2$ and $\mathfrak{U}_2=(2,w)$     \\[0.5ex]
\hline\vskip2pt Case 4  & \vskip2pt $\mathbb{Q}\big(\sqrt{-13}\big)$  & \vskip2pt $L_2=\mathfrak{U}_2v_2$ with $h(v_2)=1/2$ and $\mathfrak{U}_2=(2,1+w)$   \\[0.5ex]
\hline\vskip2pt Case 5  & \vskip2pt $\mathbb{Q}\big(\sqrt{-15}\big)$  & \vskip2pt $L_2=\mathfrak{U}_2v_2$ with $h(v_2)=1/2$ and $\mathfrak{U}_2=(2,1+w)$   \\[0.5ex]
\hline\vskip2pt Case 6  & \vskip2pt $\mathbb{Q}\big(\sqrt{-22}\big)$  & \vskip2pt $L_2=\mathfrak{U}_2v_2$ with $h(v_2)=1/2$ and $\mathfrak{U}_2=(2,w)$     \\[0.5ex]
\hline\vskip2pt Case 7  & \vskip2pt $\mathbb{Q}\big(\sqrt{-35}\big)$  & \vskip2pt $L_2=\mathfrak{U}_2v_2$ with $h(v_2)=1/5$ and $\mathfrak{U}_2=(5,2+w)$   \\[0.5ex]
\hline\vskip2pt Case 8  & \vskip2pt $\mathbb{Q}\big(\sqrt{-37}\big)$  & \vskip2pt $L_2=\mathfrak{U}_2v_2$ with $h(v_2)=1/2$ and $\mathfrak{U}_2=(2,1+w)$   \\[0.5ex]
\hline\vskip2pt Case 9  & \vskip2pt $\mathbb{Q}\big(\sqrt{-51}\big)$  & \vskip2pt $L_2=\mathfrak{U}_2v_2$ with $h(v_2)=1/5$ and $\mathfrak{U}_2=(5,1+w)$   \\[0.5ex]
\hline\vskip2pt Case 10 & \vskip2pt $\mathbb{Q}\big(\sqrt{-58}\big)$  & \vskip2pt $L_2=\mathfrak{U}_2v_2$ with $h(v_2)=1/2$ and $\mathfrak{U}_2=(2,w)$     \\[0.5ex]
\hline\vskip2pt Case 11 & \vskip2pt $\mathbb{Q}\big(\sqrt{-91}\big)$  & \vskip2pt $L_2=\mathfrak{U}_2v_2$ with $h(v_2)=1/7$ and $\mathfrak{U}_2=(7,3+w)$   \\[0.5ex]
\hline\vskip2pt Case 12 & \vskip2pt $\mathbb{Q}\big(\sqrt{-115}\big)$ & \vskip2pt $L_2=\mathfrak{U}_2v_2$ with $h(v_2)=1/5$ and $\mathfrak{U}_2=(5,-3+w)$  \\[0.5ex]
\hline\vskip2pt Case 13 & \vskip2pt $\mathbb{Q}\big(\sqrt{-123}\big)$ & \vskip2pt $L_2=\mathfrak{U}_2v_2$ with $h(v_2)=1/3$ and $\mathfrak{U}_2=(3,1+w)$   \\[0.5ex]
\hline\vskip2pt Case 14 & \vskip2pt $\mathbb{Q}\big(\sqrt{-187}\big)$ & \vskip2pt $L_2=\mathfrak{U}_2v_2$ with $h(v_2)=1/7$ and $\mathfrak{U}_2=(7,-2+w)$  \\[0.5ex]
\hline\vskip2pt Case 15 & \vskip2pt $\mathbb{Q}\big(\sqrt{-235}\big)$ & \vskip2pt $L_2=\mathfrak{U}_2v_2$ with $h(v_2)=1/5$ and $\mathfrak{U}_2=(5,2+w)$   \\[0.5ex]
\hline\vskip2pt Case 16 & \vskip2pt $\mathbb{Q}\big(\sqrt{-267}\big)$ & \vskip2pt $L_2=\mathfrak{U}_2v_2$ with $h(v_2)=1/3$ and $\mathfrak{U}_2=(3,1+w)$   \\[0.5ex]
\hline\vskip2pt Case 17 & \vskip2pt $\mathbb{Q}\big(\sqrt{-403}\big)$ & \vskip2pt $L_2=\mathfrak{U}_2v_2$ with $h(v_2)=1/11$ and $\mathfrak{U}_2=(11,6+w)$ \\[0.5ex]
\hline\vskip2pt Case 18 & \vskip2pt $\mathbb{Q}\big(\sqrt{-427}\big)$ & \vskip2pt $L_2=\mathfrak{U}_2v_2$ with $h(v_2)=1/7$ and $\mathfrak{U}_2=(7,3+w)$   \\[0.5ex]
\hline
\end{tabular}
\end{center}}
\begin{center}
{\bf{\color{blue}Table 1.}} Non-free unary unimodular Hermitian lattices of class number $2$.
\end{center}

\newpage
\section{Appendix 2}\label{Sec:App2}
Below, we record a representative from each different isometry class of non-free unary unimodular Hermitian lattices over $E=\mathbb{Q}\big(\sqrt{-d}\big)$, with class number $3$.
One can find details in the work of Arno, Robinson and Wheeler \cite{ARW}.
{\small\begin{center}
\setlength{\arrayrulewidth}{0.2mm}
\setlength{\tabcolsep}{2pt}
\renewcommand{\arraystretch}{0.88}
\begin{tabular}{|m{1.12cm}|m{1.52cm}|m{7.14cm}|}
\hline\vskip2pt Cases   &\vskip2pt $\mathbb{Q}\big(\sqrt{-d}\big)$   & \vskip2pt Isometry class representative                                             \\[0.5ex]

\hline\vskip2pt Case 1  &\vskip2pt $\mathbb{Q}\big(\sqrt{-23}\big)$  & \vskip2pt $L_2=\mathfrak{U}_2v_2$ with $h(v_2)=1/2$ and $\mathfrak{U}_2=(2,w)$      \\[0.5ex]
                        &                                            &           $L_3=\mathfrak{U}_3v_3$ with $h(v_3)=1/2$ and $\mathfrak{U}_3=(2,-1+w)$   \\[0.5ex]

\hline\vskip2pt Case 2  &\vskip2pt $\mathbb{Q}\big(\sqrt{-31}\big)$  & \vskip2pt $L_2=\mathfrak{U}_2v_2$ with $h(v_2)=1/2$ and $\mathfrak{U}_2=(2,w)$      \\[0.5ex]
                        &                                            &           $L_3=\mathfrak{U}_3v_3$ with $h(v_3)=1/2$ and $\mathfrak{U}_3=(2,-1+w)$   \\[0.5ex]

\hline\vskip2pt Case 3  &\vskip2pt $\mathbb{Q}\big(\sqrt{-59}\big)$  & \vskip2pt $L_2=\mathfrak{U}_2v_2$ with $h(v_2)=1/3$ and $\mathfrak{U}_2=(3,w)$      \\[0.5ex]
                        &                                            &           $L_3=\mathfrak{U}_3v_3$ with $h(v_3)=1/3$ and $\mathfrak{U}_3=(3,-1+w)$   \\[0.5ex]

\hline\vskip2pt Case 4  &\vskip2pt $\mathbb{Q}\big(\sqrt{-83}\big)$  & \vskip2pt $L_2=\mathfrak{U}_2v_2$ with $h(v_2)=1/3$ and $\mathfrak{U}_2=(3,w)$      \\[0.5ex]
                        &                                            &           $L_3=\mathfrak{U}_3v_3$ with $h(v_3)=1/3$ and $\mathfrak{U}_3=(3,-1+w)$   \\[0.5ex]

\hline\vskip2pt Case 5  &\vskip2pt $\mathbb{Q}\big(\sqrt{-107}\big)$ & \vskip2pt $L_2=\mathfrak{U}_2v_2$ with $h(v_2)=1/3$ and $\mathfrak{U}_2=(3,w)$      \\[0.5ex]
                        &                                            &           $L_3=\mathfrak{U}_3v_3$ with $h(v_3)=1/3$ and $\mathfrak{U}_3=(3,-1+w)$   \\[0.5ex]

\hline\vskip2pt Case 6  &\vskip2pt $\mathbb{Q}\big(\sqrt{-139}\big)$ & \vskip2pt $L_2=\mathfrak{U}_2v_2$ with $h(v_2)=1/5$ and $\mathfrak{U}_2=(5,w)$      \\[0.5ex]
                        &                                            &           $L_3=\mathfrak{U}_3v_3$ with $h(v_3)=1/5$ and $\mathfrak{U}_3=(5,-1+w)$   \\[0.5ex]

\hline\vskip2pt Case 7  &\vskip2pt $\mathbb{Q}\big(\sqrt{-211}\big)$ & \vskip2pt $L_2=\mathfrak{U}_2v_2$ with $h(v_2)=1/5$ and $\mathfrak{U}_2=(5,1+w)$    \\[0.5ex]
                        &                                            &           $L_3=\mathfrak{U}_3v_3$ with $h(v_3)=1/5$ and $\mathfrak{U}_3=(5,-2+w)$   \\[0.5ex]

\hline\vskip2pt Case 8  &\vskip2pt $\mathbb{Q}\big(\sqrt{-283}\big)$ & \vskip2pt $L_2=\mathfrak{U}_2v_2$ with $h(v_2)=1/7$ and $\mathfrak{U}_2=(7,2+w)$    \\[0.5ex]
                        &                                            &           $L_3=\mathfrak{U}_3v_3$ with $h(v_3)=1/7$ and $\mathfrak{U}_3=(7,-3+w)$   \\[0.5ex]

\hline\vskip2pt Case 9  &\vskip2pt $\mathbb{Q}\big(\sqrt{-307}\big)$ & \vskip2pt $L_2=\mathfrak{U}_2v_2$ with $h(v_2)=1/7$ and $\mathfrak{U}_2=(7,w)$      \\[0.5ex]
                        &                                            &           $L_3=\mathfrak{U}_3v_3$ with $h(v_3)=1/7$ and $\mathfrak{U}_3=(7,-1+w)$   \\[0.5ex]

\hline\vskip2pt Case 10 &\vskip2pt $\mathbb{Q}\big(\sqrt{-331}\big)$ & \vskip2pt $L_2=\mathfrak{U}_2v_2$ with $h(v_2)=1/5$ and $\mathfrak{U}_2=(5,1+w)$    \\[0.5ex]
                        &                                            &           $L_3=\mathfrak{U}_3v_3$ with $h(v_3)=1/5$ and $\mathfrak{U}_3=(5,-2+w)$   \\[0.5ex]

\hline\vskip2pt Case 11 &\vskip2pt $\mathbb{Q}\big(\sqrt{-379}\big)$ & \vskip2pt $L_2=\mathfrak{U}_2v_2$ with $h(v_2)=1/5$ and $\mathfrak{U}_2=(5,w)$      \\[0.5ex]
                        &                                            &           $L_3=\mathfrak{U}_3v_3$ with $h(v_3)=1/5$ and $\mathfrak{U}_3=(5,-1+w)$   \\[0.5ex]

\hline\vskip2pt Case 12 &\vskip2pt $\mathbb{Q}\big(\sqrt{-499}\big)$ & \vskip2pt $L_2=\mathfrak{U}_2v_2$ with $h(v_2)=1/5$ and $\mathfrak{U}_2=(5,w)$      \\[0.5ex]
                        &                                            &           $L_3=\mathfrak{U}_3v_3$ with $h(v_3)=1/5$ and $\mathfrak{U}_3=(5,-1+w)$   \\[0.5ex]

\hline\vskip2pt Case 13 &\vskip2pt $\mathbb{Q}\big(\sqrt{-547}\big)$ & \vskip2pt $L_2=\mathfrak{U}_2v_2$ with $h(v_2)=1/11$ and $\mathfrak{U}_2=(11,2+w)$  \\[0.5ex]
                        &                                            &           $L_3=\mathfrak{U}_3v_3$ with $h(v_3)=1/11$ and $\mathfrak{U}_3=(11,-3+w)$ \\[0.5ex]

\hline\vskip2pt Case 14 &\vskip2pt $\mathbb{Q}\big(\sqrt{-643}\big)$ & \vskip2pt $L_2=\mathfrak{U}_2v_2$ with $h(v_2)=1/7$ and $\mathfrak{U}_2=(7,w)$      \\[0.5ex]
                        &                                            &           $L_3=\mathfrak{U}_3v_3$ with $h(v_3)=1/7$ and $\mathfrak{U}_3=(7,-1+w)$   \\[0.5ex]

\hline\vskip2pt Case 15 &\vskip2pt $\mathbb{Q}\big(\sqrt{-883}\big)$ & \vskip2pt $L_2=\mathfrak{U}_2v_2$ with $h(v_2)=1/13$ and $\mathfrak{U}_2=(13,w)$    \\[0.5ex]
                        &                                            &           $L_3=\mathfrak{U}_3v_3$ with $h(v_3)=1/13$ and $\mathfrak{U}_3=(13,-1+w)$ \\[0.5ex]

\hline\vskip2pt Case 16 &\vskip2pt $\mathbb{Q}\big(\sqrt{-907}\big)$ & \vskip2pt $L_2=\mathfrak{U}_2v_2$ with $h(v_2)=1/13$ and $\mathfrak{U}_2=(13,4+w)$  \\[0.5ex]
                        &                                            &           $L_3=\mathfrak{U}_3v_3$ with $h(v_3)=1/13$ and $\mathfrak{U}_3=(13,-5+w)$ \\[0.5ex]
\hline
\end{tabular}
\end{center}}
\begin{center}
{\bf{\color{blue}Table 2.}} Non-free unary unimodular Hermitian lattices of class number $3$.
\end{center}

\begin{acknowledgments}
I truly appreciate the anonymous referees and Professor Kirsten Eisentr\"{a}ger for their invaluable suggestions, and acknowledge the support of $2020-2021$ Faculty Research Fellowship from College of Arts \& Sciences at Texas A\&M University-San Antonio.
\end{acknowledgments}


\end{document}